\documentclass[11pt, reqno]{amsart}

\usepackage[T1]{fontenc}
\usepackage{amsmath}
\usepackage{amssymb}
\usepackage{amsthm}
\usepackage[left=3.5cm, right=3.5cm, paperheight=11.8in]{geometry}
\usepackage{hyperref}
\usepackage{fancyhdr}
\usepackage[textsize=small]{todonotes}
\usepackage{enumitem}
\usepackage{comment}
\usepackage{nicefrac}
\usepackage{bm}
\usepackage{mathrsfs}
\usepackage{graphicx}
\usepackage[utf8]{inputenc}
\usepackage{cancel}
\usepackage{mathtools}
\usepackage{tikz}
\usetikzlibrary{cd,decorations.pathreplacing,matrix,arrows,positioning,automata,shapes,shadows,calc,fadings,decorations,snakes,through,intersections}
\usepackage{float}

\theoremstyle{plain}
\newtheorem{theorem}{Theorem}[section]
\newtheorem{lemma}[theorem]{Lemma}
\newtheorem{corollary}[theorem]{Corollary}
\newtheorem{proposition}[theorem]{Proposition}

\theoremstyle{definition}
\newtheorem{definition}[theorem]{Definition}

\newtheorem{example}[theorem]{Example}
\newtheorem{remark}[theorem]{Remark}

\newtheorem{question}[theorem]{Question}

\newtheorem{claim}{Claim}

\newcommand{\cI}{\mathcal{I}}
\newcommand{\I}{\cI}
\newcommand{\cJ}{\mathcal{J}}
\newcommand{\J}{\cJ}

\hypersetup{
    pdftitle={Complexity and Polishability of characterized subgroups},
    pdfauthor={Paolo Leonetti},
    pdfmenubar=false,
    pdffitwindow=true,
    pdfstartview=FitH,
    colorlinks=true,
    linkcolor=blue,
    citecolor=green,
    urlcolor=cyan
}

\providecommand{\MR}[1]{}

\providecommand{\MR}{\relax\ifhmode\unskip\space\fi MR }

\providecommand{\href}[2]{#2}

\subjclass[2020]{
Primary: 54H11, 
03E15;          
Secondary: 22A05, 
40A35,          
54H05,          
43A40,          
11K06           
}

\keywords{characterized subgroups; 
ideal convergence; 
analytic $P$-ideals;
Polishable subgroups;
generalized density ideals;
Borel complexity;
Wadge reducibility}

\begin{document}


\title{Complexity and Polishability of characterized subgroups on the unit circle}


\author[P.~Leonetti]{Paolo Leonetti}
\address{
Universit\`a degli Studi dell'Insubria, via Monte Generoso 71, 21100 Varese, Italy 
}
\email{leonetti.paolo@gmail.com}

\begin{abstract}
\noindent 
Given an ideal $\mathcal I$ on $\omega$, a subgroup $H$ of the unit circle $\mathbb T$ is said to be $\mathcal I$-characterized if there exists a sequence of integers $\bm a=(a_n:n\in\omega)$ such that
$$
H=
\left\{
x\in\mathbb T:
\mathcal I\text{-}\lim_{n\to\infty}a_nx=0
\right\}.
$$
We investigate the descriptive complexity and Polishability of these subgroups in terms of the structural and topological properties of the ideal $\mathcal I$.

Our main structural result shows that $\mathcal{I}$ is an analytic $P$-ideal if and only if all $\mathcal I$-characterized subgroups are Polishable. In such case, we explicitly describe a compatible finer Polish group topology. Using results on Polishable subgroups, we obtain a trichotomy for their possible Borel complexities.

If $\mathcal{I}$ is a generalized density ideal, we show the sharper dichotomy that every proper $\mathcal I$-characterized subgroup is either countable or $F_{\sigma\delta}$-complete. We also prove that this fails for general analytic $P$-ideals by constructing a subgroup characterized by a summable ideal which is neither $F_\sigma$ nor $F_{\sigma\delta}$-complete. Finally, we give explicit descriptions of the subgroups associated with the sequences of powers, the Fibonacci sequence, and the sequence of factorials. We conclude with several open questions. 
\end{abstract}


\maketitle



\thispagestyle{empty}



\section{Introduction}\label{sec:intro}

In this work, we are going to study certain subgroups of the unit circle $\mathbb{T}:=\mathbb{R}/\mathbb{Z}$. To be more precise, let $\mathcal{I}$ be an ideal on the nonnegative integers $\omega$, that is, a family of subsets of $\omega$ which is stable under finite unions and subsets. Unless otherwise stated, it is always assumed that the family of finite sets $\mathrm{Fin}:=[\omega]^{<\omega}$ is contained in $\mathcal{I}$, and that $\mathcal{I}$ is proper, that is, $\omega\notin \mathcal{I}$. In place of $\omega$, we will consider also ideals on countably infinite sets. 
An ideal $\mathcal{I}$ is said to be a $P$-ideal if for every sequence $(S_n: n \in \omega)$ in $\mathcal{I}$ there exists $S \in \mathcal{I}$ such that $S_n\setminus S$ is finite for all $n \in \omega$. 
Identifying $\mathcal{P}(\omega)$ with the Cantor space, we can speak about the topological complexity of ideals. 
Among important examples of ideals are the family $\mathcal{Z}:=\{S\subseteq \omega: \lim_n |S\cap n|/n=0\}$ of asymptotic density zero sets, the summable ideal $\I_{1/n}:=\{S\subseteq \omega: \sum_{n \in S}1/(n+1)<\infty\}$, and maximal ideals (that is, the complements of free ultrafilters on $\omega$). 
For instance, $\mathcal{Z}$ is an $F_{\sigma\delta}$ $P$-ideal which is not $F_\sigma$, while both $\mathrm{Fin}$ and $\I_{1/n}$ are $F_\sigma$ $P$-ideals. In addition, it is well known that 
\begin{equation}\label{eq:inclusionideals}
\mathrm{Fin}\subseteq \I_{1/n} \subseteq \mathcal{Z}.
\end{equation}
Let us also recall that a sequence $\bm{x}=(x_n: n \in \omega) \in \mathbb{T}^\omega$ is $\mathcal{I}$-convergent to $\eta \in \mathbb{T}$, shortened as $\mathcal{I}\text{-}\lim \bm{x}=\eta$ or $\mathcal{I}\text{-}\lim_n x_n=\eta$, if $\{n \in \omega: \|x_n-\eta\|\ge \varepsilon\} \in \mathcal{I}$ for all $\varepsilon>0$. Of course, $\mathrm{Fin}$-convergence coincides with ordinary convergence. Finally, we identify each real number $x$ with its equivalence class $x+\mathbb{Z} \in \mathbb{T}$, and write $\|x\|:=\min\{|x-z|:z \in \mathbb{Z}\}$. 

With the above premises, we recall the notion of $\I$-characterized subgroup. 
\begin{definition}\label{def:mainHaI}
    Let $\mathcal{I}$ be an ideal on $\omega$. A subgroup $H$ of $\mathbb{T}$ is said to be $\mathcal{I}$\emph{-characterized} if there exists a sequence of integers $\bm{a}=(a_n: n \in \omega)\in \mathbb{Z}^\omega$ such that 
    \begin{equation}\label{eq:defHaI}
    H=\mathsf{H}_{\bm{a}}(\mathcal{I}):=\left\{x \in \mathbb{T}: \mathcal{I}\text{-}\lim_{n\to \infty} a_nx=0\right\}.
    \end{equation}
    In the case $\mathcal{I}=\mathrm{Fin}$, the subgroup $H$ is said to be \emph{characterized}. We denote the family of $\mathcal{I}$-characterized subgroups by 
    $$
    \mathscr{H}(\mathcal{I}):=\{\mathsf{H}_{\bm{a}}(\mathcal{I}): \bm{a} \in \mathbb{Z}^\omega\}.
    $$ 
\end{definition}

\subsection{Literature} It is readily seen that the characterized subgroups $\mathsf{H}_{\bm a}(\mathrm{Fin})$ can be regarded as family of points admitting a sufficiently good rational approximation along the prescribed sequence of denominators indexed by $\bm{a}$. In fact, if $\bm{a}$ is a strictly increasing sequence of positive integers then $x \in \mathsf{H}_{\bm a}(\mathrm{Fin})$ if and only if there are rationals $(b_n/a_n: n \in \omega)$ such that $|x-b_n/a_n|=o(1/a_n)$ as $n\to \infty$. 
This connects characterized subgroups with the classical study of
convergence phenomena arising from continued fractions and Diophantine
approximation \cite{MR1062392, MR947645}, cf. also \cite{Matomaki2009}. 
It is worth noting that this point of view is complementary to the classical metric theory of uniform distribution: indeed, by a classical theorem of Weyl, if $\bm{a}$ is a strictly increasing sequence of positive integers, then $(a_nx: n \in \omega)$ is uniformly distributed for Lebesgue-almost every \(x\in\mathbb T\), see e.g. \cite[Chapter 4]{MR419394} (in particular, in such case, $\mathsf{H}_{\bm a}(\mathrm{Fin})$ has Lebesgue measure zero). 

The study of characterized subgroups is, however, rooted in an older and
broader problem: understanding how the algebraic structure of a topological
group is reflected by convergence phenomena. Already in the classical theory
of topological groups, topologically torsion elements were introduced
independently by Braconnier \cite{MR13158} and Vilenkin \cite{MR14104}, and
they played an important role in the structure theory of locally compact
abelian groups and profinite groups. From this perspective, a characterized
subgroup of \(\mathbb T\) is not an ad hoc object: it is the subgroup of
elements which are ``topologically torsion'' with respect to a prescribed
sequence of characters, cf. \cite{MR346087} and \cite[Section 14.4]{MR4510389}. This makes the class of
characterized subgroups rigid enough to have strong structural and
regularity properties, but also flexible enough to contain many nontrivial
examples. For instance, Erd\"os and Taylor \cite{MR92032} proved that if a
positive sequence \(\bm a\) satisfies
\(\sup_n a_{n+1}/a_n<\infty\), then
\(\mathsf{H}_{\bm a}(\mathrm{Fin})\) is countable. Conversely, B\'ir{\'o},
Deshouillers and S{\'o}s \cite{MR1877772} proved that every countable
subgroup of \(\mathbb T\) is characterized.


Lastly, we recall that alternative constructions of subgroups which involve ideals on $\omega$ in the same spirit of \eqref{eq:defHaI} are known to be fruitful, see e.g. Farah and Solecki \cite[p. 516]{MR2189217}, Solecki \cite[Section 3]{MR2197115}, and Hu and Solecki \cite[Section 2]{SoleckiHu}. 
We refer to \cite{
MR2032835, MR4932230, 
MR3145818, MR3288121}
and references therein 
for detailed surveys on the history and additional motivations for the study of characterized subgroups and their ideal version. 

\subsection{Objectives and Plan} 
It is remarkable that \emph{every} subgroup of $\mathbb{T}$ is $\I$-characterized for some ideal $\I$ on $\omega$, see \cite[Theorem 2.1]{MR2227021} and \cite[Theorem 2.11]{FKLT26}. 
Accordingly, our point is not merely to decide whether a subgroup can be represented in the form \(\mathsf H_{\bm a}(\mathcal I)\), but rather to measure how complicated the ideal \(\mathcal I\) must be in order to obtain a given subgroup. Thus the assignment 
$
\mathcal I\mapsto \mathscr H(\mathcal I)
$ 
may be viewed as a way of transferring structural and descriptive-set-theoretic information from ideals on \(\omega\) to subgroups of the circle. In this direction, the present work studies to what extent the Borel or projective complexity of \(\mathcal I\) is reflected by the complexity of the subgroups \(\mathsf H_{\bm a}(\mathcal I)\). 

Our work is also motivated by the study of Polishable subgroups of
the compact Polish group $\mathbb{T}$, that is, Borel subgroups which admit a finer Polish group topology, see \cite{
MR2189217, MR2267154, 
MR1396895, MR5045103}. 
It is known that characterized subgroups $\mathsf{H}_{\bm{a}}(\mathrm{Fin})$ are Polishable, see e.g. \cite[Corollary~1]{MR2729343} or \cite[Theorem 1.21]{MR3461178}. In our setting, our question is to characterize the family of ideals $\I$ for which every $\I$-characterized subgroup is Polishable. 
We show that 
this happens precisely for analytic $P$-ideals. 
In particular, we fix a gap in the proofs of \cite[Proposition 2.4]{MR4078214} and \cite[Remark 1.12]{MR4780094}. 
Finally, we provide explicit characterizations of subgroups $\mathsf{H}_{\bm{a}}(\I)$ in the special cases where $\bm{a}$ is the sequence of powers, the Fibonacci sequence, and the sequence of factorials. 
These descriptions suggest a natural connection with automata theory and symbolic restrictions on expansions \cite{AlloucheShallit2003}. 

More precisely, some instances of our main results follow below: 
\begin{enumerate}[label={\rm (\arabic*)}]

\item [(a)] If $\mathcal{I}$ is a $\mathbf{\Pi}^0_\xi$ ideal for some countable ordinal $3\le \xi<\omega_1$, then every $\mathcal{I}$-characterized subgroup is $\mathbf{\Pi}^0_\xi$; conversely, for every ideal $\mathcal{I}$ there exists a sequence $\bm{a}$ such that 
the complexity of $\I$ is not smaller than the one of $\mathsf{H}_{\bm{a}}(\mathcal{I})$ 
(see Theorems \ref{thm:upperbound} and \ref{thm:wadge}, respectively). In particular, if $\mathcal{I}$ is Borel, then the maximal Borel rank attained by an $\mathcal{I}$-characterized subgroup coincides with the Borel rank of $\mathcal{I}$ (see Corollary \ref{cor:borelrank}). 
In addition, for every $F_{\sigma\delta}$ ideal $\mathcal{I}$ there exists an $F_{\sigma\delta}$-complete $\mathcal{I}$-characterized subgroup of $\mathbb{T}$ (see Corollary \ref{cor:Fsigmadeltacomplete});

\item [(b)] Every $\mathcal{I}$-characterized subgroup of $\mathbb{T}$ is Polishable if and only if $\mathcal{I}$ is an analytic $P$-ideal (see Theorem \ref{thm:analyticPidealscharacterization});


\item [(c)] If $\mathcal{I}$ is a generalized density ideal, every proper $\mathcal{I}$-characterized subgroup of $\mathbb{T}$ is either countable or $F_{\sigma\delta}$-complete (see Theorem \ref{thm:Fsigmadelcomplete}). In particular, every proper uncountable characterized subgroup of $\mathbb{T}$ is $F_{\sigma\delta}$-complete. 
This dichotomy cannot be extended to all analytic $P$-ideals: there exists an $\mathcal{I}_{1/n}$-characterized subgroup which is neither $F_\sigma$ nor $F_{\sigma\delta}$-complete (see Theorem \ref{thm:strangecounterexample});


\item [(d)] For the sequences of powers $(b^n:n\in\omega)$, the Fibonacci numbers $(f_n:n\in\omega)$, and the factorials $(n!:n\ge 1)$, the corresponding $\mathcal{I}$-characterized subgroups admit explicit descriptions (see Theorems \ref{thm:armacost}, \ref{thm:fibonacci}, and \ref{thm:factorialcharacterization}, respectively);


\end{enumerate}

The proofs of the results from Sections \ref{sec:mainresults} and \ref{sec:exampleapplications} are given in Sections \ref{sec:proofsmain} and \ref{sec:proofexamples}, respectively. 
Lastly, in Section \ref{sec:openquestions} we conclude with some open questions.

%
%
%
%

\section{Main results}\label{sec:mainresults}

\subsection{Strictly increasing positive representations} 
Since $0 \in \mathbb{T}$ has a local base of symmetric open neighborhoods, it is immediate that 
$$
\mathsf{H}_{\bm{a}}(\I)=\mathsf{H}_{(|a_n|: n \in \omega)}(\I),
$$
hence it can be assumed without loss of generality that $\bm{a} \in \omega^\omega$. In addition, in the case $\I=\mathrm{Fin}$, it is well known that $\mathsf{H}_{\bm{a}}(\mathrm{Fin})=\mathbb{T}$ if and only if $\bm{a}$ is finitely supported, see \cite{Schoen}; cf. \cite[Theorem D]{MR4932230} and also \cite[Theorem 7.8]{MR419394} for a textbook exposition.

However, a recurrent convention in the literature on characterized subgroups of the circle is to work not with arbitrary integer sequences $\bm{a}$, but with strictly increasing sequences of positive integers. As anticipated in Section \ref{sec:intro}, this is natural from the Diophantine approximation point of view. 
This convention appears already in the work of B{\'i}r{\'o}, Deshouillers and S{\'o}s \cite{MR1877772}, where characterizing sequences are sequences
of positive integers, and also in Beiglb{\"o}ck's strengthening for countable subgroups \cite{MR2362430}. 
It is also implicit in several concrete classes of examples, such as continued fraction denominators \cite{MR947645} and lacunary sequences \cite{MR92032}, cf. also \cite{MR4594809, MR48504}. 

For this reason, it is useful to know that the positivity and strict monotonicity of the characterizing sequence are not genuine restrictions. Our first main result (Theorem \ref{thm:increasingSequenceSameFin} below) shows something more: every proper characterized subgroup of \(\mathbb T\) admits a strictly increasing positive characterizing sequence, and this sequence can be chosen to dominate any prescribed sequence. 


\begin{theorem}\label{thm:increasingSequenceSameFin}
Let $H$ be a proper characterized subgroup of $\mathbb{T}$ and pick a sequence of integers $(r_n: n\in \omega)$. Then there exists a strictly increasing sequence of positive integers $\bm{a}=(a_n: n \in \omega)$ such that 
$$
H=\mathsf{H}_{\bm{a}}(\mathrm{Fin})
\quad \text{ and }\quad 
a_n\ge r_n \text{ for all }n \in \omega.
$$
\end{theorem}

The statement in the case where $r_n=0$ for all $n\in \omega$ can be found also in \cite[Section 3.2]{MR3864800}. From the opposite direction, it is worth noting that the sequence $\bm{a}$ in Theorem \ref{thm:increasingSequenceSameFin} cannot satisfy, in general, certain upper bounds such as $\sup_n a_{n+1}/a_n<\infty$. In fact, 
it is well known that there are proper uncountable characterized subgroups of $\mathbb{T}$ (e.g. $\mathsf{H}_{(2^{2^n})}(\mathrm{Fin})$): more precisely, it is known that if $\bm{a}$ is a strictly increasing sequence of positive integers and $a_n$ divides $a_{n+1}$ for all $n \in \omega$ then $\mathsf{H}_{\bm{a}}(\mathrm{Fin})$ is countable if and only if $\sup_n a_{n+1}/a_n<\infty$, see \cite{MR4594809, MR48504}. For instance $\mathsf{H}_{(n!)}(\mathrm{Fin})$ is a proper uncountable subgroup, cf. Remark \ref{rmk:examplecomplete} below. 

We do not know whether the analogue of Theorem \ref{thm:increasingSequenceSameFin} holds for $\mathcal{Z}$, see Question \ref{q:Zstrictlyincreasing}. 

\subsection{Bounds on topological complexities} 
%
%
We now turn to the descriptive complexity of \(\mathcal I\)-characterized
subgroups. Our first result in this direction gives a general upper
bound: if \(\mathcal I\) is sufficiently low in the Borel hierarchy, then the same is true for every subgroup \(\mathsf H_{\bm a}(\mathcal I)\). In turn, this improves on \cite[Theorem 2.5(i)--(iii)]{FKLT26}. 

\begin{theorem}\label{thm:upperbound}
    Let $\I$ be a $\mathbf{\Pi}^0_\xi$ ideal on $\omega$ for some countable ordinal $3\le \xi<\omega_1$. Then each $\mathsf{H}_{\bm{a}}(\I)$ is a $\mathbf{\Pi}^0_\xi$ subgroup.
\end{theorem}

We remark that special instances of Theorem \ref{thm:upperbound} (for certain analytic $P$-ideals, hence in the case $\xi=3$) appeared in \cite{protasov, MR1240629, MR4078214}. 

Then, we show that the above upper bound is, in a precise sense, optimal. To this aim, given topological spaces $X,Y$ and subsets $A\subseteq X$ and $B\subseteq Y$, we say that $A$ \emph{is Wadge reducible to $B$}, shortened as 
$$
A\le_{\mathrm{W}} B,
$$
if there exists a continuous function $f: X\to Y$ such that $f^{-1}[B]=A$, see \cite[p. 156]{MR1321597}. Informally, 
in the following result we show that the topological complexity of some $\I$-characterized subgroup cannot be \textquotedblleft simpler\textquotedblright\,than the one of the ideal $\I$. 


\begin{theorem}\label{thm:wadge}
    There exists $\bm{a} \in \omega^\omega$ such that $\I \le_{\mathrm{W}} \mathsf{H}_{\bm{a}}(\I)$ for every ideal $\I$ on $\omega$.
\end{theorem}

As a consequence, we prove that if $\I$ is not in a given pointclass $\mathbf{\Gamma}$, then the same holds for some $\I$-characterized subgroup. More precisely, a pointclass \(\mathbf{\Gamma}\) is said to be closed under continuous preimages if, for
every pair of Polish spaces \(X,Y\), every continuous map \(f:X\to Y\), and every
\(A\in\mathbf{\Gamma}(Y)\), we have \(f^{-1}[A]\in\mathbf{\Gamma}(X)\). This is the case, for
instance, for all Borel classes \(\mathbf{\Sigma}^0_\xi,\mathbf{\Pi}^0_\xi\)
with \(1\leq\xi<\omega_1\), for the class of Borel sets, for analytic sets, and
for coanalytic sets (and, more generally, for $\mathbf{\Sigma}^1_m$, $\mathbf{\Pi}^1_m$, and $\mathbf{\Delta}^1_m$ with $m\ge 1$), see \cite{MR1321597}. An analogue definition applies for pointclasses closed under preimages by Baire class one maps. 
\begin{corollary}\label{cor:pointclassessimple}
Let \(\mathbf{\Gamma}\) be a pointclass which is closed under continuous preimages. 
Let $\I$ be an ideal on $\omega$ 
with $\I\notin \mathbf{\Gamma}(2^\omega)$, then there exists $\bm{a} \in \omega^\omega$ such that $\mathsf{H}_{\bm{a}}(\I) \notin \mathbf{\Gamma}(\mathbb{T})$. 

For instance, if $\I$ is not analytic \textup{[}or Borel, or $F_{\sigma\delta}$, respectively\textup{]} then there exists $\bm{a} \in \omega^\omega$ such that $\mathsf{H}_{\bm{a}}(\I)$ is not analytic \textup{[}or Borel, or $F_{\sigma\delta}$, resp.\textup{]}. 
\end{corollary}

In particular, if $\I$ is a maximal $P$-ideal (whose existence is independent of $\mathsf{ZFC}$) then it is well known that $\I$ is not analytic (see e.g. \cite[Proposition 2.5]{FKL24} and \cite[Theorem 21.6]{MR1321597}), hence by Corollary \ref{cor:pointclassessimple} there exists $\bm{a} \in \omega^\omega$ such that $\mathsf{H}_{\bm{a}}(\I)$ is not analytic. This provides an answer to \cite[Question 6.6(c)]{protasov}.

To state the next consequence, we recall \cite[Definition 22.9]{MR1321597}: given a pointclass $\mathbf{\Gamma}$ in Polish spaces and given Polish spaces $X,Y$ with $X$ zero-dimensional, a subset $A \subseteq Y$ is said to be $\mathbf{\Gamma}$\emph{-complete} if $A \in \mathbf{\Gamma}(Y)$ and $B\le_{\mathrm{W}}A$ for all $B \in \mathbf{\Gamma}(X)$. 

Thus, we obtain sharp rank estimates in the case for Borel ideals:  
more precisely, given a Borel ideal $\I$ on $\omega$, we denote its Borel rank by 
$$
\mathrm{rank}(\I):=\min\{\xi<\omega_1: \I \text{ is a }\mathbf{\Pi}^0_{\xi}\text{-subset of }2^\omega\}. 
$$
Analogously, for each Borel subgroup $H$ of $\mathbb{T}$, we write $\mathrm{rank}(H):=\min\{\xi<\omega_1: H \text{ is a }\mathbf{\Pi}^0_{\xi}\text{-subset of }\mathbb{T}\}$; see \cite[Section 1.1]{MR2189217} and references therein. Before we state the next application, recall that if an ideal $\I$ is Borel, then every $\I$-characterized subgroup of $\mathbb{T}$ is Borel, see \cite[Theorem 2.5(iv)]{FKLT26} or Corollary \ref {cor:Boreanalyticetc} below; in particular, we can speak about the Borel rank of $\mathsf{H}_{\bm{a}}(\I)$ whenever $\I$ is Borel. 

\begin{corollary}\label{cor:borelrank}
Let $\I$ be a Borel ideal on $\omega$. Then the following hold\textup{:}
 \begin{enumerate}[label={\rm (\roman*)}]
     \item \label{item:1borrank} If $\I$ is $\mathbf{\Pi}^0_m$-complete ideal for some $m\ge 3$, then $\mathsf{H}_{\bm{a}}(\I)$ is $\mathbf{\Pi}^0_m$-complete subgroup for some $\bm{a} \in \omega^\omega$\textup{;}
     \item \label{item:2borrank} $\mathrm{rank}\left(\mathsf{H}_{\bm{a}}(\I)\right)\le \mathrm{rank}(\I)$ for all $\bm{a} \in \omega^\omega$\textup{;}
     \item \label{item:3borrank} $\mathrm{rank}\left(\mathsf{H}_{\bm{a}}(\I)\right)=\mathrm{rank}(\I)$ for some $\bm{a} \in \omega^\omega$\textup{.}
 \end{enumerate}
\end{corollary}

In particular, 
it follows by Corollary \ref{cor:borelrank} and \cite[Remark 10.13]{FKL24} that for each integer $m\ge 3$ there exists a $\mathbf{\Pi}^0_{m}$-ideal $\I$ and a sequence $\bm{a} \in \omega^\omega$ such that $\mathsf{H}_{\bm{a}}(\I)$ is a $\mathbf{\Pi}^0_{m}$-subgroup of $\mathbb{T}$ which is not $\mathbf{\Pi}^0_{m-1}$. This provides an answer to \cite[Question 6.6(b)]{protasov} (in the case of not necessarily $P$-ideals). 

In the following result, we recall that the pointclasses $\mathbf{\Sigma}^1_m$, $\mathbf{\Pi}^1_m$, and $\mathbf{\Delta}^1_m$ with $m\ge 1$ are closed under countable intersections and preimages of Baire class one maps (in particular $\mathbf{\Delta}^1_1$ coincides with the family of Borel sets); note also this does not hold, in general, for the pointclasses $\mathbf{\Sigma}^0_\xi$ or $\mathbf{\Pi}^0_\xi$ with $\xi<\omega$. 
\begin{proposition}\label{prop:completenesspropH}
    Let \(\mathbf{\Gamma}, \mathbf{\Gamma}^\prime\) be pointclasses such that $\mathbf{\Gamma}$ is closed under countable intersections and preimages of Baire class one maps, and $\mathbf{\Gamma}^\prime$ is closed under continuous preimages. 
    Let $\I$ be an ideal on $\omega$ such that 
$\I\in \mathbf{\Gamma}(2^\omega)\setminus \mathbf{\Gamma}^\prime(2^\omega)$. Then there exists $\bm{a} \in \omega^\omega$ such that $\mathsf{H}_{\bm{a}}(\I) \in \mathbf{\Gamma}(\mathbb{T})\setminus \mathbf{\Gamma}^\prime(\mathbb{T})$. 
\end{proposition}

As an application, we recover and improve \cite[Theorem 2.5(iv)]{FKLT26}.
\begin{corollary}\label{cor:Boreanalyticetc}
    Let $\I$ be an ideal on $\omega$ which is $\mathbf{\Sigma}^1_m$ \textup{[}or $\mathbf{\Pi}^1_m$ or $\mathbf{\Delta}^1_m$, respectively\textup{]} for some $m\ge 1$. Then $\mathscr{H}(\I)\subseteq \mathbf{\Sigma}^1_m$ \textup{[}or $\mathbf{\Pi}^1_m$ or $\mathbf{\Delta}^1_m$, resp.\textup{]}.
\end{corollary}

Under $\mathbf{\Sigma}^1_1$-determinacy, see e.g. \cite[Definition 26.3]{MR1321597}, we have also the following:
\begin{corollary}\label{cor:analyticcomplete}
    \textup{(}Assume $\mathbf{\Sigma}^1_1$-determinacy.\textup{)} Let $\I$ be an analytic-complete ideal on $\omega$. 
    Then there exists $\bm{a} \in \omega^\omega$ such that $\mathsf{H}_{\bm{a}}(\I)$ is analytic-complete.
\end{corollary}

Some open questions related to Borel rank and completeness are listed in Section \ref{sec:openquestions}.







\subsection{Polishable characterized subgroups} 
Let us recall that $G$ is a Polish group if it is a Polish space (that is, a separable and completely metrizable space) which is endowed with a continuous group operation. It follows by \cite[p.62]{MR1321597} that, if $G$ is a Polish group, then the function that maps each element to its inverse is also automatically continuous. 
\begin{definition}\label{def:Polishable}
    A subgroup $H$ of a Polish group $G$ is said to be \emph{Polishable} if 
    there exists a Polish group topology $\tau$ on $H$ having the same Borel sets as $H$ when considered as a topological subgroup of $G$. 
\end{definition}
The notion of Polishable subgroup has been introduced by Kechris and Louveau in \cite{MR1396895}. 
It is remarkable that, if such a topology $\tau$ exists, it is unique, see \cite[Theorem 9.10]{MR1321597}. 
It is known that a subgroup $H$ is Polishable if and only if there exists a continuous surjective homomorphism from a Polish group $P$ onto $H$, 
cf. \cite[Definition 1.18]{MR3461178}. 
In addition, all Polishable subgroups are Borel, and they might attain arbitrarily high Borel complexity; a Borel subgroup of $G$ is Polishable if and only if it admits a finer Polish group topology. 
Detailed studies on Polishable subgroups can be found in \cite{MR2189217, MR2267154, MR5045103}. 

Regarding $\mathcal{P}(\omega)$ as a Polish group under the operation of symmetric difference, an ideal $\mathcal{I}$ on $\omega$ is a subgroup of $\mathcal{P}(\omega)$. 
Now, 
define the ideal 
$$
\mathrm{Fin}\otimes\emptyset
:=
\left\{
A\subseteq\omega^2:
A\subseteq k\times \omega
\text{ for some }k\in\omega
\right\},
$$
see e.g. \cite[Chapter 1]{MR1711328}. Recall also that a map $\varphi: \mathcal{P}(\omega)\to [0,\infty]$ is said to be a \emph{lower semicontinuous submeasure}, or shortly \emph{lscsm}, if it is a submeasure (that is, it is monotone, subadditive, and $\varphi(\emptyset)=0$) such that $\varphi(F)<\infty$ for all finite $F \subseteq \omega$, and $\varphi(A)=\sup\{\varphi(A\cap n): n \in \omega\}$ for all $A\subseteq \omega$. 
Notice that the latter property corresponds to the lower semicontinuity of the submeasure $\varphi$, regarding its domain $\mathcal{P}(\omega)$ as the Cantor space $2^\omega$, that is, if $A_n \to A$ then $\liminf_n \varphi(A_n) \ge \varphi(A)$. Examples of lscsms include $\varphi(A)=|A|$ or $\varphi(A)=\sum_{n \in A}1/(n+1)$ or $\varphi(A)=\sup_{n\ge 1} |A\cap n|/n$, cf. also \cite[Chapter 1]{MR1711328}. 
Given ideals $\I,\J$ on $\omega$, we write $\I\le_{\mathrm{RB}}\J$ if  there exists a finite-to-one map $f:\omega\to\omega$ such that
$A\in\I$ if and only if $f^{-1}[A]\in\J$ 
for every $A\subseteq\omega$. 
Accordingly, we recall an important result of Solecki: 
\begin{theorem}\label{thm:solecki}
    Let $\I$ be an ideal on $\omega$. Then the following are equivalent\textup{:}
    \begin{enumerate}[label={\rm (\roman*)}]
    \item \label{item:1solecki} $\I$ is Polishable\textup{;}
    \item \label{item:2solecki} $\I$ is an analytic $P$-ideal\textup{;}
    \item \label{item:3solecki} $\I=\{S\subseteq \omega: \lim_n \varphi(S\setminus n)=0\}$ for some finite lscsm $\varphi$\textup{.} 
    \end{enumerate}
If, in addition, $\I$ is analytic, then they are also equivalent to\textup{:}
    \begin{enumerate}[label={\rm (\roman*)}]
    \setcounter{enumi}{3} 
    \item \label{item:4solecki} $\mathrm{Fin}\otimes \emptyset \not\le_{\mathrm{RB}} \mathcal{I}$\text{.}
    \end{enumerate}
\end{theorem}
\begin{proof}
    See \cite[Theorem 2.1 and Theorem 3.1]{MR1708146}.
\end{proof}

In particular, every analytic $P$-ideal is $F_{\sigma\delta}$. 
We remark that the family of analytic $P$-ideals is large and includes, among others, all Erd{\H o}s--Ulam ideals introduced by Just and Krawczyk in \cite{MR748847}, ideals generated by nonnegative regular matrices \cite{Filipow18, MR4041540},   
certain ideals used by Louveau and Veli\u{c}kovi\'{c} \cite{Louveau1994}, 
and, more generally, density-like ideals and generalized density ideals \cite{MR3436368, MR4404626}. Additional pathological examples can be found in \cite{MR0593624}. 
It has been suggested in \cite{MR4124855, MR3436368} that the theory of analytic $P$-ideals may have some relevant yet unexploited potential for the study of the geometry of Banach spaces. 

In Theorem \ref{thm:analyticPidealscharacterization} below, we are going to show that the 
conditions given in Theorem \ref{thm:solecki} are also equivalent to the Polishability of every $\I$-characterized subgroup of $\mathbb{T}$. 

In this regard, B\'ir\'o proved 
in \cite{MR2388789} 
that every characterized subgroup of $\mathbb{T}$ is Polishable, 
cf. also \cite[Proposition 1.13]{MR2921827}. 
Also, it has been claimed in \cite[Proposition 2.4]{MR4078214} that the analogue statement holds for the ideal $\mathcal{Z}$, namely, every $\mathcal{Z}$-characterized subgroup of $\mathbb{T}$ is Polishable. 
However, as observed in \cite[Remark 2.16]{FKLT26}, the argument given in \cite[Proposition 2.4]{MR4078214} appears to show that $\mathsf{H}_{\bm{a}}(\mathcal{Z})$, as a Borel subset of $\mathbb{T}$, admits a finer Polish topology. It does not seem to verify that this topology makes the group operations continuous, which is required for Polishability as a subgroup. 
%
The same gap appears in \cite[Remark 1.12]{MR4780094} replacing $\mathcal{Z}$ with an arbitrary analytic $P$-ideal. Thus, it was left as open question in \cite[Question 2.17]{FKLT26} whether the claimed property holds for $\mathcal{Z}$ or, more generally, for all analytic $P$-ideals. 

Below, we show that the analogue ideal statement holds precisely for analytic $P$-ideals. 
In particular, this improves on B\'ir\'o's result \cite{MR2388789}, fixes the proof gaps in \cite{MR4078214, MR4780094}, and answers an open question in \cite{FKLT26}. Here, we write also $c_0(\I)$ for the set of sequences with values in $\mathbb{T}$ which are $\I$-convergent to $0$, namely, 
$$
c_0(\I):=\left\{\bm{z}=(z_0,z_1,\ldots) \in \mathbb{T}^\omega: \I\text{-}\lim_{n\to \infty} z_n=0\right\}. 
$$
Note that $c_0(\I)$ is a subgroup of the Polish group $\mathbb{T}^\omega$. 
\begin{theorem}\label{thm:analyticPidealscharacterization} 
Let $\mathcal{I}$ be an ideal on $\omega$. Then the following are equivalent\textup{:}
\begin{enumerate}[label={\rm (\roman*)}]
    \item \label{item:1polishability} $\I$ is an analytic $P$-ideal\textup{;}
    \item \label{item:2polishability} $c_0(\I)$ is Polishable\textup{;}
    \item \label{item:3polishability} Every $\mathcal{I}$-characterized subgroup of $\mathbb{T}$ is Polishable\textup{.} 
\end{enumerate}
\end{theorem}

Actually, as it turns out in the proof of Theorem \ref{thm:analyticPidealscharacterization}, we will show something stronger: if $\I$ is an analytic $P$-ideal then each $\I$-characterized subgroup $\mathsf{H}_{\bm{a}}(\I)$ admits a finer Polish group topology induced by the translation-invariant metric
$$
d_{\varphi,\bm{a}}(x,y):=\|x-y\|+\inf\{\varepsilon>0: \varphi(\{n\in \omega: \|a_n(x-y)\|\ge \varepsilon\}) \le \varepsilon\}
$$ 
for all $x,y \in \mathsf{H}_{\bm{a}}(\I)$, where $\varphi$ is a lscsm as in Theorem \ref{thm:solecki}, cf. Equation \eqref{eq:metricpolishability}. 
This seems to be related to the notion of Fr\'{e}chetability subspace studied in \cite[Section 9]{MR5045103}.


As a consequence, using also a result of B\'ir\'o in \cite{MR2388789}, we obtain the following: 
\begin{corollary}\label{cor:birocounterexample}
There exists an $F_\sigma$ subgroup $H$ of $\mathbb{T}$ which is not Polishable, hence not $\I$-characterized for every analytic $P$-ideal $\I$. 
\end{corollary}

We also have the following analogue.
\begin{proposition}\label{prop:FsigmadeltacompletenotPolishable}
There exists an $F_{\sigma\delta}$-complete subgroup of $\mathbb{T}$ which is not Polishable, hence not $\I$-characterized for every analytic $P$-ideal $\I$. 
\end{proposition}

Along the same lines, following \cite[Section 3]{MR5045103}, 
define the difference class 
$$
\mathsf{D}(\mathbf{\Pi}^0_2):=\{A\setminus B: A,B \text{ are }G_\delta \text{ subsets of }\mathbb{T} \},
$$
cf. also \cite[Section 22.3]{MR1321597}. 
Observe that a subset $S\subseteq \mathbb{T}$ is $\mathsf{D}(\mathbf{\Pi}^0_2)$-complete if and only if $S \in \mathsf{D}(\mathbf{\Pi}^0_2)$ and $\mathbb{T}\setminus S \notin \mathsf{D}(\mathbf{\Pi}^0_2)$.  
Using a non-trivial result of Lupini \cite{MR5045103}, we get the following consequences. 
\begin{corollary}\label{corollary:lupini}
Let $\I$ be an analytic $P$-ideal on $\omega$. Then each proper infinite $\I$-characterized subgroup $H$ of $\mathbb{T}$ is either $F_{\sigma}$-complete or $F_{\sigma\delta}$-complete or 
$\mathsf{D}(\mathbf{\Pi}^0_2)$-complete. 
\end{corollary}

To sum up, if $\I$ is an analytic $P$-ideal on $\omega$, then each $\I$-characterized subgroup is $F_{\sigma\delta}$ and admits a finer Polish group topology generated by a translation-invariant metric. In the case $\I=\mathrm{Fin}$, another necessary condition was found by Gabriyelyan in \cite[Corollary 1]{MR2729343}, namely, every characterized subgroup $H$ of $\mathbb{T}$ is a locally quasi-convex Polishable subgroup. 
More explicitly, 
we recall that a group topology $\tau$ on a subgroup $H$ of $\mathbb{T}$ (in our case, the finer Polish group topology) is \emph{locally quasi-convex} if 
for every \(\tau\)-neighborhood \(U\) of \(0\), there exists a
\(\tau\)-neighborhood \(V\) of \(0\) contained in $U$ such that 
$$
\forall y \in H\setminus V, \exists \phi \in \widehat{H}, \quad \sup_{x \in V}\|\phi(x)\|\le \nicefrac{1}{4}<\|\phi(y)\|.
$$ 
Here, $\widehat{H}$ stands for the set of $\tau$-continuous characters on $H$. As it is clear from its definition, the property of local quasi-convexity does not depend on the choice of the compatible metric generating $\tau$. In the following result, we find all ideals $\I$ for which every $\I$-characterized subgroup is locally quasi-convex Polishable. 

To this aim, we will use also the Fubini sum $\mathrm{Fin}\oplus \mathcal{P}(\omega):=\{A\subseteq \{0,1\}\times \omega: |A \cap \{0\}\times \omega|<\infty\}$ and the Fubini product 
$$
\emptyset \otimes \mathrm{Fin}:=
\left\{
A\subseteq\omega^2:
A_i\in\mathrm{Fin}
\text{ for all }i\in\omega
\right\}
$$
where $A_i:=\{j \in \omega: (i,j) \in A\}$ for all $A\subseteq \omega^2$ and $i \in \omega$. 

Observe that, thanks to \cite[Theorem 2.6(i)]{FKLT26}, if $\I$ and $\J$ are isomorphic ideals then $\mathscr{H}(\I)=\mathscr{H}(\J)$; hence it makes sense to write, for instance, $\mathscr{H}(\emptyset\otimes \mathrm{Fin})$. 

%

\begin{proposition}\label{prop:locallyquasiconvex}
Let $\I$ be an ideal on $\omega$. Then the following are equivalent\textup{:}
\begin{enumerate}[label={\rm (\roman*)}]
\item \label{item:1countablygenerated} $\I=\mathrm{Fin}$ or $\I$ is isomorphic to $\mathrm{Fin}\oplus \mathcal{P}(\omega)$ or $\emptyset \otimes \mathrm{Fin}$\textup{;}
\item \label{item:2countablygenerated} Every $\I$-characterized subgroup of $\mathbb{T}$ is locally quasi-convex Polishable\textup{.}
\end{enumerate}
\end{proposition}

As we recall in the next example, not all Polishable subgroups are locally quasi-convex.
\begin{example}\label{example:gabriy}
    Following Gabriyelyan \cite{Gabriy}, the subgroup 
    $$
    G_2:=\left\{ x \in \mathbb{T}: \sum_{n\in \omega}\left\|2^{n^2}x\right\|^2<\infty \right\}
    $$
    is $F_{\sigma}$ Polishable, but not locally quasi-convex, see \cite[Proposition 1 and Theorem 2]{Gabriy}. It follows by Proposition \ref{prop:locallyquasiconvex} that $G_2$ is not characterized. Since every countable subgroup is characterized (see B\'ir\'o, Deshouillers and S\'os \cite{MR1877772}), we obtain that $G_2$ is uncountable. Taking into account that $G_2$ is Borel, it follows by \cite[Theorem 13.6]{MR1321597} that $G_2$ has cardinality $\mathfrak{c}$. Finally, we will obtain from Theorem \ref{thm:Fsigmadelcomplete} below that $G_2$ is also not $\I$-characterized for every generalized density ideal $\I$. 
\end{example}

For the next result, we say that a topology $\tau$ on a subgroup $H\subseteq \mathbb{T}$ is \emph{uniformly free from small subgroups}, shortened as \emph{UFSS}, if there exists a neighborhood $U$ of $0$ such that the family of sets $\{U_m: m\ge 1\}$, where 
$$
U_m:=\left\{x\in H:kx\in U\text{ for every }1\leq k\leq m\right\},
$$
form a neighborhood basis at $0$; see \cite[Definition 3.1]{MR2651665} and \cite[p. 239]{MR263969}. 

\begin{proposition}\label{prop:locallyquasiconvex2}
Let $\I$ be an ideal on $\omega$. Then the following are equivalent\textup{:}
\begin{enumerate}[label={\rm (\roman*)}]
\item \label{item:1countablygenerated222} $\I=\mathrm{Fin}$ or $\I$ is isomorphic to $\mathrm{Fin}\oplus \mathcal{P}(\omega)$\textup{;}
\item \label{item:2countablygenerated222} Every $\I$-characterized subgroup of $\mathbb{T}$ is locally quasi-convex Polishable and its Polish group topology is UFSS\textup{.}
\end{enumerate}
\end{proposition}

As a consequence of the previous results, we have the following. 
\begin{corollary}\label{cor:strictFinemptysetFin}
    $\mathscr H(\mathrm{Fin})
    \subsetneq
    \mathscr H(\emptyset\otimes\mathrm{Fin})\subsetneq \mathscr{H}(\mathcal{Z})$.  
\end{corollary}

\subsection{Topological dichotomies} 
Several questions have been asked in the literature regarding the Borel complexities of characterized subgroups. Among others, we recall below the following ones, see \cite[Section 4]{MR3288121}: 
\begin{enumerate}
    \item [(a)] \label{item:questionA}
    Does there exist a characterized subgroup which is not $G_{\delta\sigma}$? 
    \item [(b)] \label{item:questionB}
    Does there exist an explicit method to decide whether a characterized subgroup is $G_{\delta\sigma}$? And a method to decide whether it is $F_\sigma$?
    \item [(c)] \label{item:questionC}
    Does there exist a proper uncountable $F_\sigma$ characterized subgroup?
\end{enumerate}
We answer all the above questions with the aid of the following result. To this aim, we recall that an ideal $\mathcal{I}$ on $\omega$ is said to be a \emph{generalized density ideal} if there exists a sequence $(\varphi_k:k \in \omega)$ of lscsms with finite pairwise disjoint supports such that 
$$
\mathcal{I}=\left\{S\subseteq \omega: \lim_{k\to \infty} \varphi_k(S)=0\right\}.
$$
Generalized density ideals have been introduced by Farah in \cite[Section 2.10]{MR1988247}, see also \cite{MR2254542}, and have been used in different contexts, see e.g. \cite{MR3436368, MR2320288}. 
It is clear generalized density ideals are analytic $P$-ideals and, as it follows by \cite[Proposition 2.3]{MR4404626} the unique $F_\sigma$ generalized density ideals are $\mathrm{Fin}$ and $\mathrm{Fin}\oplus \mathcal{P}(\omega)$. However, the family of generalized density ideals is very rich. Indeed, in addition, it includes $\emptyset\otimes \mathrm{Fin}$, $\mathcal{Z}$, the density ideals as defined in \cite[Section 1.13]{MR1711328}, all the Erd{\H o}s--Ulam ideals introduced by Just and Krawczyk in \cite{MR748847}, and a large class of generalized density ideals has been defined by Louveau and Veli\v{c}kovi\'{c} in \cite{MR1708151, Louveau1994}, cf. also \cite[Section 2.11]{MR1988247}.



\begin{theorem}\label{thm:Fsigmadelcomplete}
Let $H$ be a proper $\I$-characterized subgroup of $\mathbb{T}$, where $\I$ is a generalized density ideal on $\omega$. Then $H$ is either countable or $F_{\sigma\delta}$-complete.
\end{theorem}

It follows that, if $H$ is a proper uncountable characterized subgroup of $\mathbb{T}$, it is necessarily $F_{\sigma\delta}$-complete. This is case, for instance, of $\mathsf{H}_{(n!)}(\mathrm{Fin})$, cf. Remark \ref{rmk:examplecomplete} below. 
In addition, it follows by Theorem \ref{thm:Fsigmadelcomplete} and B\'ir{\'o}, Deshouillers and S{\'o}s' result \cite{MR1877772} that a proper $G_{\delta\sigma}$ subgroup is characterized if and only if it is $F_\sigma$ if and only if it is countable. 
Lastly, Theorem \ref{thm:Fsigmadelcomplete} implies that there are no proper uncountable $F_\sigma$ characterized subgroups of $\mathbb{T}$; it is remarkable, as observed in \cite[Section 4]{MR3288121}, that the analogue claim where $\mathbb{T}$ is replaced by a non-isomorphic compact metrizable abelian group fail. 
These observations answer questions (a)--(c) above.


At this point, it is natural to ask whether Theorem \ref{thm:Fsigmadelcomplete} can be extended to all analytic $P$-ideals. In particular, since $\mathrm{Fin}\subseteq \I_{1/n} \subseteq \mathcal{Z}$ by \eqref{eq:inclusionideals} and both $\mathrm{Fin}$ and $\mathcal{Z}$ are generalized density ideals, one might be tempted to conjecture that the analogue claim holds for the summable ideal $\I_{1/n}$. However, rather surprisingly, it fails.

\begin{theorem}\label{thm:strangecounterexample}
    There exists $H\in \mathscr{H}(\I_{1/n})$ which is neither $F_\sigma$ nor $F_{\sigma\delta}$-complete.
\end{theorem}

We remark that, thanks to Corollary \ref{corollary:lupini}, the subgroup $H$ constructed in Theorem \ref{thm:strangecounterexample} 
is necessarily $\mathsf{D}(\mathbf{\Pi}^0_2)$-complete. In addition, we show that $H$ is  locally quasi-convex Polishable in Remark \ref{rmk:strangeLQC} below. 

Lastly, by a suitable modification of the construction given in the proof of Theorem \ref{thm:strangecounterexample}, it is possible to show that there exists an $\I_{1/n}$-characterized subgroup of $\mathbb{T}$ which is $\mathsf{D}(\mathbf{\Pi}^0_2)$-complete and Polishable, but not locally quasi-convex, see Remark \ref{rmk:strangenotLQCksdjhg} below.


\section{Examples and Applications}\label{sec:exampleapplications}

We start providing some explicit examples of $\I$-characterized subgroups (for some special choices for $\bm{a}$). Then we continue with some applications and a summarizing figure in Section \ref{subsec:applications}. 

\subsection{Examples: Powers} 
Given an integer $b\ge 2$, let us write each $x \in \mathbb{T}$ as 
\begin{equation}\label{eq:representationx}
x=\sum_{k \in \omega}
d_{x,k}b^{-k}
=\overline{d_0.d_1d_2d_3\ldots}_{(b)}
\end{equation}
through its canonical (unique) base $b$-representation, where $d_{x,0}:=0$, $d_{x,k} \in \{0,1,\ldots,b-1\}$ for all $k \in \omega$, and $d_{x,k}\neq b-1$ for infinitely many $k \in\omega$. 
If the base $b\ge 2$ is understood, we write $\overline{d_1\cdots d_k}$ for a block of digits $d_1,\ldots,d_k \in \{0,1,\ldots,b-1\}$; a block of $k$ consecutive digits $d \in \{0,1,\ldots,b-1\}$ is denoted by $d^k$.

Armacost proved in \cite{MR637201} that, if $p\ge 2$ is a prime number, then $\mathsf{H}_{(p^n)}(\mathrm{Fin})$ is precisely the $p$-Pr\"ufer group $\mathbb{Z}(p^\infty)=\mathbb{Z}[1/p]/\mathbb{Z}$, that is, 
$$
\mathsf{H}_{(p^n)}(\mathrm{Fin})=\left\{\,\frac{r}{p^m}: r \in \mathbb{Z} \text{ and }m \in \omega\right\}/\mathbb{Z}.
$$
In the following result, given an integer $b\ge 2$ and an ideal $\I$ on $\omega$, we characterize the subgroup $\mathsf{H}_{(b^n)}(\I)$ 
(cf. also \cite{MR4264228} for the special case $\mathcal{I}=\mathcal{Z}$). 
\begin{theorem}\label{thm:armacost}
    Let $\mathcal{I}$ be an ideal on $\omega$, and pick an integer $b\ge 2$. 
    Then 
    \begin{equation}\label{eq:pruferII}
    \mathsf{H}_{(b^n)}(\mathcal{I})=
    \left\{x \in \mathbb{T}: P_x-m \in \I \text{ and }Q_x-m \in \I \text{ for all }m\ge 1\right\},
    \end{equation}
    where $P_x:=\{k \in \omega: d_{x,k} \in \{1,\ldots,b-2\}\}$ and $Q_x:=\{k \in \omega: d_{x,k}\neq d_{x,k+1}\}$.
\end{theorem}
Observe that, if $\mathcal{I}=\mathrm{Fin}$ and $b:=p\ge 2$ is prime then the subgroup in \eqref{eq:pruferII} coincides with $p$-Pr\"ufer group. 
%
Hereafter, we say that an ideal $\I$ on $\omega$ is  
\emph{translation invariant} if $A+k:=\{n\in\omega:n-k\in A\}\in\I$ for all $A\in\I$ and $k\in\mathbb{Z}$. 
Accordingly, we have the following simpler characterization.
\begin{corollary}\label{cor:armacostconsequence}
Let $\I$ be 
a
translation invariant ideal on $\omega$, and pick an integer $b\ge 2$. 
Then
$$
\forall b\ge 2, \quad 
\mathsf{H}_{(b^n)}(\I)=\left\{x \in \mathbb{T}: P_x\cup Q_x \in \I\right\}.
$$

In particular, if $b=2$, then
$$
\mathsf{H}_{(2^n)}(\I)=\left\{x \in \mathbb{T}: \{k \in \omega: d_{x,k}\neq d_{x,k+1}\} \in \I\right\}.
$$
\end{corollary}

In the case $\I=\mathcal{Z}$, Corollary \ref{cor:armacostconsequence} provides an answer to \cite[Problem 6.9]{MR4932230}. 
It turns out that the addition hypothesis of translation invariance provide us some interesting consequences. 
\begin{theorem}\label{thm:propertiespowers}
    Let $\mathcal{I}, \mathcal{J}$ be 
    translation invariant ideals on $\omega$, and pick an integer $b\ge 2$. Then the following hold\textup{:}
     \begin{enumerate}[label={\rm (\roman*)}]
     \item \label{prop:1powers} $\mathsf{H}_{(b^n)}(\I)$ is countable if and only if $\I=\mathrm{Fin}$\textup{;}
     \item \label{prop:2powers} $\mathsf{H}_{(b^n)}(\I)$ is $F_\sigma$ if and only if $\I$ is an $F_\sigma$ ideal\textup{;}
     \item \label{prop:3powers} $\mathcal{I}=\mathcal{J}$ if and only if $\mathsf{H}_{(b^n)}(\I)=\mathsf{H}_{(b^n)}(\J)$\textup{.}
     \end{enumerate}
\end{theorem}

In particular, Theorem \ref{thm:propertiespowers}\ref{prop:3powers} strengthens \cite[Proposition 3.2]{MR4078214}, which studies the case $\I=\mathrm{Fin}$, $\J=\mathcal{Z}$, and $b=2$. 
    It is worth noting that the hypothesis of translation invariance is fundamental in each item of Theorem \ref{thm:propertiespowers}. To see it, define 
    $$
    \J_1:=\{S\subseteq \omega: S\cap 2\omega \in \mathrm{Fin}\}
    \quad \text{ and }\quad 
    \J_2:=\{S\in \J_1: S\setminus 2\omega \in \mathcal{Z}\}. 
    $$
    That is, $\J_1$ and $\J_2$ are isomorphic copies on $\omega$ of $\mathrm{Fin}\oplus \mathrm{Fin}$ and $\mathrm{Fin}\oplus \mathcal{Z}$, respectively.  
    It is immediate to see that both $\J_1$ and $\J_2$ are not translation invariant, they are distinct from $\mathrm{Fin}$, and $\mathsf{H}_{(b^n)}(\mathrm{Fin})=\mathsf{H}_{(b^n)}(\J_1)=\mathsf{H}_{(b^n)}(\J_2)$ (which is countable by item \ref{prop:1powers}). At the same time, $\J_1$ is an $F_\sigma$ ideal, while $\J_2$ is not $F_\sigma$. Hence all items \ref{prop:1powers}--\ref{prop:3powers} fail without the hypothesis of translation invariance. 
    

\begin{corollary}\label{cor:consequencepowers22}
Pick an integer $b\ge 2$ and let $\I$ be a 
translation invariant ideal on $\omega$. Then the following hold\textup{:}
\begin{enumerate}[label={\rm (\roman*)}]
\item \label{item:1corpowers} If $\I=\mathrm{Fin}$ then $\mathsf{H}_{(b^n)}(\I)$ is a countably infinite subgroup\textup{;}
\item \label{item:2corpowers} If $\I\neq \mathrm{Fin}$ is $F_\sigma$ then $\mathsf{H}_{(b^n)}(\I)$ is an uncountable proper $F_\sigma$ subgroup\textup{;}
\item \label{item:3corpowers} If $\I\neq \mathrm{Fin}$ is a generalized density ideal  then $\mathsf{H}_{(b^n)}(\I)$ is an $F_{\sigma\delta}$-complete subgroup\textup{.}
\end{enumerate}
\end{corollary}

\subsection{Examples: Fibonacci} 
Now, let us consider the Fibonacci sequence $(f_n: n \in \omega)$ defined by $f_0:=0$, $f_1:=1$ and $f_{n+2}:=f_{n+1}+f_n$ for all $n \in \omega$. 
For each $x \in [0,1)$ and $n \in \omega$, let $p_n(x)$ be a  nearest integer to $f_nx$, e.g., $p_n(x):=\lfloor f_nx+\nicefrac{1}{2}\rfloor$. Define 
$$
\forall n \in \omega, \quad 
r_{x,n}:=p_{n+2}(x)-p_{n+1}(x)-p_n(x).
$$
Accordingly, denote the golden ratio by $\varphi:=\frac{1+\sqrt{5}}{2}$. 
\begin{lemma}\label{lem:fiborepresent}
    Fix $x \in \mathbb{T}$. Then 
    $r_{x,n}\in \{-1,0,1\}$ for all $n \in \omega$. 
    In addition, 
    $$
    x=\sum_{n \in \omega} 
    r_{x,n} \varphi^{-n-1} 
    \quad 
    \text{ in }\mathbb{T}.
    $$
\end{lemma}

Using the above simple representation, we describe the elements of $\mathsf{H}_{(f_n)}(\I)$, in the case where $\I$ is translation invariant. To this aim, for each $x \in \mathbb{T}$, we write 
$$
R_x:=\{n \in \omega: r_{x,n}\neq 0\}
$$
for the \textquotedblleft support\textquotedblright\,of $x \in \mathbb{T}$  in the representation given in Lemma \ref{lem:fiborepresent}. This allows to describe explicitly the elements of $\mathsf{H}_{(f_n)}(\I)$.
\begin{theorem}\label{thm:fibonacci}
Let $\I$ be a translation invariant ideal on $\omega$. Then 
$$
\mathsf{H}_{(f_n)}(\I)=\{x\in \mathbb{T}: R_x \in \I\}. 
$$
\end{theorem}
As a special case, it follows that $\mathsf{H}_{(f_n)}(\mathrm{Fin})$ coincides with the subgroup generated by $\varphi$, which is known since the work of Larcher \cite{MR947645}. 
In the case  $\I=\mathcal{Z}$, Theorem \ref{thm:fibonacci} provides an answer to \cite[Problem 7.9]{MR4932230}. 

Mimicking Theorem \ref{thm:propertiespowers} and Corollary \ref{cor:consequencepowers22} above, we obtain the analogue results.

\begin{theorem}\label{thm:propertiesfibo}
    Let $\mathcal{I}, \mathcal{J}$ be 
    translation invariant ideals on $\omega$. Then the following hold\textup{:}
     \begin{enumerate}[label={\rm (\roman*)}]
     \item \label{prop:1fibo} $\mathsf{H}_{(f_n)}(\I)$ is countable if and only if $\I=\mathrm{Fin}$\textup{;}
     \item \label{prop:2fibo} $\mathsf{H}_{(f_n)}(\I)$ is $F_\sigma$ if and only if $\I$ is an $F_\sigma$ ideal\textup{;}
     \item \label{prop:3fibo} $\mathcal{I}=\mathcal{J}$ if and only if $\mathsf{H}_{(f_n)}(\I)=\mathsf{H}_{(f_n)}(\J)$\textup{.}
     \end{enumerate}
\end{theorem}

\begin{corollary}\label{cor:consequencefibo22}
Let $\I$ be a 
translation invariant ideal on $\omega$. Then the following hold\textup{:}
\begin{enumerate}[label={\rm (\roman*)}]
\item \label{item:1corpowersfibo} If $\I=\mathrm{Fin}$ then $\mathsf{H}_{(f_n)}(\I)$ is a countably infinite subgroup\textup{;}
\item \label{item:2corpowersfibo} If $\I\neq \mathrm{Fin}$ is $F_\sigma$ then $\mathsf{H}_{(f_n)}(\I)$ is an uncountable proper $F_\sigma$ subgroup\textup{;}
\item \label{item:3corpowersfibo} If $\I\neq \mathrm{Fin}$ is a generalized density ideal  then $\mathsf{H}_{(f_n)}(\I)$ is an $F_{\sigma\delta}$-complete subgroup\textup{.}
\end{enumerate}
\end{corollary}
Thus, $\mathsf{H}_{(f_n)}(\mathcal{Z})$ is an $F_{\sigma\delta}$-complete subgroup (hence, with cardinality $\mathfrak{c}$ and distinct from $\mathsf{H}_{(f_n)}(\mathrm{Fin})$), which answers \cite[Question 6.4]{MR4078214} and strengthens \cite[Corollary 2.5]{MR4458173}. 

In the case where $\I$ is not translation invariant (or even satisfies a stronger property as in Definition \cite[Definition 2.4]{MR4594809}), it is still possible that $\mathsf{H}_{(f_n)}(\I) \neq \mathsf{H}_{(f_n)}(\mathrm{Fin})$. The next example provides a negative answer to \cite[Conjecture 4.8]{MR4594809}.
\begin{example}\label{example:fibonaccino}
    Consider the irrational $\alpha:=\varphi^{-1}=[0;1,1,\ldots]$ and let $(a_n:n\in\omega)$ be the sequence of denominators of its convergents, so that $a_n=f_{n+1}$ for each $n \in \omega$. 
    For each $j \in \{0,1,2\}$, set $R_j:=3\omega+j$. Then $\{R_0,R_1,R_2\}$ is a partition of $\omega$ and  
    $$
    \J:=\{A\subseteq\omega:A\cap R_2\in\mathrm{Fin}\}
    $$ 
    is an $F_\sigma$ $P$-ideal. Moreover, if $A\in\J$ is an infinite set, then at least one between $A\cap R_0$ and $A\cap R_1$ is infinite, so that $A+2\notin\J$ or $A+1\notin\J$ (hence $\J$ satisfies the property stated in Definition \cite[Definition 2.4]{MR4594809}). Set also $x:=\nicefrac{1}{2}$. 

    Since $a_n=f_{n+1}$ is even if and only if $n\in R_2$, we have
    $$
    \|a_nx\|=
    \begin{cases}
    \,0&\text{ if }n\in R_2,\\
    \,x&\text{ if }n\notin R_2.
    \end{cases}
    $$
    Taking into account that $\mathsf{H}_{(a_n)}(\mathrm{Fin})=\mathsf{H}_{(f_n)}(\mathrm{Fin})$ contains only irrational values except $0$ and that $\omega\setminus R_2\in \J$, it follows that $x \in \mathsf H_{(a_n)}(\J)\setminus \mathsf{H}_{(a_n)}(\mathrm{Fin})$. In particular, we conclude that $\mathsf H_{(a_n)}(\J)\neq \mathsf{H}_{(a_n)}(\mathrm{Fin})$ and $\mathsf H_{(f_n)}(\I)\neq \mathsf{H}_{(f_n)}(\mathrm{Fin})$, where $\I:=\{A\subseteq \omega: A-1 \in \J\}$. 
\end{example}


\subsection{Examples: Factorials} 
Let us write each \(x\in\mathbb T\) in its canonical factorial expansion in the
shifted form
\[
x=
\sum_{n=1}^{\infty}\frac{q_{x,n}}{(n+1)!},
\]
where \(q_{x,n}\in\{0,1,\ldots,n\}\) for every \(n\geq1\), and
\(q_{x,n}\neq n\) for infinitely many \(n\in \omega\). As usual, this is the unique
such representation of \(x\). 
In the case of the sequence of factorials $(n!: n\ge 1)$, a special case of the following characterization (for translation invariant $P$-ideals) appeared recently in \cite[Theorem 1.8]{DiSantoDikranjanGiordanoBrunoWeber2026}, cf. also \cite{MR4780094}. 

\begin{theorem}\label{thm:factorialcharacterization}
Let \(\mathcal I\) be an ideal on \(\omega\). Then
$$
\mathsf{H}_{(n!:\, n\ge 1)}(\mathcal I)=
\left\{
x \in \mathbb{T}:
\mathcal I\text{-}\lim_{n\to\infty}
\frac{q_{x,n}}{n}=0
\,\,\text{ in }\mathbb{T}
\right\}.
$$
\end{theorem}

This provides an answer to \cite[Problem 6.10]{MR4932230} in the case of sequence $(n!: n\ge 1)$. 
Although we are not going to prove a list of special properties of $\mathsf{H}_{(n!)}(\I)$ as in the previous cases, we collect below some easy consequences of our results. 
\begin{remark}\label{rmk:examplecomplete}
    Let $\I$ be an ideal on $\omega$. 
    For each $S\subseteq \omega$, consider the point $x \in \mathbb{T}$ such that $q_{x,n}=1$ if $n \in S$ and $q_{x,n}:=0$ otherwise, for each $n\ge 1$. 
    It follows by Theorem \ref{thm:factorialcharacterization} that $x \in \mathsf{H}_{(n!)}(\I)$. At the same time, the point $y \in \mathbb{T}$ defined by $q_{y,n}:=\lfloor \nicefrac{n}{2}\rfloor$ for all $n\ge 1$ does not belong to $\mathsf{H}_{(n!)}(\I)$ by Theorem \ref{thm:factorialcharacterization}. 
    Therefore $\mathsf{H}_{(n!)}(\I)$ is a proper subgroup of $\mathbb{T}$ with cardinality $\mathfrak{c}$. It follows by Theorem \ref{thm:Fsigmadelcomplete} that both $\mathsf{H}_{(n!)}(\mathrm{Fin})$ and $\mathsf{H}_{(n!)}(\mathcal{Z})$ (and, more generally, $\mathsf{H}_{(n!)}(\I)$ with $\I$ generalized density ideal) are $F_{\sigma\delta}$-complete. 
    In addition, the latter subgroups are distinct, cf. Remark \ref{rmk:factorial} below. 
\end{remark}

Finally, differently from previous cases, we show that $\mathsf{H}_{(n!)}(\I)$ cannot be $F_\sigma$. 
\begin{proposition}\label{prop:factnotFsigma}
Let $\I$ be an ideal on $\omega$. Then $\mathsf{H}_{(n!)}(\I)$ is not an $F_\sigma$ subgroup. 
\end{proposition}

\subsection{Applications}\label{subsec:applications} 
As consequence of our results, we show that for every ideal $\I$ there exists an $\I$-characterized subgroup which is not $G_{\delta\sigma}$. 
\begin{proposition}\label{prop:noPi02}
    There are no ideals $\I$ on $\omega$ such that $\mathscr{H}(\I)\subseteq \mathbf{\Sigma}^0_3$. In particular, there are no ideals $\I$ such that every $\I$-characterized subgroup is $F_\sigma$. 
\end{proposition}

In particular, this improves on Corollary \ref{cor:borelrank}\ref{item:1borrank} in the case $m=3$:
\begin{corollary}\label{cor:Fsigmadeltacomplete}
Let $\I$ be an $F_{\sigma\delta}$ ideal on $\omega$. Then there exists an $F_{\sigma\delta}$-complete $\I$-characterized subgroup of $\mathbb{T}$. 
\end{corollary}

In the opposite direction of Proposition \ref{prop:noPi02}, notice that there are $2^{\mathfrak{c}}$ subgroups of $\mathbb{T}$, while $|\mathscr{H}(\I)|\le |\omega^\omega|=\mathfrak{c}$ for every ideal $\I$. This implies that, for every ideal $\I$ on $\omega$ there exists a subgroup of $\mathbb{T}$ which is not $\I$-characterized. 

Thus, one might ask whether there exist ideals which contains, let's say, all Borel subgroups. To answer it, we will need the following result.
\begin{proposition}\label{prop:RKorder}
    Let $\{\I_\xi: \xi <\mathfrak{c}\}$ be a family of \textup{(}not necessarily distinct\textup{)} ideals on $\omega$. Then there exists an ideal $\J$ on $\omega$ such that 
    $\mathscr{H}(\I_\xi)\subseteq \mathscr{H}(\J)$ for each $\xi<\mathfrak{c}$. 
\end{proposition}

As a consequence, we provide a strong affirmative answer: 

\begin{corollary}\label{cor:analyticconalaytc}
    Let $\{H_\xi: \xi<\mathfrak{c}\}$ be a family of \textup{(}not necessarily distinct\textup{)} subgroups of $\mathbb{T}$. Then there exists an ideal $\J$ on $\omega$ such that 
    $
    \left\{H_\xi: \xi<\mathfrak{c}\right\}\subseteq \mathscr{H}(\J).
    $ 
    
    In particular, there exists an ideal $\J$ on $\omega$ such that every analytic or coanalytic subgroup of $\mathbb{T}$ is $\J$-characterized.
\end{corollary}

Lastly, putting together our previous results, we determine when all $\I$-characterized subgroups are actually characterized. 
\begin{proposition}\label{prop:equalitymathscrHFin}
    Let $\I$ be an ideal on $\omega$. Then $\mathscr{H}(\I)=\mathscr{H}(\mathrm{Fin})$ if and only if either $\I=\mathrm{Fin}$ or $\I$ is isomorphic to $\mathrm{Fin}\oplus \mathcal{P}(\omega)$. 
\end{proposition}

As a consequence, we improve on \cite[Corollary 2.8]{FKLT26}:
\begin{corollary}\label{cor:lastone}
    Let $\I\neq \mathrm{Fin}$ be a meager ideal on $\omega$ which is not isomorphic to $\mathrm{Fin}\oplus \mathcal{P}(\omega)$. Then $\mathscr{H}(\mathrm{Fin})\subsetneq \mathscr{H}(\I)$. 
\end{corollary}

Finally, in Figure \ref{fig:schematic-characterized-polishable} below, we provide a picture summarizing properties of proper $\I$-characterized subgroups of $\mathbb{T}$ for the ideal $\mathrm{Fin}$, $\emptyset\otimes \mathrm{Fin}$, and $\mathcal{Z}$. More precisely: 
\begin{enumerate}[label={\rm (\roman*)}]
\item The dark gray area is empty by Theorem \ref{thm:Fsigmadelcomplete}; the light gray area does not contain any $\I$-characterized subgroup with $\I$ analytic $P$-ideal by Corollary \ref{corollary:lupini};
\item Every countable subgroup is characterized \cite{MR1877772};
\item $\mathscr H(\mathrm{Fin})
    \subsetneq
    \mathscr H(\emptyset\otimes\mathrm{Fin})\subsetneq \mathscr{H}(\mathcal{Z})$ by Corollary \ref{cor:strictFinemptysetFin}; 
\item Each $\emptyset\otimes\mathrm{Fin}$-characterized subgroup is locally quasi-convex Polishable thanks to Proposition \ref{prop:locallyquasiconvex}; each $\mathcal{Z}$-characterized subgroup is Polishable by Theorem \ref{thm:analyticPidealscharacterization};
\item \textbf{A}: $F_\sigma$ non-Polishable subgroup in Corollary \ref{cor:birocounterexample};
\item \textbf{B}: subgroup $G_2$ in Example \ref{example:gabriy}; 
\item \textbf{C}: $\mathsf{H}_{(n!)}(\mathrm{Fin})$ is $F_{\sigma\delta}$-complete by Remark \ref{rmk:examplecomplete};
\item \textbf{D}: subgroup defined in the proof of Proposition \ref{prop:locallyquasiconvex2}; 
\item \textbf{E}: the existence of a $\mathcal{Z}$-characterized subgroup which is $F_{\sigma\delta}$-complete but not locally quasi-convex Polishable follows by Theorem \ref{thm:Fsigmadelcomplete} and Proposition \ref{prop:locallyquasiconvex};
\item \textbf{F}: $F_{\sigma\delta}$-complete non-Polishable subgroup in Proposition \ref{prop:FsigmadeltacompletenotPolishable};
\item \textbf{G}: see Theorem \ref{thm:strangecounterexample} and Remark \ref{rmk:strangenotLQCksdjhg};
\item \textbf{H}: see Theorem \ref{thm:strangecounterexample} and Remark \ref{rmk:strangeLQC}.
\end{enumerate}

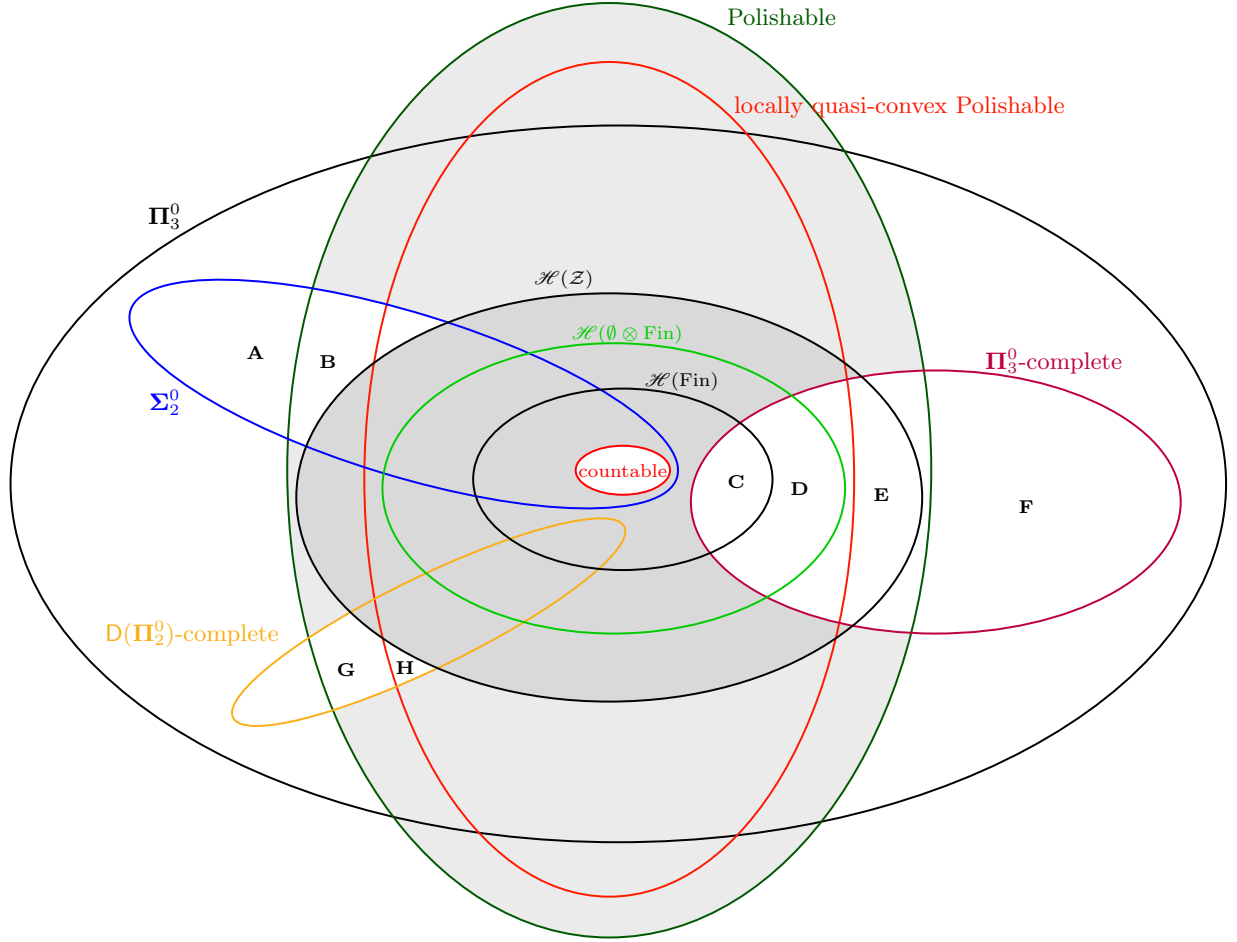
\begin{figure}[ht]
\centering
\begin{tikzpicture}[
    scale=1.2,
    every path/.style={line width=0.75pt},
    region/.style={draw}
]


\begin{scope}
\fill[gray!16] (-0.10,0.00) ellipse (3.55 and 5.15);
\end{scope}

\fill[white, rotate around={-17:(-2.55,0.75)}]
    (-2.47,0.89) ellipse (3.15 and 0.90);

\fill[white, rotate around={26:(-2.20,-1.45)}]
    (-2.20,-1.7) ellipse (2.65 and 0.50);

\fill[gray!30]
    (-0.10,-0.30) ellipse (3.45 and 2.25);

\fill[white]
    (3.5,-0.35) ellipse (2.70 and 1.45);
    
\fill[white]
    (0.05,0.00) ellipse (0.52 and 0.27);


\draw[region, black]
    (0.00,-0.15) ellipse (6.70 and 3.95);
\node at (-5,2.8) {\footnotesize $\mathbf{\Pi}^0_3$};

\draw[region, green!35!black]
    (-0.10,0.00) ellipse (3.55 and 5.15);
\node at (1.8,5) {\footnotesize \textcolor{green!35!black}{Polishable}};

\draw[region, red!80!orange]
    (-.1,-0.10) ellipse (2.7 and 4.6);
\node at (3.1,4) {\footnotesize \textcolor{red!80!orange}{locally quasi-convex Polishable}};

\draw[region, blue, rotate around={-17:(-2.55,0.75)}]
    (-2.4,0.89) ellipse (3.15 and 0.9);
\node at (-5,.75) {\footnotesize \textcolor{blue}{$\mathbf{\Sigma}^0_2$}};

\draw[region, yellow!30!orange, rotate around={26:(-2.20,-1.45)}]
    (-2.20,-1.7) ellipse (2.4 and 0.50);
\node at (-4.7,-1.8) {\footnotesize \textcolor{yellow!30!orange}{$\mathsf{D}(\mathbf{\Pi}^0_2)$-complete}};

\draw[region, purple]
    (3.5,-0.35) ellipse (2.70 and 1.45);
\node at (4.8,1.2) {\footnotesize \textcolor{purple}{$\mathbf{\Pi}^0_3$-complete}};

\draw[region, black]
    (-0.10,-0.30) ellipse (3.45 and 2.25);
\node at (.7,.98) {\tiny {$\mathscr{H}(\mathrm{Fin})$}};
\node at (-0.6,2.1) {\tiny {$\mathscr{H}(\mathcal{Z})$}};

\draw[region, green!80!black]
    (-0.05,-0.20) ellipse (2.55 and 1.60);
\node at (0.1,1.5) {\tiny \textcolor{green!80!black}{$\mathscr{H}(\emptyset\otimes \mathrm{Fin})$}};

\draw[region, black]
    (0.05,-0.10) ellipse (1.65 and 1.00);


\draw[region, red]
    (0.05,0.00) ellipse (0.52 and 0.27);
\node at (0.05,0.00) {\tiny \textcolor{red}{countable}};

\node at (-4,1.3) {\tiny {$\textbf{A}$}};
\node at (-3.2,1.2) {\tiny {$\textbf{B}$}};
\node at (1.3,-.13) {\tiny {$\textbf{C}$}};
\node at (2,-.2) {\tiny {$\textbf{D}$}};
\node at (2.9,-.27) {\tiny {$\textbf{E}$}};
\node at (4.5,-.4) {\tiny {$\textbf{F}$}};
\node at (-3,-2.2) {\tiny {$\textbf{G}$}};
\node at (-2.35,-2.18) {\tiny {$\textbf{H}$}};

\end{tikzpicture}
\caption{Proper $\I$-characterized subgroups of $\mathbb{T}$ for ideals  $\mathrm{Fin}, \emptyset \otimes \mathrm{Fin}, \mathcal{Z}$.}
\label{fig:schematic-characterized-polishable}
\end{figure}

\section{Proofs of Results from Section \ref{sec:mainresults}}\label{sec:proofsmain}

\begin{proof}
[Proof of Theorem \ref{thm:increasingSequenceSameFin}]
As it follows by \cite{Schoen} (cf. \cite[Theorem D]{MR4932230} and also \cite[Theorem 7.8]{MR419394} for a textbook exposition) 
there exists a sequence $\bm{b} \in \omega^\omega$ such that 
$$
H=\mathsf{H}_{\bm{b}}(\mathrm{Fin})
\quad \text{ and }\quad 
\mathrm{supp}(\bm{b})\in \mathrm{Fin}^+.
$$
Define 
$
P:=\{q\in\omega\setminus \{0\}: \{n \in \omega: b_n=q\} \notin \mathrm{Fin}\}
$ 
and 
$
R:=\{q\in\omega\setminus \{0\}: \emptyset \neq \{n \in \omega: b_n=q\} \in \mathrm{Fin}\}.
$ 
Since the support of \(\bm b\) is infinite, at least one of the sets \(P\) and \(R\) is nonempty. Moreover, if \(P=\emptyset\), then \(R\) is infinite. At this point, set 
$$
P^\natural:=\{x \in \mathbb{T}: \forall q \in P, qx=0\}
$$ 
and 
$$
Q^\natural:=\{x \in \mathbb{T}: \forall \varepsilon>0, \{q\in R:\|qx\|\geq\varepsilon\}\in\mathrm{Fin}\}.
$$ 

\begin{claim}\label{claim:HsplitPnaturalQnatural}
$H=P^\natural\cap Q^\natural$. 
\end{claim}
\begin{proof} 
First, suppose that \(x\in H\), so that $\lim_n b_nx=0$. If $P=\emptyset$ then $H\subseteq \mathbb{T}=P^\natural$. Otherwise, pick \(q\in P\). Then there exists $S_q \in \mathrm{Fin}^+$ such that $b_n=q$ for all $n \in S_q$. In particular, we have $\lim_{n\in S_q} b_nx=0$ so that $qx=0$, i.e., $H\subseteq P^\natural$. 

Now fix \(\varepsilon>0\) and suppose for the sake of contradiction that $T:=\{q\in R:\|qx\|\geq\varepsilon\}$ is infinite. For each $q \in T$, pick $n_q \in \omega$ such that $b_{n_q}=q$. Since the integers $\{n_q: q \in T\}$ are distinct, we have $\|b_{n_q}x\|\ge \varepsilon$ for all $q \in T$. This contradicts the standing hypothesis $\lim_n b_nx=0$. Therefore also $H\subseteq Q^\natural$. 

Conversely, we need to show that $P^\natural\cap Q^\natural\subseteq H$. To this aim, pick $x \in P^\natural\cap Q^\natural$. Fix \(\varepsilon>0\). To complete the proof, it will be enough to show that  $A_\varepsilon:=\{n\in\omega:\|b_nx\|\geq\varepsilon\}$ is finite.  
In fact, if \(n\in A_\varepsilon\), then \(b_n\neq0\), so that $b_n \in P\cup R$. If $b_n \in P$ then $b_nx=0$, which is impossible for $n \in A_\varepsilon$. Hence 
$
A_\varepsilon
\subseteq
\{n\in\omega:b_n\in R\text{ and }\|b_nx\|\geq\varepsilon\}.
$ 
By assumption, the set $F_\varepsilon:=\{q\in R:\|qx\|\geq\varepsilon\}$ is finite. 
Since each \(q\in R\) occurs only finitely many times among the values \(b_n\), we get
\[
A_\varepsilon
\subseteq
\bigcup_{q\in F_\varepsilon}\{n\in\omega: b_n =q\}\in\mathrm{Fin}.
\]
Therefore $x \in H$. This proves the converse inequality.
\end{proof}

\begin{claim}\label{claim:claimPnonempty}
If $P\neq \emptyset$ and $g:=\mathrm{gcd}(P)$, then 
$P^\natural=\{x \in \mathbb{T}: gx=0\}$. 
\end{claim}
\begin{proof}
    Suppose that $P\neq \emptyset$ and fix $x \in \mathbb{T}$. If $gx=0$ then $qx=0$ for all $q \in P$ since $g$ divides every $q \in P$. To prove the converse, since $g:=\mathrm{gcd}(P)$, it follows that there exist \(q_0,\ldots,q_{r-1}\in P\) and \(c_0,\ldots,c_{r-1}\in\mathbb Z\) such that $g=c_0q_0+\cdots+c_{r-1}q_{r-1}$. Hence, if $x \in P^\natural$ then $q_ix=0$ for all $i<r$, which implies that $gx=0$. 
\end{proof}

Now, we construct the strictly increasing sequence \(\bm c\) of positive integers such that 
\begin{equation}\label{eq:representationwithsequencec}
H=\mathsf{H}_{\bm{c}}(\mathrm{Fin}). 
\end{equation} 

\medskip

\textsc{Case 1}: \(P=\emptyset\). Then \(R\) is infinite. Let $\bm c=(c_n:n\in\omega)$ be the strictly increasing enumeration of \(R\). For every \(x\in\mathbb T\), the convergence $\lim_nc_nx=0$ is equivalent to saying that, for every \(\varepsilon>0\), there are only finitely many \(q\in R\) for which \(\|qx\|\geq\varepsilon\). Since \(P=\emptyset\), it follows by Claim \ref{claim:HsplitPnaturalQnatural} that 
\eqref{eq:representationwithsequencec} holds. 
 
\medskip

\textsc{Case 2}: \(P\neq \emptyset\). Set $g:=\gcd(P)$ and let $\bm c=(c_n:n\in\omega)$ be the strictly increasing enumeration of the infinite set of positive integers 
\[
R\cup\{g,2g,3g,\ldots\}.
\]

To show \eqref{eq:representationwithsequencec}, fix \(x\in\mathbb T\). Note that the condition $x \in \mathsf H_{\bm c}(\mathrm{Fin})$ (i.e., $\lim_n c_nx=0$), is equivalent to 
$$
\forall \varepsilon>0, \qquad
\{q\in R:\|qx\|\geq\varepsilon\}\in\mathrm{Fin}
\quad \text{ and }\quad 
\{k\in \omega :\|kgx\|\geq\varepsilon\}\in\mathrm{Fin}.
$$
With the above notation, the first condition is precisely $x \in Q^\natural$. In addition, the second condition is equivalent to $\lim_k kgx=0$, that is, $gx \in \mathsf{H}_{\bm{u}}(\mathrm{Fin})$. 
It follows that $gx=\lim_k((k+1)gx)-\lim_k kgx=0$. 
In turn, the latter is equivalent to $x \in P^\natural$ by Claim \ref{claim:claimPnonempty}. 
To sum up, $x \in \mathsf H_{\bm c}(\mathrm{Fin})$ if and only if $x \in P^\natural \cap Q^\natural$. It follows by Claim \ref{claim:HsplitPnaturalQnatural} that \eqref{eq:representationwithsequencec} holds. 

\medskip

To complete the proof, fix $\bm{r}\in \mathbb{Z}^\omega$ and let $\bm{c}$ be a strictly increasing sequence of positive integers  satisfying \eqref{eq:representationwithsequencec}. Since $\lim_n c_n=\infty$, it is possible to choose inductively a strictly increasing
sequence \((m_k:k\in\omega)\) in \(\omega\) such that, for all $k \in \omega$: 
 \begin{enumerate}[label={\rm (\roman*)}]
\item $c_{m_k}\ge r_{2k}$, 
\item $c_{m_k}+c_k\ge r_{2k+1}$, and  
\item $c_{m_{k+1}}>c_{m_{k}}+c_{k}$. 
 \end{enumerate}

Finally, define the sequence \(\bm a=(a_n:n\in\omega)\) by
\[
\forall k \in \omega, \qquad 
a_{2k}:=c_{m_k}
\quad\text{and}\quad
a_{2k+1}:=c_{m_k}+c_k
\]
By construction, \(a_n\ge r_n\) for all \(n\in\omega\). In addition, $a_0=c_{m_0}>0$, and $\bm{a}$ is strictly increasing. 
Taking into account \eqref{eq:representationwithsequencec}, it will be enough to show that 
$$
\mathsf H_{\bm a}(\mathrm{Fin})=\mathsf H_{\bm c}(\mathrm{Fin}).
$$
Indeed, if \(x\in\mathsf H_{\bm c}(\mathrm{Fin})\), then $\lim_k c_kx=0$ and
\(\lim_k c_{m_k}x=0\). Therefore \(\lim_k a_{2k}x=0\) and
\(\lim_k a_{2k+1}x=\lim_k(c_{m_k}+c_k)x=0\), so that
\(x\in\mathsf H_{\bm a}(\mathrm{Fin})\). Hence $\mathsf H_{\bm c}(\mathrm{Fin})\subseteq \mathsf H_{\bm a}(\mathrm{Fin})$.

Conversely, pick \(x\in\mathsf H_{\bm a}(\mathrm{Fin})\), so that 
\(\lim_k a_{2k}x=0\) and \(\lim_k a_{2k+1}x=0\). Since $c_k=a_{2k+1}-a_{2k}$ 
for every $k \in \omega$, we get \(\lim_k c_kx=0\). Hence also $\mathsf H_{\bm a}(\mathrm{Fin})\subseteq \mathsf H_{\bm c}(\mathrm{Fin})$. 

As claimed, we conclude that $H=\mathsf{H}_{\bm{c}}(\mathrm{Fin})=\mathsf{H}_{\bm{a}}(\mathrm{Fin})$.
\end{proof}

\medskip

\begin{proof}
    [Proof of Theorem \ref{thm:upperbound}]
    Fix $\bm{a} \in \omega^\omega$. For every $k \in \omega$ define $\Psi_k:\mathbb T\to 2^\omega$ by 
    $$
    \forall x \in \mathbb{T}, \quad 
    \Psi_k(x):=\{n\in\omega:\|a_nx\|\ge 2^{-k}\}. 
    $$
    \begin{claim}\label{claim:1Psik}
        Each $\Psi_k$ has at most countably many discontinuities. 
    \end{claim}
    \begin{proof}
        Fix $k \in \omega$. For each $n \in \omega$, the \(n\)-th coordinate of \(\Psi_k\) is the
characteristic function of the closed set 
 $
\{x\in\mathbb T:\|a_nx\|\ge 2^{-k}\}.
$ 
This coordinate can be discontinuous only at points \(x\in\mathbb T\) satisfying 
$ 
\|a_nx\|=2^{-k},
$ 
hence at finitely many points. 
Therefore the set of discontinuities of \(\Psi_k\) is contained in a countable union of finite
sets. 
\end{proof}

\begin{claim}\label{claim:Pi0m}
For each $k \in \omega$, we have $\Psi_k^{-1}[\I]\in\mathbf{\Pi}^0_\xi(\mathbb{T})$.
\end{claim}
\begin{proof}
Fix \(k\in\omega\). 
Let \(D_k\) be the set of discontinuities of \(\Psi_k\), which is countable by Claim \ref{claim:1Psik}. 
Define $Z_k:=\mathbb T\setminus D_k$. 
Then \(Z_k\) is a \(G_\delta\) subset of \(\mathbb T\), and
\(\Psi_k\restriction Z_k\) is continuous. Since
\(\mathcal I\in\mathbf{\Pi}^0_\xi(2^\omega)\), we have 
$
(\Psi_k\restriction Z_k)^{-1}[\mathcal I]\in
\mathbf{\Pi}^0_\xi(Z_k).
$ 
Hence, there exists
\(C_k\in\mathbf{\Pi}^0_\xi(\mathbb T)\) such that 
$
(\Psi_k\restriction Z_k)^{-1}[\mathcal I]=C_k\cap Z_k.
$ 
This implies that 
\[
\Psi_k^{-1}[\mathcal I]
=
(C_k\cap Z_k)\cup \bigl(D_k\cap \Psi_k^{-1}[\mathcal I]\bigr).
\]
The first term belongs to \(\mathbf{\Pi}^0_\xi(\mathbb T)\), while the second one
is countable, hence belongs to \(\mathbf{\Sigma}^0_2(\mathbb T)\subseteq
\mathbf{\Pi}^0_\xi(\mathbb T)\), because \(\xi\ge 3\). This proves the claim because 
\(\mathbf{\Pi}^0_\xi(\mathbb{T})\) is closed under finite unions.
\end{proof}

Taking into account the identity $\mathsf{H}_{\bm{a}}(\I)=\bigcap_k \Psi_k^{-1}[\I]$, the conclusion follows by Claim \ref{claim:Pi0m} and the fact that $\mathbf{\Pi}^0_\xi(\mathbb{T})$ is closed under countable intersections. 
\end{proof}

\medskip

\begin{proof}
    [Proof of Theorem \ref{thm:wadge}]
    We shall construct a
    sequence $\bm a\in\omega^\omega$ and a continuous map
    $\Phi:2^\omega\to\mathbb T$ such that
    \begin{equation}\label{eq:claimwadgefirstresult}
    \Phi^{-1}[\mathsf H_{\bm a}(\mathcal I)]=\mathcal I.
    \end{equation}

    To this aim, for each $n \in \omega$ define 
    \begin{equation}\label{eq:defawadge}
    a_n:=2^{2^{n+2}-1} 
    \quad \text{ and }\quad q_n:=2^{2^{n+2}}. 
    \end{equation}
    Observe that, for all $m,n \in \omega$ with $m<n$, we have $q_n/(2q_m) \in \omega$, and that
    $$
    \delta_n:=\sum_{k>n}\frac{q_n}{2q_k}
    =\frac{1}{2}\sum_{j=1}^\infty\frac{1}{2^{2^{n+2}(2^j-1)}}
    \le \frac{1}{2}\sum_{j=1}^\infty \frac{1}{2^{2^{n+2}j}}
    $$
    for all $n \in \omega$. Hence $\delta_n<\nicefrac{1}{8}$ for all $n \in \omega$, and $\lim_n \delta_n=0$. 

    At this point, define the map $\Phi: 2^\omega\to \mathbb{T}$ by 
    \begin{equation}\label{eq:definitionPhi}
    \forall A\subseteq \omega, \quad 
    \Phi(A):=\sum_{m \in A}\frac{1}{q_m}+\mathbb{Z}. 
    \end{equation}
    Note that each $\Phi(A)$ is well defined since the above series converges. In addition, $\Phi$ is continuous: indeed, the tails $\sum_{m\ge N, m \in A}\frac{1}{q_m}$ converge to $0$ uniformly in $A\subseteq \omega$, while the finite sums $\sum_{m< N, m \in A}\frac{1}{q_m}$ depend only on the first $N$ coordinates of $A$. 
    Finally, we claim that \eqref{eq:claimwadgefirstresult} holds, that is, for each $A\subseteq \omega$ we have $A \in \I$ if and only if $\Phi(A) \in \mathsf{H}_{\bm{a}}(\I)$. 

    To this end, fix $A \subseteq \omega$, and observe that 
    \begin{displaymath}
    \begin{split}
    a_n \Phi(A)=\frac{q_n}{2}\sum_{m \in A}\frac{1}{q_m}
    &=\sum_{\substack{m\in A\\ m<n}}\frac{q_n}{2q_m}
    +
    \frac{1}{2}\chi_A(n)
    +
    \sum_{\substack{m\in A\\ m>n}}\frac{q_n}{2q_m}\\
    &
    =
    \frac{1}{2}\chi_A(n)
    +
    \sum_{\substack{m\in A\\ m>n}}\frac{q_n}{2q_m}
    \,\,\,\text{ in }\mathbb{T}.
    \end{split}
    \end{displaymath}
    It follows that, if $n\notin A$, then $0\le a_n\Phi(A) \le \delta_n<\nicefrac{1}{8}$. In the opposite case, if $n\in A$, then $\|a_n\Phi(A)\|\ge \nicefrac{1}{2}-\delta_n>\nicefrac{3}{8}$. 

    To conclude, suppose that $A \in \I$ and pick $\varepsilon\in (0,\nicefrac{1}{4})$. Then there is $F_\varepsilon \in \mathrm{Fin}$ such that $\{n \in \omega: \|a_n\Phi(A)\|\ge \varepsilon\}\subseteq A\cup F_\varepsilon \in \I$, so that $\Phi(A) \in \mathsf{H}_{\bm{a}}(\I)$. 
    Vice versa, suppose that $A\in \I^+$. Then $\{n \in \omega: \|a_n\Phi(A)\|\ge \nicefrac{1}{4}\}\supseteq A \in \I^+$, hence $\Phi(A) \notin \mathsf{H}_{\bm{a}}(\I)$. This completes the proof. 
\end{proof}

\medskip

\begin{proof}
    [Proof of Corollary \ref{cor:pointclassessimple}]
    This is an immediate consequence of Theorem \ref{thm:wadge}. 
\end{proof}

\medskip

\begin{proof}
    [Proof of Corollary \ref{cor:borelrank}] 
    \ref{item:1borrank} Since $\I$ is $\mathbf{\Pi}^0_m$-complete for some $m\ge 3$, then $\I\in \mathbf{\Pi}^0_m(2^\omega)\setminus \mathbf{\Sigma}^0_m(2^\omega)$. 
    It follows 
    by Theorem \ref{thm:upperbound} 
    and Corollary \ref{cor:pointclassessimple} that there exists $\bm{a} \in \omega^\omega$ such that $\mathsf{H}_{\bm{a}}(\I)\in \mathbf{\Pi}^0_m(\mathbb{T})\setminus \mathbf{\Sigma}^0_m(\mathbb{T})$.     
    The completeness claim follows by \cite[Theorem 22.10 and Exercise 24.20]{MR1321597}.

    \medskip

    \ref{item:2borrank} It is a consequence of Theorem \ref{thm:upperbound}. 

    \medskip
    
    \ref{item:3borrank} We proceed as in item \ref{item:1borrank}. Let $\xi<\omega_1$ be a countable ordinal such that $\I \in \mathbf{\Pi}^0_{\xi}(2^\omega)\setminus \mathbf{\Pi}^0_{\alpha}(2^\omega)$ for all $\alpha<\xi$. 
    It follows 
    by Theorem \ref{thm:upperbound} 
    and Corollary \ref{cor:pointclassessimple} that there exists $\bm{a} \in \omega^\omega$ such that $\mathsf{H}_{\bm{a}}(\I) \in \mathbf{\Pi}^0_{\xi}(\mathbb{T})\setminus \mathbf{\Pi}^0_{\alpha}(\mathbb{T})$ for all $\alpha<\xi$. 
\end{proof}

\medskip

\begin{proof}
    [Proof of Proposition \ref{prop:completenesspropH}] 
    Pick $\bm{a}$ as in the proof of Theorem \ref{thm:wadge}. Fix also an  ideal $\I$ such that $\I\in \mathbf{\Gamma}(2^\omega)\setminus \mathbf{\Gamma}^\prime(2^\omega)$. Thanks to Theorem \ref{thm:wadge} we have $\mathsf{H}_{\bm{a}}(\I) \notin \mathbf{\Gamma}^\prime(\mathbb{T})$. Hence, it remains to prove that $\mathsf H_{\bm a}(\mathcal I)\in\mathbf{\Gamma}(\mathbb{T})$. 

    To this aim, for each $k\in \omega$, define the map 
    $
    F_k:\mathbb T\to2^\omega
    $ 
    by
    $$
\forall x \in \mathbb{T}, \quad     F_k(x):=
    \{n\in\omega:\|a_nx\|\geq 2^{-k}\}.
    $$
    Notice that $F_k$ is Baire class one: indeed, for each
    $n\in\omega$, the $n$-th coordinate map of $F_k$ is 
    $ 
    x\mapsto
    \chi_{K_n}(x),
    $  
    where $K_n:=\{y\in\mathbb T:\|a_ny\|\geq 2^{-k}\}$ is closed in $\mathbb{T}$ since the map $y\mapsto a_ny$ is continuous. 
    Hence each coordinate map is Baire class one, and therefore
    $F_k$ is Baire class one as well. 

    By the definition of $\I$-convergence, we obtain that 
    $$
    \mathsf{H}_{\bm{a}}(\I)=
    \{x \in \mathbb{T}: \forall k\in\omega, F_k(x) \in \I\}
    =\bigcap_{k \in \omega}F_k^{-1}[\I].
    $$
    The claim follows by the hypotheses that $\I \in \mathbf{\Gamma}(2^\omega)$ and that $\mathbf{\Gamma}(2^\omega)$ is closed under countable intersections and preimages of Baire class one maps. 
\end{proof}

\medskip

\begin{proof}
    [Proof of Corollary \ref{cor:Boreanalyticetc}]
    It is enough to set $\mathbf{\Gamma}=\mathbf{\Sigma}^1_m$ \textup{[}or $\mathbf{\Pi}^1_m$ or $\mathbf{\Delta}^1_m$, respectively\textup{]} and $\mathbf{\Gamma}^\prime=\{\emptyset\}$ in Proposition \ref{prop:completenesspropH} (note the second part of the proof of Proposition \ref{prop:completenesspropH} does not depend on the choice of $\bm{a}$).
\end{proof}

\medskip

\begin{proof}
     [Proof of Corollary \ref{cor:analyticcomplete}] 
    Recall by \cite[Theorem 26.4 and Exercise 26.5]{MR1321597} that, if $X$ is a Polish space and we assume $\mathbf{\Sigma}^1_1$-determinacy, then a set $S\subseteq X$ is analytic-complete if and only if $S \in \mathbf{\Sigma}^1_1(X)\setminus \mathbf{\Pi}^1_1(X)$. Using twice this equivalence, the claim follows by setting  $\mathbf{\Gamma}=\mathbf{\Sigma}^1_1$ and $\mathbf{\Gamma}^\prime=\mathbf{\Pi}^1_1$ in Proposition \ref{prop:completenesspropH}. 
\end{proof}

\medskip

Before we proceed to the proof of Theorem  \ref{thm:analyticPidealscharacterization}, we need the following auxiliary result. 
\begin{proposition}\label{prop:preimagePolishable}
Let \(f:X\to Y\) be a continuous homomorphism between two Polish groups $X,Y$. Let $K$ be a Polishable subgroup of $Y$. 
Then $f^{-1}[K]$ is a Polishable subgroup. 
\end{proposition}
\begin{proof}
Define \(L:=f^{-1}[K]\), and let \(\tau_K\) be a Polish group topology on \(K\) with the same Borel sets as \(K\) has as a subspace of \(Y\). 
First, observe that the inclusion map
\[
\iota:(K,\tau_K)\to Y
\]
is continuous. Indeed, it is a Borel homomorphism from a Polish group into a Polish group,
hence it is Baire measurable. Therefore, by
\cite[Theorem 9.10]{MR1321597}, it is continuous.

Now consider the map
$
\Phi:L\to X\times (K,\tau_K)$ by 
\begin{equation}\label{eq:defPhiIII}
\forall x \in L, \quad 
\Phi(x):=(x,f(x)).
\end{equation}

We claim that \(\Gamma:=\Phi[L]\) is closed in \(X\times (K,\tau_K)\). Indeed, suppose that 
$((x_m,k_m): m\in \omega)$ is a sequence with values in $\Gamma$ which is convergent to \(((x,k)\) in
\(X\times (K,\tau_K)\). 
Then \(x_m\to x\) in \(X\), and, since \(\iota\) is
continuous, \(k_m\to \iota(k)\) in \(Y\). On the other hand, $\iota(k_m)=f(x_m)$ for every \(m\), and the continuity of \(f\) gives \(f(x_m)\to f(x)\) in
\(Y\). Since \(Y\) is Hausdorff, \(\iota(k)=f(x)\). Hence \((x,k)\in\Gamma\). This proves that \(\Gamma\) is closed.

Since \(X\times (K,\tau_K)\) is a Polish group, it follows that \(\Gamma\) is a Polish subgroup.
Transport its topology to \(L\) through the bijection \(\Phi:L\to\Gamma\), and denote the resulting topology by \(\tau_L\). Then \((L,\tau_L)\) is a Polish group. 
To complete the proof, we need to compare the Borel structures. Since the first-coordinate map $\Gamma\to X$ defined by $(x,k)\mapsto x$ is continuous, 
the identity map $(L,\tau_L)\to X$ is continuous as well. 
Hence every Borel subset of \(L\) inherited from \(X\) is
\(\tau_L\)-Borel. 

Conversely, let \(U\subseteq K\) be a \(\tau_K\)-open set. Pick a Borel set \(B\subseteq Y\) such that $U=K\cap B$. Then 
$
\{x\in L:f(x)\in U\}
=
L\cap f^{-1}[B]
$ 
is Borel in the relative topology inherited from \(X\). Since the topology \(\tau_L\) is generated by the original \(X\)-coordinate together with
sets of this form, and since \((L,\tau_L)\) is second countable, every
\(\tau_L\)-open set is Borel in the relative topology. 
Hence every \(\tau_L\)-Borel set is
inherited-Borel. 
Therefore \(\tau_L\) is a Polish group topology on \(L=f^{-1}[K]\) with the
same Borel sets as the subspace \(L\subseteq X\), that is,  \(L\) is
Polishable.
\end{proof}

\smallskip

\begin{remark}\label{rmk:compatiblemetricpolishability}
With the notations above, let \(d_X\) be a compatible complete metric on \(X\), and let \(d_K\) be a
compatible complete metric on \((K,\tau_K)\). For each \(x,y\in L\), define 
\[
d_L(x,y):=d_X(x,y)+d_K(f(x),f(y)). 
\]
Then \(d_L\) is precisely the metric transported to \(L\) from the product
metric on \(X\times (K,\tau_K)\) through the map $\Phi$ defined in \eqref{eq:defPhiIII}.
\end{remark}

\medskip

\begin{proof}
[Proof of Theorem \ref{thm:analyticPidealscharacterization}]
\ref{item:1polishability} $\implies$ \ref{item:2polishability}. 
Let $\I$ be an analytic $P$-ideal on $\omega$. By the equivalence \ref{item:2solecki} $\Longleftrightarrow$ \ref{item:3solecki} in Theorem \ref{thm:solecki},   
there exists a finite lscsm 
$\varphi:\mathcal P(\omega)\to[0,\infty]$ such that
$\mathcal I=\mathrm{Exh}(\varphi)$. 
Replacing $\varphi$, if necessary, by the lscsm $\varphi^\prime$ defined by 
$\varphi^\prime(S):=\varphi(S)+\sum_{n\in S}2^{-n}$ for all $S\subseteq \omega$, we may suppose without loss of generality that
$0<\varphi(\{n\})<\infty$ for every $n\in\omega$ (in fact, $\mathrm{Exh}(\varphi)=\mathrm{Exh}(\varphi^\prime)$).

For each $\bm z=(z_n:n\in\omega)\in\mathbb T^\omega$ and $\varepsilon>0$, define 
$$
A_\varepsilon(\bm z):=\{n\in\omega:\|z_n\|\geq\varepsilon\},
$$
so that 
$
c_0(\mathcal I)=\{\bm z\in\mathbb T^\omega:\forall \varepsilon>0, A_\varepsilon(\bm{z}) \in \mathrm{Exh}(\varphi)\}.
$ 
For each $\bm z\in\mathbb T^\omega$, define also 
\begin{equation}\label{eq:definitionpvarphi}
p_\varphi(\bm z):=
\inf\{\varepsilon>0:\varphi(A_\varepsilon(\bm z))\leq\varepsilon\}.
\end{equation}
Observe that $p_\varphi(\bm z)<\infty$, since $A_\varepsilon(\bm z)=\emptyset$ whenever $\varepsilon>\nicefrac12$. 
Hence, for each $\bm{z} \in \mathbb{T}^\omega$ and $\varepsilon>0$, the set 
$\{\varepsilon>0:\varphi(A_\varepsilon(\bm z))\le \varepsilon\}$ 
has the form 
$(\alpha,\infty)$ or 
$[\alpha,\infty)$ for some $\alpha \in [0,\nicefrac{1}{2}]$. 
At this point, set 
$$
\forall \bm z,\bm w\in c_0(\mathcal I),\quad 
d_\varphi(\bm z,\bm w):=p_\varphi(\bm z-\bm w). 
$$
We are going to show that $d_\varphi$ is a (translation invariant) metric on $c_0(\mathcal I)$ which induces a Polish group topology having the same Borel sets as $c_0(\I)$ when considered as a topological subgroup of $\mathbb{T}^\omega$.

\smallskip

\textsc{Metric.} First, \(p_\varphi\) is subadditive. Let \(\bm z,\bm w\in\mathbb T^\omega\), and
pick \(\alpha>p_\varphi(\bm z)\) and \(\beta>p_\varphi(\bm w)\). Choose
\(\alpha_0<\alpha\) and \(\beta_0<\beta\) such that $\varphi(A_{\alpha_0}(\bm z))\le \alpha_0$ and $\varphi(A_{\beta_0}(\bm w))\le \beta_0$. 
Then 
$
A_{\alpha+\beta}(\bm z+\bm w)
\subseteq
A_{\alpha_0}(\bm z)\cup A_{\beta_0}(\bm w)$,
and therefore, since $\varphi$ is a lscsm, we obtain 
\[
\varphi(A_{\alpha+\beta}(\bm z+\bm w))
\le
\alpha_0+\beta_0
<
\alpha+\beta.
\]
Hence $p_\varphi(\bm z+\bm w)\le \alpha+\beta$. Letting \(\alpha\downarrow p_\varphi(\bm z)\) and
\(\beta\downarrow p_\varphi(\bm w)\), we obtain 
$p_\varphi(\bm z+\bm w)
\le
p_\varphi(\bm z)+p_\varphi(\bm w)$. 
Thus \(d_\varphi\) satisfies the triangle
inequality. Also, \(p_\varphi(-\bm z)=p_\varphi(\bm z)\), so \(d_\varphi\) is
symmetric.

\smallskip

Now, pick $\bm{z} \in c_0(\I)$ such that \(p_\varphi(\bm z)=0\). 
If \(\bm z\neq \bm{0}\), there exists \(n\in\omega\) such that \(z_n\neq 0\), and we can choose $\varepsilon \in (0,\min\{\|z_n\|,\varphi(\{n\})\})$. 
Since \(p_\varphi(\bm z)=0\), there is \(0<\tau<\varepsilon\) such that 
$
\varphi(A_\tau(\bm z))\le \tau$. 
But \(n\in A_\tau(\bm z)\), hence
\[
\varphi(\{n\})\le \varphi(A_\tau(\bm z))\le \tau<\varphi(\{n\}),
\]
which is impossible. Thus \(\bm z=\bm{0}\). Therefore \(d_\varphi\) is a metric on $c_0(\I)$.

\smallskip

\textsc{Completeness.} Let $(\bm z^m:m\in\omega)$ be a $d_\varphi$-Cauchy sequence in
$c_0(\mathcal I)$, with $\bm z^m=(z^m_n:n\in\omega)$ for all $m \in\omega$. Fix $n\in\omega$, $\eta>0$, and pick $\delta\in (0,\min\{\eta,\varphi(\{n\})\})$. 
For all sufficiently large $m,\ell\in \omega$, we have 
$p_\varphi(\bm z^m-\bm z^\ell)<\delta$. Hence there is $\tau\in (0,\delta)$ such that
$\varphi(A_\tau(\bm z^m-\bm z^\ell))\leq\tau$. If
$\|z^m_n-z^\ell_n\|\geq\eta$, then $n\in A_\tau(\bm z^m-\bm z^\ell)$, which is impossible.
Thus $(z^m_n:m\in\omega)$ is Cauchy in $\mathbb T$. Let
$\bm z=(z_n:n\in\omega)$ be the coordinatewise limit.

We claim that $\bm z^m\to\bm z$ in $d_\varphi$. We will use the following consequence of lower
semicontinuity of $\varphi$: if $\varphi(B_\ell)\leq c$ for all $\ell\in \omega$ and
$B\subseteq\liminf_\ell B_\ell$, then $\varphi(B)\leq c$. Indeed, for every $r\in \omega$, the
finite set $B\cap r$ is eventually contained in $B_\ell$, hence
$\varphi(B\cap r)\leq c$, so that $\varphi(B)=\lim_r \varphi(B\cap r)\leq c$.

Fix $\varepsilon>0$ and choose $\eta \in (0,\nicefrac{\varepsilon}{2})$. 
Since the sequence $(\bm z^m:m\in\omega)$ is
$d_\varphi$-Cauchy, there is $M$ such that $p_\varphi(\bm z^m-\bm z^\ell)<\eta$ for all
$m,\ell\geq M$. Then $\varphi(A_\eta(\bm z^m-\bm z^\ell))\leq\eta$ for all
$m,\ell\geq M$. Fixing $m\geq M$, we have 
$\|z^m_n-z^\ell_n\|
\geq
\|z^m_n-z_n\|-\|z^\ell_n-z_n\|
>
\eta$ 
whenever \(n\in A_{2\eta}(\bm z^m-\bm z)\) and $\ell$ is sufficiently large, so that 
\[
A_{2\eta}(\bm z^m-\bm z)
\subseteq
\liminf_{\ell\to\infty}A_\eta(\bm z^m-\bm z^\ell).
\]
Hence, using the previous observation, $\varphi(A_{2\eta}(\bm z^m-\bm z))\leq\eta$. 
It follows that 
$p_\varphi(\bm z^m-\bm z)\leq2\eta<\varepsilon$ for all $m\geq M$.

Finally, $\bm z\in c_0(\I)$. To this end, fix $\varepsilon,\theta>0$. Pick $m\in \omega$ such
that $p_\varphi(\bm z-\bm z^m)<\min\{\nicefrac{\varepsilon}{2},\theta\}$. Then for some
$\delta \in (0,\min\{\nicefrac{\varepsilon}{2},\theta\})$ we have
$\varphi(A_\delta(\bm z-\bm z^m))<\theta$. Since
\[
A_\varepsilon(\bm z)
\subseteq
A_{\varepsilon/2}(\bm z^m)\cup A_\delta(\bm z-\bm z^m),
\]
and $\bm z^m\in c_0(\I)$, we get
$$
\lim_{r\to \infty}\varphi(A_\varepsilon(\bm z)\setminus r)\leq
\lim_{r\to \infty}\varphi(A_{\varepsilon/2}(\bm z^m)\setminus r)
+\varphi(A_\delta(\bm z-\bm z^m))
\le \theta.
$$
Since $\theta>0$ is
arbitrary, it follows that $A_\varepsilon(\bm z)\in\mathrm{Exh}(\varphi)$ for every $\varepsilon>0$.

\smallskip

\textsc{Separability.} Let $Q\subseteq\mathbb T$ be countable dense with $0\in Q$, and let
$C_0\subseteq \mathbb{T}^\omega$ be the set of finitely supported sequences with values in $Q$. Then $C_0$ is
countable and contained in $c_0(\I)$. Fix $\bm z\in c_0(\I)$ and
$\varepsilon>0$. Choose $\delta \in (0,\varepsilon)$ and an integer $M\in \omega$ such that
$\varphi(A_\delta(\bm z)\setminus M)<\delta$. Pick $q_0,\ldots,q_{M-1}\in Q$ with
$\|z_n-q_n\|<\delta$ for all $n<M$, and set $q_n=0$ for all $n\geq M$. Then
\[
A_\delta(\bm z-\bm q)\subseteq A_\delta(\bm z)\setminus M,
\]
which implies that $d_\varphi(\bm z,\bm q)<\varepsilon$. Thus $(c_0(\I), d_\varphi)$ is Polish.

\smallskip

\textsc{Continuity of addition.} 
Since \(d_\varphi\) is translation-invariant and $p_\varphi$ is subadditive, it follows that 
\[
d_\varphi(\bm x+\bm y,\bm x'+\bm y')
=
p_\varphi((\bm x-\bm x')+(\bm y-\bm y'))
\leq
p_\varphi(\bm x-\bm x')+p_\varphi(\bm y-\bm y')
\]
for all 
\(\bm x,\bm y,\bm x',\bm y'\in c_0(\mathcal I)\). Hence addition is continuous.

\smallskip

\textsc{Borel structure.} It remains to compare the Borel structures. First, the identity map 
$
(c_0(\mathcal I),d_\varphi)\to c_0(\mathcal I)\subseteq\mathbb T^\omega
$ 
is continuous. Indeed, if \(\bm z^m\to\bm z\) in \(d_\varphi\), then the
argument used above shows that \(\bm z^m\to\bm z\) coordinatewise. Hence every
Borel subset of \(c_0(\mathcal I)\) inherited from \(\mathbb T^\omega\) is
\(d_\varphi\)-Borel.

Conversely, we show that every \(d_\varphi\)-open ball is Borel in the relative topology inherited from \(\mathbb T^\omega\). 
To this aim, fix
\(\bm y\in c_0(\mathcal I)\) and \(t>0\). For each \(q>0\), define
\[
B_q^{\bm y}(\bm z):=
\{n\in\omega:\|z_n-y_n\|\ge q\}.
\]
The map
$
\bm z\mapsto B_q^{\bm y}(\bm z)
$ 
from \(\mathbb T^\omega\) to \(2^\omega\) is Borel, because every coordinate
condition 
$
\|z_n-y_n\|\ge q
$ 
is closed. In addition, since \(\varphi\) is lower semicontinuous, the map 
$
B\mapsto\varphi(B)
$ 
from \(2^\omega\) to \([0,\infty]\) is Borel. Therefore, for every rational
\(q>0\), the set
\[
\{\bm z\in\mathbb T^\omega:\varphi(B_q^{\bm y}(\bm z))<q\}
\]
is Borel.

Now, 
$
p_\varphi(\bm z-\bm y)<t
$ 
if and only if there exists a rational \(q\) with \(q \in (0,t)\) such that 
$
\varphi(B_q^{\bm y}(\bm z))<q.
$ 
Indeed, the \textsc{if} part is immediate. Conversely, if
\(p_\varphi(\bm z-\bm y)<t\), choose \(\alpha<t\) such that 
$
\varphi(A_\alpha(\bm z-\bm y))\le \alpha,
$ 
and then choose a rational \(q \in (\alpha,t)\), so that 
\[
\varphi(A_q(\bm z-\bm y))\le \varphi(A_\alpha(\bm z-\bm y))\le \alpha<q.
\]

By the above observations, it follows that
\[
\{\bm z\in c_0(\mathcal I):d_\varphi(\bm z,\bm y)<t\}
\]
is Borel in the relative topology inherited from \(\mathbb T^\omega\). To conclude, since \((c_0(\mathcal I),d_\varphi)\) is separable, every \(d_\varphi\)-open set
is a countable union of \(d_\varphi\)-open balls. Hence every
\(d_\varphi\)-Borel set is also Borel in the relative topology inherited from \(\mathbb T^\omega\).

Therefore \(c_0(\mathcal I)\) is Polishable.

\medskip

\ref{item:2polishability} $\implies$ \ref{item:3polishability}. Pick $\bm{a} \in \omega^\omega$ and let $f_{\bm{a}}: \mathbb{T}\to \mathbb{T}^\omega$ be the continuous homomorphism defined by 
$$
\forall x \in \mathbb{T}, \quad 
f_{\bm{a}}(x):=(a_nx: n \in \omega). 
$$
Since $c_0(\I)$ is a Polishable subgroup of the Polish group $\mathbb{T}^\omega$ and since $\mathsf{H}_{\bm{a}}(\I)=f_{\bm{a}}^{-1}[c_0(\I)]$, it follows by Proposition \ref{prop:preimagePolishable}
that $\mathsf{H}_{\bm{a}}(\I)$ is a Polishable subgroup of $\mathbb{T}$. 

According to Remark \ref{rmk:compatiblemetricpolishability}, given the lscsm $\varphi$ and the sequence $\bm{a} \in \omega^\omega$, the subgroup $\mathsf{H}_{\bm{a}}(\I)$ becomes Polishable under the metric $d_{\varphi,\bm{a}}$ defined by 
\begin{equation}\label{eq:metricpolishability}
\begin{split}
d_{\varphi,\bm{a}}(x,y)&:=\|x-y\|+d_\varphi(f_{\bm{a}}(x),f_{\bm{a}}(y))\\
&=\|x-y\|+p_\varphi((a_n(x-y): n \in \omega))\\
&=\|x-y\|+\inf\{\varepsilon>0: \varphi(\{n\in \omega: \|a_n(x-y)\|\ge \varepsilon\}) \le \varepsilon\}.
\end{split}
\end{equation}
for each $x,y \in \mathsf{H}_{\bm{a}}(\I)$.

\medskip

\ref{item:3polishability} $\implies$ \ref{item:1polishability}. 
Let $\I$ be an ideal on $\omega$ which is \emph{not} an analytic $P$-ideal. We need to show that there exists $\bm{a} \in \omega^\omega$ such that $\mathsf{H}_{\bm{a}}(\I)$ is not Polishable. We divide the rest of the proof in two cases. 

\medskip

\textsc{Case 1: $\I$ is analytic.} 
Let $\I$ be an analytic ideal on $\omega$ which is not a $P$-ideal. 
By the equivalence \ref{item:2solecki} $\Longleftrightarrow$ \ref{item:4solecki} in Theorem \ref{thm:solecki}, we have that $\mathrm{Fin}\otimes \emptyset \le_{\mathrm{RB}} \mathcal{I}$. Now, let $\mathcal J$ be an ideal on $\omega$ which is isomorphic to  $\mathrm{Fin}\otimes \emptyset$ and observe, proceeding verbatim as in \cite[Theorem 2.6]{FKLT26}, that 
$$
\mathscr{H}(\J)=\{\mathsf{H}_{\bm{b}}(\mathrm{Fin}\otimes \emptyset): \bm{b}=(b_{i,j}: (i,j) \in \omega^2) \in \omega^{\omega^2}\}.
$$
Since $\J\le_{\mathrm{RB}} \mathcal{I}$, it will be enough to show that there exists a sequence $\bm{b}\in \omega^{\omega^2}$ such that $\mathsf{H}_{\bm{b}}(\mathrm{Fin}\otimes \emptyset)$
%
%
%
is a Borel subgroup of $\mathbb T$ which is not Polishable.

To this aim, set for convenience 
$M_n:=n+1$ and $r_n:=(10M_n)^2$ for all $n\in\omega$, and define recursively $q_0:=1$ and $q_{n+1}:=q_nr_n$ for all $n\in\omega$. 
For each $i,n\in\omega$, put $L_{i,n}:=\left\lfloor M_n/(i+1)\right\rfloor$ and $R_i:=\{mq_n:n\in\omega,\ 1\leq m\leq L_{i,n}\}$. 
Since each $R_i$ is infinite, it is possible to pick an enumeration $R_i=\{b_{i,j}:j\in\omega\}$ for every $i\in\omega$. Hence, we regard
$$
\bm b=(b_{i,j}:(i,j)\in\omega^2)
$$
as a sequence indexed by
$\omega^2$.

At this point, for each $x\in\mathbb T$ and $i \in \omega$, define
\[
s_i(x):=\sup_{j\in\omega}\|b_{i,j}x\|
=
\sup_{n\in\omega}\sup_{1\leq m\leq L_{i,n}}\|mq_nx\|.
\]
By the definition of $\mathrm{Fin}\otimes \emptyset$, we have
\begin{equation}\label{eq:J0si}
\mathsf H_{\bm b}(\mathrm{Fin}\otimes \emptyset)
=
\{x\in\mathbb T: \lim_{i\to \infty} s_i(x)=0\}.
\end{equation}

Define also
\[
\forall x \in \mathbb{T}, \quad p_0(x):=\sup_{n\in\omega}M_n\|q_nx\|,
\]
where the value $+\infty$ is allowed, and set $H:=\{x \in \mathbb{T}: p_0(x)<\infty\}$.

We shall use the following elementary
fact. 
\begin{claim}\label{claim:multipleestimate}
Pick $y\in\mathbb T$ and an integer $L\geq1$. 
If $\max_{1\leq m\leq L}\|my\|<\nicefrac14$, then
\[
L\|y\|\leq \max_{1\leq m\leq L}\|my\|.
\]
\end{claim}
\begin{proof}
Suppose for the sake of contradiction that
$L\|y\|\geq\nicefrac14$, and let $m\leq L$ be the least positive integer
such that $m\|y\|\geq\nicefrac14$. 
Since the hypothesis gives
$\|y\|<\nicefrac14$, we have $m>1$. By minimality,
also $(m-1)\|y\|<\nicefrac14$, and consequently $m\|y\|<\nicefrac12$. 
Thus $\|my\|=m\|y\|\geq\nicefrac14$, a contradiction. Therefore
$L\|y\|<\nicefrac14$, and so $\|Ly\|=L\|y\|$. 
\end{proof}

\begin{claim}\label{claim:equalityHHfin}
    $H=\mathsf{H}_{\bm{b}}(\mathrm{Fin}\otimes \emptyset)$. 
\end{claim}
\begin{proof}
Fix $x \in \mathbb{T}$. First, suppose that $x \in H$, so that $C:=p_0(x)<\infty$. Then
$\|q_nx\|\leq C/M_n$ for every $n\in\omega$. Hence, if
$1\leq m\leq L_{i,n}$, then
\[
\|mq_nx\|
\leq m\|q_nx\|
\leq \frac{M_n}{i+1}\cdot\frac{C}{M_n}
=
\frac{C}{i+1}.
\]
Therefore $s_i(x)\leq C/(i+1)$ for all $i\in\omega$. 
Hence $x \in \mathsf{H}_{\bm{b}}(\mathrm{Fin}\otimes \emptyset)$ by \eqref{eq:J0si}.

Conversely, suppose that $\lim_i s_i(x)=0$. 
Pick $i\in \omega$ such that
$s_i(x)<\nicefrac14$, and $n\in\omega$ with $M_n\geq 2i+2$. 
Since $L_{i,n}\geq M_n/(2i+2)$, it follows by Claim \ref{claim:multipleestimate} (applied to
$y=q_nx$ and $L=L_{i,n}$) that 
\[
M_n\|q_nx\|
\leq
\frac{M_n}{L_{i,n}}s_i(x)
\leq
(2i+2)s_i(x).
\]
For the finitely many $n$ such that $M_n<2i+2$, we have
$M_n\|q_nx\|\leq M_n/2\le i+1$. Hence $p_0(x)<\infty$, i.e., $x \in H$.
\end{proof}

In particular, it follows by Claim \ref{claim:equalityHHfin} that $H$ is a subgroup of $\mathbb{T}$. For every $k\in\omega$, define
$B_k:=\{x\in\mathbb T:p_0(x)\leq k\}$, so that 
\[
B_k
=
\bigcap_{n\in\omega}
\{x\in\mathbb T:M_n\|q_nx\|\leq k\}.
\]
Since each $B_k$ is closed in $\mathbb T$, and $H=\bigcup_kB_k$, then $H$ is
an $F_\sigma$ subgroup of $\mathbb T$.

It remains to prove that $H$ is not Polishable. Suppose, towards a
contradiction, that $\tau$ is a Polish group topology on $H$ having the same
Borel sets as the ones inherited from $\mathbb T$. Since $H=\bigcup_kB_k$
and each $B_k$ is $\tau$-Borel, by Baire's category theorem \cite[Theorem 8.4]{MR1321597} there is $k\in\omega$ such that $B_k$ is
nonmeager in $(H,\tau)$. By Pettis' theorem \cite[Theorem 9.9]{MR1321597}, $B_k-B_k$ contains a
$\tau$-neighbourhood of $0$. Since $B_k-B_k\subseteq B_{2k}$, there is a
$\tau$-neighbourhood $U$ of $0$ such that $U\subseteq B_{2k}$. Thus, let $D$ be a countable $\tau$-dense subset of $H$, and observe that
\begin{equation}\label{eq:coverbyD}
H\subseteq D+B_{2k}.
\end{equation}
To contradict the Polishability property, it will be enough to construct an uncountable subset of $H$ which cannot be covered 
by $D+B_{2k}$, which will contradict Inequality \eqref{eq:coverbyD}. 

To this aim, choose $A\geq1$ such that
$\nicefrac{9A}{10}>4k$ and, for each $n\ge 4A$, define 
\[
c_n:=\left\lfloor\frac{Ar_n}{M_n}\right\rfloor.
\]
Then
\[
M_n\frac{c_n}{r_n}\geq A-\frac{1}{100M_n}
\quad \text{ and }\quad 
c_n\leq \frac{Ar_n}{M_n}\leq \frac{r_n}{4}.
\]

For each $\eta\in 2^{\{n\in\omega:n\geq 4A\}}$, define
\[
x_\eta:=
\sum_{n=4A}^{\infty}\eta(n)\frac{c_n}{q_{n+1}}+\mathbb Z \in \mathbb{T}.
\]
The series is absolutely convergent. We first show that $x_\eta\in H$. Fix
$n\geq 4A$. Since $q_{m+1}$ divides $q_n$ whenever $m<n$, we have
\[
q_nx_\eta
=
\eta(n)\frac{c_n}{r_n}
+
\sum_{m>n}\eta(m)c_m\frac{q_n}{q_{m+1}}
\,\,\text{ in }\mathbb{T}
\]
For each $n \in \omega$, using $q_n/q_{m+1}=1/(r_n\cdots r_m)$ and $r_\ell \ge 2$ for all $\ell$, we get
$$
\sum_{m>n}c_m\frac{q_n}{q_{m+1}}
\leq
\sum_{m>n}\frac{A}{M_mr_n r_{n+1}\cdots r_{m-1}}
\le 
\frac{A}{r_n}\sum_{m>n}\frac{1}{2^{m-n-1}}
\le \frac{2A}{r_n}.
$$
Therefore
\[
\forall n\ge 4A, \quad 
M_n\|q_nx_\eta\|
\leq
A+\frac{2AM_n}{r_n}
=
A+\frac{A}{50M_n}
<2A.
\]
For the finitely many $n<4A$, we have $M_n\|q_nx_\eta\|\leq M_n/2 \le 2A$. Hence
$p_0(x_\eta)<\infty$, and so $x_\eta\in H$.

Next, we show that distinct points of the type $x_\eta$ are far apart from each other. Let
$\eta\neq\theta$, and let $n\geq 4A$ be the least index such that
$\eta(n)\neq\theta(n)$. Then 
\[
q_n(x_\eta-x_\theta)
=
\pm\frac{c_n}{r_n}
+
\sum_{m>n}(\eta(m)-\theta(m))c_m\frac{q_n}{q_{m+1}}
\,\,\text{ in }\mathbb{T}.
\]
Again, the tail has absolute value at most $2A/r_n$. Moreover,
$c_n/r_n\leq A/M_n\leq\nicefrac14$ and $2A/r_n\leq\nicefrac14$, so the above
representative has absolute value $<\nicefrac12$. Therefore, using also the
lower estimate on $c_n/r_n$, we obtain
\[
M_n\|q_n(x_\eta-x_\theta)\|
\geq
A-\frac{1}{100M_n}-\frac{A}{50M_n}
>
\frac{9A}{10}.
\]
Hence $p_0(x_\eta-x_\theta)>\nicefrac{9A}{10}>4k$.

To conclude, the set $P:=\{x_\eta:\eta\in 2^{\{n\in\omega:n\geq 4A\}}\}$ is an uncountable subset of $H$, and
$p_0(x-y)>4k$ for all distinct $x,y\in P$. On the other hand, each translate
$d+B_{2k}$ contains at most one point of $P$. 
Indeed, if
$x,y\in P\cap(d+B_{2k})$, then $x=d+u$ and $y=d+v$ for some
$u,v\in B_{2k}$, and therefore
\[
p_0(x-y)=p_0(u-v)\leq p_0(u)+p_0(v)\leq4k.
\]
Thus $x=y$. Since $D$ is countable, then $P\cap (D+B_{2k})$ is at most countable. This contradicts Inequality \eqref{eq:coverbyD}. Therefore $H$ is not Polishable.

    \medskip

\textsc{Case 2: $\I$ is not analytic.} Suppose now that $\I$ is not analytic. In particular, $\I$ is not Borel. It follows by Corollary \ref{cor:pointclassessimple} that $\mathsf{H}_{\bm{a}}(\I)$ is not Borel for some $\bm{a} \in \omega^\omega$, hence not Polishable. 
\end{proof}

\medskip

\begin{proof}
[Proof of Corollary \ref{cor:birocounterexample}]
    Let $K\subseteq \mathbb{T}$ be a Kronecker set, that is, a nonempty compact set such that for every $\varepsilon>0$ and every continuous map $f: K\to \mathbb{T}$ there exists $n \in \mathbb{Z}$ such that $\|f(x)-nx\|<\varepsilon$ for all $x \in K$, and suppose that $K$ is uncountable. 
    Using a non-trivial result of J. Aaronson and N. Nadkarni \cite{MR907232}, 
    B\'ir\'o proved in \cite{MR2388789} that the subgroup $H:=\langle K\rangle$ generated by $K$ is not Polishable. 

    In addition, $H$ is $F_\sigma$: for a direct proof, observe that $C:=K\cup (-K)\cup \{0\}$ is compact. Hence also $m\star C:=C+\cdots+C$, where $C$ is repeated $m$ times, is compact (being a continuous image of a compact set). 
    Therefore $H=\bigcup_{m\ge 1}(m\star C)$ is $F_\sigma$. 

    The conclusion follows by Theorem \ref{thm:analyticPidealscharacterization}. 
\end{proof} 

\medskip

\begin{proof}
[Proof of Proposition \ref{prop:FsigmadeltacompletenotPolishable}] 
We proceed as in the proof of the implication \ref{item:3polishability} $\implies$ \ref{item:1polishability} of Theorem \ref{thm:analyticPidealscharacterization}. 
Let $\J$ be a copy on $\omega$ of the $F_{\sigma\delta}$-complete ideal $\emptyset\otimes\mathrm{Fin}$.\footnote{Indeed, by \cite[Exercise~23.2]{MR1321597}, the set
$
P:=\{x\in\omega^\omega:|\{n\in\omega:x(n)=i\}|<\infty\text{ for every }i\in\omega\}
$
is $F_{\sigma\delta}$-complete. The continuous map
$
x\mapsto\{(i,n)\in\omega^2:x(n)=i\}
$
reduces $P$ to $\emptyset\otimes\mathrm{Fin}$.} Set $M_n:=n+1$ and $r_n:=(10M_n)^2$ for all $n\in\omega$, and define recursively $q_0:=2$, $q_{n+1}:=q_nr_n$ and $a_n:=q_n/2$ for all $n \in \omega$.  
In particular, $q_n/(2q_m)\in\omega$ whenever $m<n$. As in Case~1 of the proof of Theorem \ref{thm:analyticPidealscharacterization}, define $H_0:=\{x\in\mathbb{T}:p(x)<\infty\}$ and $B_k:=\{x\in\mathbb{T}:p(x)\leq k\}$ for all $k \in \omega$, where $p(x):=\sup_{n\in\omega}M_n\|q_nx\|$. 
The function $p$ is subadditive, each $B_k$ is closed, and $H_0=\bigcup_kB_k$. Hence $H_0$ is an $F_\sigma$ subgroup. Define
$$
H:=H_0\cap\mathsf{H}_{\bm a}(\J).
$$
Since $\J$ is $F_{\sigma\delta}$, Theorem \ref{thm:upperbound} implies that $H$ is an $F_{\sigma\delta}$ subgroup.

\begin{claim}\label{claim:2222222A}
    $H$ is $F_{\sigma\delta}$-complete.
\end{claim} 
\begin{proof}
    For every $n\in\omega$, we have 
$$
\delta_n:=\sum_{m>n}\frac{q_n}{2q_m}\leq\frac{1}{r_n}.
$$
In particular, $\delta_n<\nicefrac18$ and $\lim_n\delta_n=0$. Define the continuous map $\Phi:2^\omega\to\mathbb{T}$ by  
$
\Phi(A):=\sum_{m\in A}1/q_m+\mathbb{Z}
$ 
as in \eqref{eq:definitionPhi}. As in the proof of Theorem \ref{thm:wadge}, $\|a_n\Phi(A)\|\leq\delta_n$ if $n\notin A$, while
$\|a_n\Phi(A)\|>\nicefrac38$ if $n\in A$. Consequently, 
$\Phi(A)\in\mathsf{H}_{\bm a}(\J)$ if and only if $A \in \J$. 
In addition, 
$$
p(\Phi(A))
\leq
\sup_{n\in\omega}M_n\sum_{m>n}\frac{q_n}{q_m}
\leq
\sup_{n\in\omega}\frac{2M_n}{r_n}
<1.
$$
Hence $\Phi(A)\in H_0$ for every $A\subseteq\omega$, so 
$
\Phi^{-1}[H]=\J.
$ 
Therefore $H$ is $F_{\sigma\delta}$-complete.
\end{proof}

\begin{claim}\label{claim:2222222B}
    $H$ is not Polishable.
\end{claim} 
\begin{proof}
The Baire category and Pettis argument used in Case~1 of the proof of Theorem \ref{thm:analyticPidealscharacterization} shows that it is enough to prove that, for every $k\in\omega$, there exists an uncountable set $P\subseteq H$ such that 
$
p(x-y)>4k
$ 
for all distinct $x,y\in P$. Indeed, otherwise a Polish group topology on $H$ would yield a countable set $D\subseteq H$ such that $H\subseteq D+B_{2k}$, whereas each translate of $B_{2k}$ could contain at most one point of $P$.

Fix $k \in \omega$ and choose $L \in \omega$ such that $\nicefrac{9L}{10}>4k$. Now, define 
$$
x_\eta:=
\sum_{n=4L}^{\infty}\eta(n)\frac{c_n}{q_{n+1}}+\mathbb{Z} 
\qquad \text{ for all }\eta\in2^{\{n\in\omega:n\geq4L\}},
$$
where $c_n:=2\left\lfloor\frac{Lr_n}{2M_n}\right\rfloor$ for each $n\ge 4L$. The estimates from Case~1 of the proof of Theorem \ref{thm:analyticPidealscharacterization} (with $L$ in place of $A$), give
$$
\sum_{m>n}c_m\frac{q_n}{q_{m+1}}
\leq\frac{2L}{r_n}
\qquad\text{and}\qquad
p(x_\eta)\leq2L.
$$
Moreover, since every $c_n$ is even, all the terms preceding the $n$th one become integers after multiplication by $a_n$, and hence
$$
\|a_nx_\eta\|
\leq
\frac{c_n}{2r_n}+\frac{L}{r_n}
\leq
\frac{L}{2M_n}+\frac{L}{r_n},
$$
which converges to $0$ as $n\to \infty$. 
Therefore
$$
x_\eta\in H_0\cap\mathsf{H}_{\bm a}(\mathrm{Fin})\subseteq H
\qquad \text{ for all }\eta\in2^{\{n\in\omega:n\geq4L\}}.
$$
Finally, if $\eta\neq \theta$ and $n:=\min \{k\geq4L: \eta(k)\neq \theta(k)\}$, the same estimates yield
$$
p(x_\eta-x_\theta)
\geq
M_n\|q_n(x_\eta-x_\theta)\|
\geq
L-\frac{1+L}{50M_n}
>
\frac{9L}{10}
>
4k.
$$
Thus 
$
P:=\{x_\eta:\eta\in2^{\{n\in\omega:n\geq4L\}}\}
$ 
has the required properties.
\end{proof}

Therefore $H$ satisfies the claimed properties. The last assertion follows from Theorem \ref{thm:analyticPidealscharacterization}.
\end{proof}

\medskip

\begin{proof}
    [Proof of Corollary \ref{corollary:lupini}]
    Since $H$ is an infinite proper subgroup of $\mathbb{T}$, it is neither open nor closed. Thanks to Theorem \ref{thm:analyticPidealscharacterization}, $H$ is Polishable. In addition, since analytic $P$-ideals are $F_{\sigma\delta}$, it follows by Theorem \ref{thm:upperbound} that $H$ is $F_{\sigma\delta}$. We conclude by \cite[Theorem 1.1]{MR5045103} that $H$ is either $F_{\sigma}$-complete or $F_{\sigma\delta}$-complete or 
    $\mathsf{D}(\mathbf{\Pi}^0_2)$-complete. 
\end{proof}

\medskip

\begin{proof}
[Proof of Proposition \ref{prop:locallyquasiconvex}]
\ref{item:1countablygenerated} $\implies$ \ref{item:2countablygenerated}.  
First, suppose that $\I=\mathrm{Fin}$ or $\I$ is isomorphic to $\mathrm{Fin}\oplus \mathcal{P}(\omega)$. Then $\I$ is Rudin--Keisler equivalent to $\mathrm{Fin}$, hence by \cite[Theorem 2.6(i)]{FKLT26} we get $\mathscr{H}(\I)=\mathscr{H}(\mathrm{Fin})$. In such case, the implication follows by \cite[Corollary 1]{MR2729343}. 

Hence, let us suppose hereafter that $\I$ is isomorphic to $\J:=\emptyset \otimes \mathrm{Fin}$. Again by \cite[Theorem 2.6(i)]{FKLT26}, it is enough to prove the claim for $\J$. To this aim, fix a sequence $\bm b=(b_{i,j}:(i,j)\in\omega^2)\in \omega^{\omega^2}$ and observe that 
$$
\mathsf{H}_{\bm b}(\J)=\left\{x \in \mathbb{T}: \lim_{j\to \infty} b_{i,j}x=0 \text{ for all }i \in \omega\right\}. 
$$
Now, define the lscsm $\varphi: \mathcal{P}(\omega^2)\to [0,\infty]$ 
by $\varphi(A):=\sup_i 2^{-i}\min\{1,|A_i|\}$ 
for every $A\subseteq \omega^2$, so that $\J=\{A\subseteq \omega^2: \inf_{F \in [\omega^2]^{<\omega}}\varphi(A\setminus F)=0\}$, 
cf. \cite[Example 1.2.3(b)]{MR1711328}. 

Endow $c_0(\mathrm{Fin})$ with the supremum metric
$d_\infty(\bm u,\bm v):=\sup_j\|u_j-v_j\|$, and define the group $G:=\mathbb T\times \left(c_0(\mathrm{Fin})\right)^\omega$. Thus, endow $G$ with its product topology. In this way, both $c_0(\mathrm{Fin})$ and $G$ are Polish groups. Define the homomorphism $T:\mathsf{H}_{\bm{b}}(\J)\to G$ by
\[
\forall x \in \mathsf{H}_{\bm{b}}(\J),\quad 
T(x):=
\left(x,\bigl((b_{i,j}x:j\in\omega): i \in \omega\bigr)\right).
\]
Of course, $T$ is a well-defined injective map. 
\begin{claim}\label{claim:Polishgroup}
    $T[\mathsf{H}_{\bm{b}}(\J)]$ is closed in $G$.
\end{claim}
\begin{proof}
Suppose that $\lim_n T(x_n)=(x,(\bm u^i)_{i\in\omega})$ in $G$, where
$\bm u^i=(u^i_j:j\in\omega)\in c_0(\mathrm{Fin})$ for all $i \in \omega$. 
Then $\lim_n x_n= x$ in
$\mathbb T$, and so $\lim_n b_{i,j}x_n= b_{i,j}x$ for all $i,j\in \omega$. On the other hand, convergence in the $i$th copy of $c_0(\mathrm{Fin})$ gives
$\lim_n b_{i,j}x_n= u^i_j$. Hence $u^i_j=b_{i,j}x$ for all $i,j \in \omega$. Since
$\bm u^i\in c_0(\mathrm{Fin})$, we have $\lim_j b_{i,j}x=0$ for
every $i\in \omega$, and therefore $x\in \mathsf{H}_{\bm{b}}(\J)$. Thus
$(x,(\bm u^i)_{i\in \omega})=T(x) \in T[\mathsf{H}_{\bm{b}}(\J)]$.
\end{proof}

It follows by Claim \ref{claim:Polishgroup} that the topology $\tau$ on $\mathsf{H}_{\bm{b}}(\J)$ pulled back from $G$ through $T$ is Polish. By the uniqueness of such topology and using Theorem \ref{thm:analyticPidealscharacterization}, we obtain that $\tau$ is generated by the metric $d$ defined by 
$$
d(x,y):=\|x-y\|+\inf\{\varepsilon>0: \varphi(\{(i,j)\in \omega^2: \|b_{i,j}(x-y)\|\ge \varepsilon\})\le \varepsilon\}
$$
for all $x,y \in \mathsf{H}_{\bm{b}}(\J)$. 

It remains to prove local quasi-convexity. To this aim, fix a neighbourhood $U$ of $0$ in $\mathsf{H}_{\bm{b}}(\J)$. Since $\tau$ is metrized by $d$, there are $r\in\omega$ and positive integers
$M_\star,M_0,\ldots,M_r$ such that the set
\[
V:=
\left\{
x\in \mathsf{H}_{\bm{b}}(\J):
\|x\|\leq\frac{1}{4M_\star}
\text{ and }
\sup_{j\in \omega}\|b_{i,j}x\|\leq\frac{1}{4M_i}
\text{ for all }i\leq r
\right\}
\]
is a neighborhood of $0$ contained in $U$. 

To complete the proof, it is enough to show that, for each $y \in \mathsf{H}_{\bm{b}}(\J)\setminus V$ there exists $q \in \omega$ such that $\sup_{x \in V}\|qx\|\le \nicefrac{1}{4}<\|qy\|$. Fix $y \in \mathsf{H}_{\bm{b}}(\J)\setminus V$ and consider the following two cases:
\begin{enumerate}
    \item  Suppose that
$\|y\|>\nicefrac{1}{4M_\star}$. Thanks to 
Claim \ref{claim:multipleestimate}, there exists $1\leq q\leq M_\star$ such
that $\|qy\|>\nicefrac14$. On the other hand, for every $x\in V$,
$\|qx\|\leq q\|x\|\leq\nicefrac14$. 
    \item  Suppose that $\|y\|\leq \nicefrac{1}{4M_\star}$. Since $y\notin V$, there are
$i\leq r$ and $j\in\omega$ such that
$\|b_{i,j}y\|>1/(4M_i)$. Again by Claim \ref{claim:multipleestimate}, there
exists $1\leq \ell\leq M_i$ such that
$\|\ell b_{i,j}y\|>\nicefrac14$. Put $q:=|\ell b_{i,j}|$. Then, for every $x\in V$,
\[
\|qx\|=\|\ell b_{i,j}x\|\leq \ell\|b_{i,j}x\|\leq\nicefrac14,
\]
while $\|qy\|=\|\ell b_{i,j}y\|>\nicefrac14$.
\end{enumerate}
Therefore $(\mathsf{H}_{\bm{b}}(\J), \tau)$ is a locally quasi-convex Polishable subgroup. 

\medskip

\ref{item:2countablygenerated} $\implies$ \ref{item:1countablygenerated}. Since every $\I$-characterized subgroup is, in particular, a Polishable subgroup, it follows by Theorem \ref{thm:analyticPidealscharacterization} that $\I$ is an analytic $P$-ideal. Hence by 
the equivalence \ref{item:2solecki} $\Longleftrightarrow$ \ref{item:3solecki} in Theorem \ref{thm:solecki} there exists a finite lscsm $\varphi: \mathcal{P}(\omega)\to [0,\infty]$ such that $\I=\{S\subseteq \omega: \lim_n \varphi(S\setminus n)=0\}$. 

Now, we prove our claim by contrapositive: suppose $\I$ is an analytic $P$-ideal such that $\I \neq \mathrm{Fin}$ and that $\I$ is isomorphic neither to $\mathrm{Fin}\oplus \mathcal{P}(\omega)$ nor to $\emptyset \otimes \mathrm{Fin}$. We need to show that there exists $\bm{a} \in \omega^\omega$ such that $\mathsf{H}_{\bm{a}}(\I)$ is not locally quasi-convex Polishable. 

Thanks to \cite[Corollary 1.2.11]{MR1711328}, it is possible to pick $S \in \I^+$ such that every infinite $A\subseteq S$ contains some infinite $B\subseteq A$ with $B \in \I$. Since $\varphi$ is a lscsm, we easily obtain 
\begin{equation}\label{eq:limitlongS}
\lim_{n\to \infty, n \in S} \varphi(\{n\})=0.
\end{equation}
Now, fix a real $\delta \in (0,\min\{\nicefrac{1}{4}, \lim_n \varphi(S\setminus n)\})$. Again by the fact that $\varphi$ is a lscsm, for each $n \in \omega$ it is possible to fix a nonempty finite set $E_n \subseteq S$ such that $\min E_n>n$ and $\varphi(E_n)>\delta$. Accordingly, set 
$$
\forall n \in\omega, \quad 
L_n:=\max\{16, |E_0|,\ldots,|E_n|\}.
$$

Now, we define recursively a sequence $\bm a\in \omega^\omega$ as it follows. Set $a_0:=1$ and, if $a_0,\ldots,a_n$ have been defined for some $n \in \omega$, let $a_{n+1}$ be a positive multiple of $2a_n$ such that $a_{n+1}>2^na_nL_{n+1}$. Observe that this implies that
\begin{equation}\label{eq:growth-conditions}
\frac{L_{n+1}a_m}{2a_{n+1}}<2^{-(n+1)}
\quad\text{and}\quad
\frac{a_m}{2a_{n+1}}<2^{-(n+1)-4}.
\end{equation}
for all $m,n \in \omega$ with $m\le n$. It follows by \eqref{eq:growth-conditions} that $\frac{L_na_m}{2a_n}<2^{-n}$ whenever $m<n$, and that $\sum_{j>m}\frac{a_m}{2a_j}<\nicefrac{1}{8}$. Taking into account that $a_m/a_n$ is an even integer for all $n<m$, it follows that 
$$
x_n:=\frac{1}{2a_n}\in \mathsf{H}_{\bm{a}}(\mathrm{Fin})\subseteq \mathsf{H}_{\bm{a}}(\I)
$$
for all $n \in \omega$: indeed, the sequence $(a_mx_n: m \in \omega)\in \mathbb{T}^\omega$ is finitely supported. Set also $d:=d_{\varphi,\bm{a}}$ as in \eqref{eq:metricpolishability}, so that $(\mathsf{H}_{\bm{a}}(\I), d)$ is a Polishable subgroup by Theorem \ref{thm:analyticPidealscharacterization}. We shall prove that $(\mathsf{H}_{\bm{a}}(\I), d)$ is not locally quasi-convex. 
\begin{claim}\label{claim:uniformdlekjh}
We have 
$$
\lim_{n\to \infty, n\in S}d(\ell x_n,0)=0,
$$
uniformly for $1\le \ell \le L_n$. 
\end{claim}
\begin{proof}
Observe that $\|\ell x_n\|\leq L_n/(2a_n)$ for each $n,\ell \in \omega$ with $1\le \ell\le L_n$. 
Moreover, 
since $\frac{L_na_m}{2a_n}<2^{-n}$ whenever $m<n$, we get 
$\|a_m\ell x_n\|<2^{-n}$ for every $m<n$,
whereas $a_m\ell x_n=0$ for all $m>n$, and $a_n\ell x_n=\ell/2$. This implies that, if $\gamma_n:=\max\{2^{-n},\varphi(\{n\})\}$, then 
$
\left\{m \in \omega: \|a_m\ell x_n\|\ge \gamma_n\right\}\subseteq \{n\}.
$ 
It follows that 
$$
d(\ell x_n,0)\le \frac{L_n}{2a_n}+\gamma_n
$$
for all $n,\ell \in \omega$ with $1\le \ell\le L_n$. We obtain the conclusion by the limit in \eqref{eq:limitlongS}. 
\end{proof}

Finally, set $U:=\{x \in \mathsf{H}_{\bm{a}}(\I): d(x,0)<\delta\}$ and suppose for the sake of contradiction that there is a quasi-convex neighborhood $V$ of $0$ in $\mathsf{H}_{\bm{a}}(\I)$ which is contained in $U$. 
Since $V$ is a $d$-neighborhood of $0$, it is possible to pick $\eta>0$ such that 
$$
\{x\in \mathsf{H}_{\bm a}(\I):d(x,0)<\eta\}\subseteq V.
$$
By Claim \ref{claim:uniformdlekjh}, we may choose a sufficiently large $N\in \omega$ such 
that, for every $n\in S\setminus N$ and every $1\leq\ell\leq L_n$, we have
$d(\ell x_n,0)<\eta$. Since $E_N\subseteq S\setminus N$ and
$|E_N|\leq L_n$ for every $n\in E_N$, it follows that
$\ell x_n \in V$ for all $n,\ell \in \omega$ with $n \in E_N$ and
$1\le \ell\le |E_N|$.

Define 
\[
y:=\sum_{n\in E_N}x_n \in \mathsf{H}_{\bm a}(\I).
\]
\begin{claim}\label{claimyinV}
$y\in V$.
\end{claim}
\begin{proof}
Suppose for the sake of contradiction that $y\notin V$. 
On the one hand, since $V$ is quasi-convex, there exists a continuous character $\phi \in \widehat{\mathsf{H}_{\bm{a}}(\I)}$ such that
\[
\sup_{z\in V}\|\phi(z)\|\leq \frac14<\|\phi(y)\|.
\]
For every $n\in E_N$ and every
$1\leq\ell\leq |E_N|$, we have $\ell x_n\in V$, and therefore 
$
\|\ell \phi(x_n)\|=\|\phi(\ell x_n)\|\leq \nicefrac14.
$ 
This implies that
\[
\forall n \in E_N, \qquad 
\|\phi(x_n)\|
\leq \frac{1}{4|E_N|}.
\]
Indeed, otherwise, letting $m\leq |E_N|$ be the least positive integer
such that $m\|\phi(x_n)\|>\nicefrac14$, 
we would have $\nicefrac14
<
m\|\phi(x_n)\|
\leq
\nicefrac12$, 
and hence 
$\|m\phi(x_n)\|
=
m\|\phi(x_n)\|
>
\nicefrac14$, 
a contradiction. 
Consequently,
\[
\|\phi(y)\|
\leq
\sum_{n\in E_N}\|\phi(x_n)\|
\leq
\sum_{n\in E_N}\frac{1}{4|E_N|}=\frac{1}{4}.
\]
This provides the desired contradiction. 
\end{proof}

\begin{claim}\label{claim:ynotinU}
$y\notin U$.
\end{claim}
\begin{proof}
    Fix
$m\in E_N$ and observe that 
$$
0\leq
\sum_{\substack{n\in E_N\\ n>m}}\frac{a_m}{2a_n}
\leq
\sum_{n>m}\frac{a_m}{2a_n}
<
\frac18.
$$
Since $a_m/a_n$ is an even integer for all $n<m$, it follows that 
$$
a_my=
\frac{1}{2}+\sum_{\substack{n\in E_N\\ n>m}}\frac{a_m}{2a_n},
$$
so that $\|a_my\|>\nicefrac{1}{4}$. Since $\delta<\nicefrac{1}{4}$ it follows that $E_N\subseteq \{m\in \omega: \|a_my\|\ge \varepsilon\}$ for all $\varepsilon \in (0,\delta)$. Hence by monotonicity 
\[
\varphi(\{m\in\omega:\|a_my\|\geq\varepsilon\})
\geq
\varphi(E_N)
>
\delta
>
\varepsilon.
\]
Therefore $d(y,0)\geq\delta$, i.e., $y\notin U$, 
\end{proof}

The final contradiction comes putting together Claim \ref{claimyinV}, Claim \ref{claim:ynotinU}, and the inclusion $V\subseteq U$. Therefore $(\mathsf{H}_{\bm{a}}(\I),d)$ is not locally quasi-convex. 
\end{proof}



\medskip

\begin{proof}
[Proof of Proposition \ref{prop:locallyquasiconvex2}] 
    \ref{item:1countablygenerated222} $\implies$ \ref{item:2countablygenerated222}. By the proof of the implication \ref{item:1countablygenerated} $\implies$ \ref{item:2countablygenerated} of Proposition \ref{prop:locallyquasiconvex}, we have $\mathscr{H}(\mathrm{Fin})=\mathscr{H}(\mathrm{Fin}\oplus \mathcal{P}(\omega))$. Hence, pick a characterized subgroup $H=\mathsf H_{\bm a}(\mathrm{Fin})$. Thanks to Proposition \ref{prop:locallyquasiconvex}, $H$ is locally quasi-convex Polishable. 
    Its finer Polish group topology is generated by the translation-invariant metric
$$
\rho_{\bm a}(x,y):=
\max\left\{\|x-y\|,\,\sup_{n\in\omega}\|a_n(x-y)\|\right\},
$$
see e.g. \cite{MR2388789}. Now, set $U:=B_{\rho_{\bm a}}(0,\nicefrac18)$. We claim that 
$$
U_m=B_{\rho_{\bm a}}\left(0,\frac{1}{8m}\right)
\qquad\text{for every }m\geq1.
$$
Indeed, the inclusion from right to left follows from
$\rho_{\bm a}(kx,0)\leq k\rho_{\bm a}(x,0)$. Conversely, if $x\in U_m$, then
$\rho_{\bm a}(kx,0)<\nicefrac18$ for every $1\leq k\leq m$. Applying Claim
\ref{claim:multipleestimate} to $x$ and to $a_nx$, for every $n\in\omega$, we obtain
$$
m\|x\|<\frac18
\quad\text{and}\quad
m\|a_nx\|<\frac18.
$$
Hence $\rho_{\bm a}(x,0)<1/(8m)$, proving the claim. Therefore $(U_m:m\geq1)$ is a neighborhood basis at $0$, that is, the Polish group topology of $H$ is UFSS. 

    \medskip

    \ref{item:2countablygenerated222} $\implies$ \ref{item:1countablygenerated222}. Thanks to Proposition \ref{prop:locallyquasiconvex}, $\I=\mathrm{Fin}$ or $\I$ is isomorphic to $\mathrm{Fin}\oplus \mathcal{P}(\omega)$ or $\emptyset \otimes \mathrm{Fin}$. Hereafter, suppose for the sake of contradiction that $\I=\emptyset \otimes \mathrm{Fin}$. To complete the proof, it is enough to construct an $\emptyset\otimes\mathrm{Fin}$-characterized subgroup of $\mathbb{T}$ whose Polish topology is not UFSS. 

    To this aim, let $(S_i:i\in\omega)$ be a partition of $\omega$ into infinite sets, and write
$S_i=\{s_i(j):j\in\omega\}$ increasingly for each $i \in \omega$. Define
$$
D_k:=2^{2^k},
\qquad
b_{i,j}:=\frac{D_{s_i(j)}}{2},
\quad \text{ and }\quad 
H:=\mathsf H_{\bm b}(\emptyset\otimes\mathrm{Fin}).
$$
Thus $H\in\mathscr H(\emptyset\otimes\mathrm{Fin})$. 
As in the proof of
Proposition \ref{prop:locallyquasiconvex}, its finer Polish group topology $\tau$ has a
neighborhood basis at $0$ consisting of
$$
W(F,\delta):=
\left\{
x\in H:
\|x\|<\delta
\text{ and }
\sup_{j\in\omega}\|b_{i,j}x\|<\delta
\text{ for all }i\in F
\right\},
$$
where $F\subseteq\omega$ is finite and $\delta>0$.

We show that $\tau$ is not UFSS. Fix a $\tau$-neighborhood
$U$ of $0$, and choose $F\in \mathrm{Fin}$ and $\delta>0$ such that $W(F,\delta)\subseteq U$.
Pick $i\notin F$, and set
$$
V:=\left\{x\in H:\sup_{j\in\omega}\|b_{i,j}x\|<\frac14\right\}.
$$
Fix $m\geq1$. Choose a large integer $k\in S_i$ such that 
$
m/D_k<\delta
$
and
$
mD_{k-1}/(2D_k)<\delta,
$
and put $x_k:=1/D_k$. Then $x_k\in H$ and, in addition, 
$$
\sup_{j\in \omega}\|b_{i,j}x_k\|=\frac12,
\quad\text{ and }\quad 
\sup_{j\in \omega}\|b_{r,j}x_k\|
\leq\frac{D_{k-1}}{2D_k}
\quad\text{for every }r\in F.
$$
Consequently, for every $1\leq q\leq m$,
$$
\|qx_k\|<\delta
\quad\text{and}\quad
\sup_{j\in \omega}\|b_{r,j}qx_k\|<\delta
\quad\text{for every }r\in F.
$$
Hence $qx_k\in U$ for every $1\leq q\leq m$, while $x_k\notin V$. Therefore no
neighborhood $U$ of $0$ in $(H,\tau)$ can have the property that
$
\{x:kx\in U\text{ for all }1\leq k\leq m\},
$
$m\geq1$, form a neighborhood basis at $0$. Therefore $\tau$ is not UFSS. 
\end{proof}

\medskip

\begin{proof}
[Proof of Corollary \ref{cor:strictFinemptysetFin}]
Since $\emptyset\otimes\mathrm{Fin}$ is meager, the inclusion $\mathscr H(\mathrm{Fin})
    \subseteq
    \mathscr H(\emptyset\otimes\mathrm{Fin})$ follows by \cite[Corollary 2.7]{FKLT26}. In addition, they are distinct by Proposition \ref{prop:locallyquasiconvex2}.

Since $\mathcal{Z}$ is an analytic $P$-ideal which is not $F_\sigma$, it follows by a result of Solecki \cite{MR1416872} that 
    $
    \emptyset\otimes\mathrm{Fin}\le_{\mathrm{RB}}\mathcal{Z}.
    $ 
    Hence $\mathscr{H}(\emptyset\otimes\mathrm{Fin})\subseteq\mathscr{H}(\mathcal{Z})$ by \cite[Theorem 2.6(i)]{FKLT26}. Finally, it follows from Proposition \ref{prop:locallyquasiconvex} that the latter inclusion is strict. 
\end{proof}

\medskip

\begin{proof}
[Proof of Theorem \ref{thm:Fsigmadelcomplete}]
Let $(\varphi_k: k \in \omega)$ be a sequence of lscsms with finite disjoint supports such that $\I=\{S\subseteq \omega: \lim_n \varphi(S\setminus n)=0\}$, where $\varphi:=\sup_k\varphi_k$. 
Fix $\bm{a} \in \omega^\omega$ and write for the sake of simplicity $H:=\mathsf{H}_{\bm{a}}(\I)$. Suppose that $H\neq \mathbb{T}$ and let $d:=d_{\varphi, \bm{a}}$ be the metric on $H$ defined in \eqref{eq:metricpolishability}, that is, 
$$
d(x,y)=\|x-y\|+p_\varphi(\bm{a}(x-y))
$$
for all $x,y \in H$, where $\bm{a}(x-y)=(a_n(x-y): n \in \omega)$.  
Taking into account from the definition of $p_\varphi$ in  \eqref{eq:definitionpvarphi} that the sets $\{\varepsilon>0: \varphi(A_\varepsilon(\bm{z}))\le \varepsilon\}$ have the form $(\alpha,\infty)$ or $[\alpha,\infty)$ for some $\alpha \in [0,\nicefrac{1}{2}]$, it easily follows that 
$p_\varphi=\sup_k p_{\varphi_k}$.  
In particular, we have 
$$
\forall x\in \mathbb{T}, \quad 
p_\varphi(\bm{a}x)=\sup_{k \in \omega}p_{\varphi_k}(\bm{a}x)
$$
Recall also from Theorem \ref{thm:analyticPidealscharacterization} that $(H,d)$ is a Polish group. The remaining part of this proof is divided into two main steps: in the first one, we suppose that $H$ is not $F_\sigma$ and we show that $H$ is $F_{\sigma\delta}$-complete; in the second one, we suppose that $H$ is $F_\sigma$ and we show that $H$ is countable. 

\bigskip

\textsc{Part one: $H$ is not $F_\sigma$.} In this part, suppose that $H$ is not $F_\sigma$. 
\begin{claim}\label{claim:nonloccompactGeneralizedDensity}
For every $r>0$, there is $\eta\in(0,r)$ such that, for every nonempty finite
$F\subseteq\omega$ and every $\delta>0$, there exists $x\in H$ such that 
\[
\|x\|<\delta,
\qquad
\max_{k \in F}p_{\varphi_k}(\bm{a}x)<\delta, 
\quad \text{ and }\quad
\eta \le p_\varphi(\bm{a}x)\le r.
\]
\end{claim}
\begin{proof}
Suppose for the sake of contradiction that the claim fails. Then there is $r>0$ such that, for every $\eta \in (0,r)$, there are
a finite nonempty $F\subseteq\omega$ and a real $\delta>0$ with the following property:
if $x\in H$, $p_\varphi(\bm{a}x)\le r$, $\|x\|<\delta$, and
$p_{\varphi_k}(\bm{a}x)<\delta$ for all $k\in F$, then $p_\varphi(\bm{a}x)<\eta$. 

We show that
\[
B:=\left\{x\in H:p_\varphi(\bm{a}x)\leq \nicefrac{r}{2}\right\}
\]
is compact in $(H,d)$. In fact, let $(x_j:j\in\omega)$ be a sequence in $B$. Passing
to a subsequence, if needed, we may suppose without loss of generality that $x_j\to x$ in $\mathbb T$. We prove
that this subsequence is $d$-Cauchy. To this end, fix $\eta>0$, and choose $F,\delta$ as
above, replacing $\eta$ by a smaller positive number if necessary. Since
$(x_j)$ is Cauchy in $\mathbb T$ and the submeasures $(\varphi_k: k \in F)$ have finite supports, there exists $n_0 \in \omega$ such that 
\[
\forall i,j \ge n_0, \qquad 
\|x_i-x_j\|<\delta
\quad\text{and}\quad
p_{\varphi_k}(\bm{a}(x_i-x_j))
<\delta \text{ for all } k\in F.
\]
Moreover, $p_\varphi(\bm{a}(x_i-x_j))\leq p_\varphi(\bm{a}x_i)+p_\varphi(\bm{a}x_j)\leq r$. Hence
$p_\varphi(\bm{a}(x_i-x_j))<\eta$. Thus $(x_j:j\in\omega)$ is $d$-Cauchy, and so by completeness it
converges in $(H,d)$. Since $B$ is closed in $(H,d)$, it
follows that $B$ is compact. 

Of course, the set $B$ contains a $d$-open neighborhood $U$ of $0$. 
Since $(H,d)$ is Polish, it is possible to pick a countable dense \(D\subseteq H\). Then
\[
H=\bigcup_{q\in D}(q+U)\subseteq \bigcup_{q\in D}(q+B).
\]
Since each \(q+B\) is compact, it follows that \((H,d)\) is \(\sigma\)-compact. 
The identity map
$(H,d)\to\mathbb T$ is continuous, hence $H$ is a countable union of compact
subsets of $\mathbb T$. Therefore $H$ is $F_\sigma$, which is the desired contradiction. 
\end{proof}

At this point, define
\[
P:=\{\alpha\in\omega^\omega:\lim_{m\to\infty}\alpha(m)=+\infty\}
\]
and recall that $P$ is $F_{\sigma\delta}$-complete by \cite[Exercise 23.2]{MR1321597}. For every $k\in\omega$, set
$r_k:=2^{-k-6}$. By Claim \ref{claim:nonloccompactGeneralizedDensity},
choose $\eta_k\in(0,r_k)$ such that the conclusion of the claim holds with
$r=r_k$, and put for notational convenience $\theta_k:=\eta_k/8$.

Fix an enumeration $((m_s,k_s):s\in\omega)$ of $\omega^2$. Choose
$\lambda_0>0$ such that 
$
\lambda_0<
\min\left\{
\frac{1}{2^{m_0+k_0+10}},
\theta_{k_0}
\right\}.
$ 
Recursively, for every $s\geq1$, choose $\lambda_s>0$ such that
$$
\lambda_s<
\min\left\{
\frac{\lambda_{s-1}}{2},
\frac{1}{2^{m_s+k_s+s+10}},
\theta_{k_s},
\frac{1}{2^{s+4}}\min_{q<s}\theta_{k_q}
\right\}.
$$
In this way, $(\lambda_s:s\in\omega)$ is strictly decreasing and, for every
$q\in\omega$, we have
$$
\sum_{s>q}\lambda_s
<
\theta_{k_q}\sum_{s>q}\frac{1}{2^{s+4}}
<
\frac{\theta_{k_q}}{8}.
$$

\smallskip

Next, we construct, by induction on $s\in \omega$, points $x_s\in H$, and natural numbers 
$L_s<t_s<M_s$ such that the intervals $I_s:=\omega \cap (L_s,M_s)$ are pairwise disjoint and increasing and, in addition, 
$$
p_\varphi(\bm{a}x_s)\le r_{k_s}, \quad 
p_{\varphi_{t_s}}(\bm{a}x_s)>4\theta_{k_s}, 
\quad \text{ and }\quad 
\sup_{k \notin I_s}p_{\varphi_k}(\bm{a}x_s)\le \lambda_s.
$$ 
Indeed, pick $s \in \omega$ and suppose that the values $x_q,L_q,t_q,M_q$ have been chosen for all $q<s$, with $M_{-1}:=0$. Choose
$L_s>M_{s-1}$ such that, in addition, 
$\sum_{q<s}p_{\varphi_j}(\bm{a}x_q)<\theta_{k_s}$ for all $j\ge L_s$. 
Apply Claim \ref{claim:nonloccompactGeneralizedDensity} with
$r=r_{k_s}$, $F=\{0,\ldots,L_s\}$, and $\delta=\lambda_s$. We get
$x_s\in H$ such that
$\|x_s\|<\lambda_s$,
$p_{\varphi_k}(\bm a x_s)<\lambda_s$ for all $k\leq L_s$,
and $\eta_{k_s}\leq p_\varphi(\bm{a}x_s)\leq r_{k_s}$. 
Since $\eta_{k_s}=8\theta_{k_s}$ and $\lambda_s<\theta_{k_s}$, there exists
$t_s>L_s$ such that 
$p_{\varphi_{t_s}}(\bm a x_s)>4\theta_{k_s}$. 
Finally, since $x_s\in H$, choose $M_s>t_s$ such that 
$
p_{\varphi_k}(\bm a x_s)<\lambda_s
\text{ for all }k\geq M_s
$ (which is possible since the lscsms $(\varphi_k)$ have finite disjoint supports). 
This completes the induction.

\begin{claim}\label{claim:wadgePH}
$P\le_{\mathrm{W}} H$. 
\end{claim}
\begin{proof}
    For each $m,k\in\omega$, let $x_{m,k}:=x_s$, where $s$ is the unique integer such
that $(m_s,k_s)=(m,k)$. Define also the map $f:\omega^\omega\to\mathbb T$ by 
\[
\forall \alpha \in \omega^\omega, \qquad 
f(\alpha):=\sum_{m\in \omega}x_{m,\alpha(m)}.
\]
Observe that the series converges uniformly in $\alpha$. Indeed,
$\|x_s\|<\lambda_s<2^{-m_s-k_s-s-10}$ for every $s\in \omega$, and the selected series has uniformly summable tails. Hence $f$ is continuous. To conclude the proof of the claim, it will be enough to show that 
\begin{equation}\label{eq:reductionP}
f^{-1}[H]=P. 
\end{equation}

First, suppose that $\alpha\in P$, and fix $\varepsilon>0$. Choose
$K\in\omega$ such that $r_k<\varepsilon/4$ for all $k\geq K$. Since
$\lim_m \alpha(m)=+\infty$, only finitely many selected pairs
$(m,\alpha(m))$ have second coordinate smaller than $K$. Choose $s_0\in\omega$ sufficiently large such that all such selected pairs occur before stage $s_0$, and also
$
\sum_{s\geq s_0}\lambda_s<\nicefrac{\varepsilon}{4}.
$ 
Since there are only finitely many selected $x_s$'s with $s<s_0$, there is
$N\in\omega$ such that
\[
\forall k\ge N, \qquad 
\sum_{\substack{s<s_0\\ x_s\text{ is selected by }\alpha}}
p_{\varphi_k}(\bm a x_s)
<
\frac{\varepsilon}{4}.
\]
Fix $k\geq N$. Among the intervals $I_s$, at most one contains $k$. If no
selected interval $I_s$ with $s\geq s_0$ contains $k$, then all selected
$x_s$'s with $s\geq s_0$ contribute at most $\sum_{s\geq s_0}\lambda_s$. If
one selected interval $I_s$ with $s\geq s_0$ contains $k$, then
$k_s\geq K$, and therefore
\[
p_{\varphi_k}(\bm a x_s)\leq p_{\varphi}(\bm{a}x_s)\leq r_{k_s}<\frac{\varepsilon}{4};
\]
all other selected $x_q$'s with $q\geq s_0$ contribute at most
$\sum_{q\geq s_0}\lambda_q$. Hence, in both cases, 
$
p_{\varphi_k}(\bm a f(\alpha))<\varepsilon
$ 
for all sufficiently large $k$. Therefore $f(\alpha)\in H$.

Conversely, suppose that $\alpha\notin P$. Then there are $K\in\omega$ and
infinitely many $m\in\omega$ such that $\alpha(m)\leq K$. Define 
$ 
\theta_\star:=\min\{\theta_0,\ldots,\theta_K\}>0.
$ 
Let $s$ range over the infinitely many stages such that
$(m_s,k_s)=(m,\alpha(m))$ for some $m\in\omega$ with
$\alpha(m)\leq K$. Observe that these indexes $s$ tend to infinity and that 
$
p_{\varphi_{t_s}}(\bm a x_s)>4\theta_{k_s}.
$ 
The selected terms $x_q$ with $q<s$ contribute less than
$\theta_{k_s}$ by the choice of $L_s$. Since the intervals $I_q$ are
increasing, every selected term $x_q$ with $q>s$ satisfies
$t_s\notin I_q$, and therefore these later terms contribute at most 
$
\sum_{q>s}\lambda_q<\nicefrac{\theta_{k_s}}8.
$ 
Hence, using the triangle inequality for
$p_{\varphi_{t_s}}$, we obtain
\[
p_{\varphi_{t_s}}(\bm a f(\alpha))
>
4\theta_{k_s}-\theta_{k_s}-\frac{\theta_{k_s}}8
>
2\theta_{k_s}
\geq
2\theta_\star.
\]
Since $\lim_s t_s=\infty$, it follows that
$f(\alpha)\notin H$. 
\end{proof}

Putting together Claim \ref{claim:wadgePH} and the fact that $H$ is $F_{\sigma\delta}$ by Corollary \ref{cor:borelrank}\ref{item:2borrank} (or also by \cite[Theorem 2.5(i)]{FKLT26}), we conclude that $H$ is $F_{\sigma\delta}$-complete.

\bigskip

\textsc{Part two: $H$ is $F_\sigma$.} In this second part, suppose that there exists a sequence $(F_m: m \in \omega)$ of compact subsets of $\mathbb{T}$ such that 
$
H=\bigcup_{m\in\omega}F_m.
$ 
Since the identity map \((H,d)\to\mathbb T\) is continuous, each \(F_m\) is closed in \((H,d)\). 
Since $(H,d)$ is Polish, it follows by the Baire category theorem that some $F_m$ has nonempty interior, so that there are \(h\in H\),
and \(R>0\) such that
\[
h+B_d(0,R)\subseteq F_m,
\]
where $B_d(0,R):=\{x \in H: d(x,0)<R\}$. 
Thus 
$
B_d(0,R)\subseteq F_m-h.
$ 
Since \(F_m-h\) is closed in \(\mathbb T\) and contained in \(H\), we get
\begin{equation}\label{eq:Tclosure-ball-H}
\overline{B_d(0,R)}^{\,\mathbb T}\subseteq H.
\end{equation}

\begin{claim}\label{claim:locallycompact}
\((H,d)\) is locally compact.
\end{claim}
\begin{proof}
Pick \(r \in (0,\nicefrac{R}{4})\).
It is enough to show that 
\[
\overline B_d(0,r):=\{x\in H:d(x,0)\leq r\}
\]
is compact. Of course, $\overline B_d(0,r)\subseteq \overline{B_d(0,R)}^{\,\mathbb T}$. Since \((H,d)\) is Polishable, hence complete, it is sufficient to prove that \(\overline B_d(0,r)\) is totally bounded.

Suppose hereafter, for the sake of contradiction, that \(\overline B_d(0,r)\) is not totally bounded. 
Then there are \(\eta>0\) and a sequence \((u_j:j\in\omega)\) with values in \(\overline B_d(0,r)\) such that $d(u_i,u_j)\geq\eta$ for all distinct $i,j \in \omega$. 
Passing to a subsequence if necessary, we may suppose without loss of generality that \((u_j: j \in \omega)\) is convergent in \(\mathbb T\) to some $u \in \mathbb{T}$. 
Since \(r<R\), it follows by \eqref{eq:Tclosure-ball-H} that \(u\in H\).

\smallskip

We show that \(d(u,0)\leq r\). Indeed, for every \(k\in\omega\), the
pseudometric \(p_{\varphi_k}\) depends only on finitely many coordinates.
Since \(u_j\to u\) in \(\mathbb T\), we have
\[
\forall k \in \omega, \quad \lim_{j\to \infty} p_{\varphi_k}(\bm a u_j)=p_{\varphi_k}(\bm a u). 
\]
Moreover, for each $j \in \omega$ we have \(d(u_j,0)\leq r\), that is, 
$
\|u_j\|+p_\varphi(\bm a u_j)\leq r. 
$ 
Taking the limit in the inequality
\(p_{\varphi_k}(\bm a u_j)\leq r-\|u_j\|\) and recalling that $\varphi_k$ has finite support, we obtain 
$p_{\varphi_k}(\bm a u)\leq r-\|u\|$ 
for all $k \in \omega$.
Since \(p_\varphi=\sup_k p_{\varphi_k}\), this implies that 
$p_\varphi(\bm a u)\leq r-\|u\|$, 
and therefore \(d(u,0)\leq r\).

\smallskip

At this point, define 
$$
\forall j \in \omega, \qquad z_j:=u_j-u.
$$
Then $(z_j: j \in \omega)$ is a sequence in $H$ which is convergent to $0$ in $\mathbb{T}$, and for all $j \in \omega$, we have $d(z_j,0)\leq d(u_j,0)+d(u,0)\leq 2r$. Passing to a tail, we may suppose without loss of generality that
$d(z_j,0)\geq \eta/2$ and $\|z_j\|\le \eta/4$ for all $j \in \omega$. 
Hence
$$
\forall j \in \omega, \qquad 
p_\varphi(\bm a z_j)\ge \eta/4.
$$
Set \(\varepsilon:=\eta/5\). For every \(j\in\omega\), choose
\(q_j\in\omega\) such that 
$
p_{\varphi_{q_j}}(\bm a z_j)>\varepsilon .
$ 
Passing to a subsequence, we may suppose that $\lim_j q_j=\infty$. 
Indeed, for
each fixed \(Q\in\omega\), since \(z_j\to0\) in \(\mathbb T\) and
\(\varphi_0,\ldots,\varphi_Q\) have finite supports, we have
$
\lim_j \max_{q\leq Q}p_{\varphi_q}(\bm a z_j)= 0.
$ 

\smallskip

Next, choose positive real numbers \((\xi_k:k\geq1)\) such that
\[
\xi_\infty:=\sum_{k=1}^{\infty}\xi_k
<
\min\left\{\frac{R-2r}{2},\frac{\varepsilon}{2}\right\}.
\]
We inductively choose a subsequence \((y_k:k\geq1)\) of \((z_j:j\in\omega)\)
and integers \(m_k<M_k\), with \(M_0:=0\), such that
\[
p_{\varphi_{m_k}}(\bm a y_k)>\varepsilon, 
\quad
\|y_k\|<\xi_k,
\quad \text{ and }\quad 
p_{\varphi_q}(\bm a y_k)<\xi_k
\]
for all integers $k\ge 1$ and $q \in \omega\setminus (M_{k-1},M_k)$. 
Indeed, suppose that \(M_{k-1}\) has been chosen. Since \(\lim_j z_j=0\) in
\(\mathbb T\), \(\lim_j q_j=\infty\), and
\(\varphi_0,\ldots,\varphi_{M_{k-1}}\) have finite supports, we can choose a sufficiently large 
\(j\) such that \(q_j>M_{k-1}\), $\|z_j\|<\xi_k$, and $p_{\varphi_q}(\bm a z_j)<\xi_k$ for all $q\leq M_{k-1}$. 
Put \(y_k:=z_j\) and \(m_k:=q_j\). Since \(y_k\in H\), we have
\(p_{\varphi_q}(\bm a y_k)\to0\), again by the fact that the lscsms $(\varphi_k)$ have finite disjoint supports. Hence there is \(M_k>m_k\) such that
$p_{\varphi_q}(\bm a y_k)<\xi_k$ for all $q\ge M_k$. This completes the induction.

\smallskip

Since $\sum_k\|y_k\|\leq\xi_\infty<\infty$, it is possible to define 
\[
x:=\sum_{k=1}^{\infty}y_k, 
\]
where is convergence is meant in $\mathbb{T}$. 
For each \(L\geq1\), write  \(s_L:=\sum_{k=1}^{L}y_k\), so that $x=\lim_L s_L$. We show that 
\[
\forall L\ge 1, \quad 
s_L\in B_d(0,R).
\]
First, observe similarly that $\|s_L\|\leq \sum_{k=1}^{L}\|y_k\|\leq\xi_\infty$. 
Fix \(q\in\omega\). 
If \(q>M_L\), then
\[
p_{\varphi_q}(\bm a s_L)
\leq
\sum_{k=1}^{L}p_{\varphi_q}(\bm a y_k)
\leq
\xi_\infty.
\]
Otherwise, let \(\ell\leq L\) be the least integer such that \(q\leq M_\ell\).
If \(k<\ell\), then \(q>M_k\), while if \(k>\ell\), then
\(q\leq M_\ell\leq M_{k-1}\). Hence all terms except possibly the
\(\ell\)-th one contribute at most \(\sum_{k\neq\ell}\xi_k\). The remaining
term contributes at most \(2r\), since 
$
p_{\varphi_q}(\bm a y_\ell)\leq p_\varphi(\bm a y_\ell)\leq d(y_\ell,0)\leq 2r.
$ 
Taking into account that $q\in \omega$ was arbitrary, this implies that 
$
p_\varphi(\bm a s_L)\leq 2r+\xi_\infty.
$ 
Consequently,
\[
d(s_L,0)
=
\|s_L\|+p_\varphi(\bm a s_L)
\leq
2r+2\xi_\infty
<
R.
\]
Thus \(s_L\in B_d(0,R)\) for every \(L\geq1\). Since \(\lim_L s_L= x\), it follows from \eqref{eq:Tclosure-ball-H} that \(x\in H\). 

\smallskip

We now derive a contradiction. Fix \(\ell\geq1\). Since
\(x=\sum_k y_k\) in \(\mathbb T\), and since \(\varphi_{m_\ell}\) has finite
support, we obtain by the triangle inequality that 
\[
p_{\varphi_{m_\ell}}(\bm a x)
\geq
p_{\varphi_{m_\ell}}(\bm a y_\ell)
-
\sum_{k\neq\ell}p_{\varphi_{m_\ell}}(\bm a y_k).
\]
The first term is \(>\varepsilon\). If \(k<\ell\), then
\(m_\ell>M_{\ell-1}\geq M_k\), while if \(k>\ell\), then
\(m_\ell<M_\ell\leq M_{k-1}\). Hence 
$
\sum_{k\neq\ell}p_{\varphi_{m_\ell}}(\bm a y_k)\leq\xi_\infty. 
$ 
Therefore
\[
\forall \ell\ge 1, \quad
p_{\varphi_{m_\ell}}(\bm a x)>\varepsilon-\xi_\infty>\varepsilon/2.
\]
Since \(m_\ell\to\infty\), this contradicts \(x\in H\). Thus
\(\overline B_d(0,r)\) is totally bounded, hence compact. Therefore
\((H,d)\) is locally compact.
\end{proof}

Let \(H_0\) be the connected component of \(0\) in \((H,d)\). Since the
identity map 
$
\iota:(H,d)\to\mathbb T
$ 
is a continuous injective homomorphism, \(\iota[H_0]\) is a connected subgroup
of \(\mathbb T\). The only connected subgroups of \(\mathbb T\) are
\(\{0\}\) and \(\mathbb T\). Since $H$ is a proper subgroup, we get 
$
\iota[H_0]=\{0\}.
$ 
As \(\iota\) is injective, \(H_0=\{0\}\). Therefore \((H,d)\) is totally
disconnected.

By van Dantzig's theorem, \((H,d)\) has a compact open subgroup \(K\), see \cite[Theorem 7.7]{HewittRossAHA1}. Then
\(\iota[K]\) is a compact subgroup of \(\mathbb T\), hence it is finite or
equal to \(\mathbb T\). Since \(\iota[K]\subseteq H\neq\mathbb T\), the subgroup
\(\iota[K]\) is finite. Since \(\iota\) is injective, \(K\) is finite. As
\((H,d)\) is Hausdorff and \(K\) is a finite open subgroup, there exists an
open set \(O\subseteq H\) such that
\[
O\cap K=\{0\}.
\]
Since \(K\) is open, it follows that \(\{0\}\) is open in \((H,d)\). Thus
\((H,d)\) is discrete. Since \((H,d)\) is separable, we conclude that \(H\) is countable. 
\end{proof}

\medskip

\begin{proof}
[Proof of Theorem \ref{thm:strangecounterexample}]
We start constructing recusrsively an increasing sequence $(\gamma_n: n \in \omega)$ of positive integers: set $\gamma_0:=1$; suppose $\gamma_j$ has already been defined for some $j\in \omega$, pick a prime number $q_j>2^{j+10}\gamma_j$, and  
define 
\[
\gamma_{j+1}:=\gamma_j(1+q_0q_1\cdots q_j).
\]
It follows by construction that $\gamma_j\mid \gamma_{j+1}$ for every $j\in\omega$, $\lim_j \gamma_j=\infty$, and 
\begin{equation}\label{eq:gamma-congruence}
\forall n\ge j, \qquad 
\frac{\gamma_n}{\gamma_j}
=
\prod_{r=j}^{n-1}(1+q_0q_1\cdots q_r)
\equiv 1 \pmod {q_j}.
\end{equation}

At this point, for each $n \in\omega$, set $d_n:=\gamma_{n+1}-\gamma_n$ and define 
\[
\forall x \in \mathbb{T}, \qquad 
V(x):=\sum_{n\in\omega}\|d_nx\|,
\]
Finally, define $\mathsf{H}_\sharp:=\{x \in \mathbb{T}: V(x)<\infty\}$. We claim the subgroup 
\[
H:=
\mathsf{H}_\sharp 
\cap 
\mathsf{H}_{(\gamma_n)}(\mathrm{Fin}).
\]
satisfies the three properties of our claim.  

\medskip

\textbf{$\bm{H}$ is not $\bm{F_{\sigma\delta}}$-complete}. It is enough to show that $H$ is $G_{\delta\sigma}$. To this aim, define $K_m:=\{x \in \mathbb{T}: V(x)\le m\}$ for each $m \in \omega$. Taking into account that 
$$
\forall m \in \omega, \quad K_m=\bigcap_{N \in \omega}\left\{x \in \mathbb{T}: \sum_{n\le N}\|d_nx\|\le m\right\},
$$
it follows that each $K_m$ is closed. 
At the same time, define
\[
R:=
\bigcap_{k\in \omega}
\bigcap_{m\in\omega}
\bigcup_{n\geq m}
\{x\in\mathbb T:\|\gamma_nx\|<2^{-k}\}.
\]
Then $R$ is $G_\delta$, and $x \in R$ if and only if $0$ is a cluster point of
$(\gamma_nx:n\in\omega)$. If $x\in \mathsf{H}_\sharp$, then the sequence
$(\gamma_nx:n\in\omega)$ has finite total variation in $\mathbb T$, hence it
is convergent. Therefore, for $x\in \mathsf{H}_\sharp$, the condition $\gamma_nx\to0$ is
equivalent to $x\in R$. Hence
\[
H=\mathsf{H}_\sharp\cap R=\bigcup_{m \in \omega}(K_m\cap R).
\]
Since each $K_m \cap R$ is $G_\delta$, it follows that $H$ is $G_{\delta\sigma}$. 


\medskip

\textbf{$\bm{H}$ is not $\bm{F_\sigma}$}. 
Define the 
translation-invariant metric
\[
\forall x,y \in H, \quad \rho(x,y):=\|x-y\|+\sum_{n=0}^\infty\|d_n(x-y)\|
\]
Observe that $\rho$ is finite on \(H\), since \(H\) is a subgroup and \(V(x-y)<\infty\).

\begin{claim}\label{claimHPolishablerho}
    $H$ is Polishable, and $\rho$ is compatible with its Polish group topology.
\end{claim}
\begin{proof}
Endow $\ell^1(\mathbb{T}):=\left\{\bm{u}\in\mathbb T^\omega:
\sum_{n}\|u_n\|<\infty
\right\}$ with the metric 
\[
\forall \bm{u}, \bm{v} \in \ell^1(\mathbb{T}), \quad 
q((u_n),(v_n)):=\sum_{n=0}^{\infty}\|u_n-v_n\|.
\]
Then \(\ell^1(\mathbb T)\) is a Polish group. Define the map $\Phi:H\to\mathbb T\times\ell^1(\mathbb T)$ by $\Phi(x):=\bigl(x,(d_nx:n\in\omega)\bigr)$ for all $x \in H$. 
Observe that \(\rho\) is the metric induced by \(\Phi\) from the product
metric on \(\mathbb T\times\ell^1(\mathbb T)\).

We claim that \(\Phi[H]\) is closed. Indeed, since 
$
\gamma_Nx=x+\sum_{n=0}^{N-1}d_nx
$ for all $N \in \omega$ 
and \(\sum_n\|d_nx\|<\infty\), the condition \(\gamma_Nx\to0\) is equivalent to $
x+\sum_{n=0}^{\infty}d_nx=0$ in $\mathbb T$. 
Hence
\[
\Phi[H]
=
\left\{
(x,\bm{u})\in\mathbb T\times\ell^1(\mathbb T):
u_n=d_nx\text{ for all }n\in\omega
\text{ and }
x+\sum_{n=0}^{\infty}u_n=0
\right\}.
\]
The map 
$
(u_n)\mapsto\sum_{n=0}^{\infty}u_n
$ 
from \(\ell^1(\mathbb T)\) to \(\mathbb T\) is continuous, and the conditions
\(u_n=d_nx\) are closed. Therefore \(\Phi[H]\) is a closed subgroup of the
Polish group \(\mathbb T\times\ell^1(\mathbb T)\). Thus \(\Phi[H]\) is Polish,
and consequently \((H,\rho)\) is a Polish group.

Moreover, the \(\rho\)-topology refines the topology inherited from \(\mathbb T\),
because 
$
\|x-y\|\leq \rho(x,y).
$ 
Finally, the two Borel structures coincide: each \(\rho\)-ball is Borel in the
inherited Borel structure, since the map 
$
x\mapsto
\|x-a\|+\sum_{n}\|d_n(x-a)\|
$ 
is Borel as the pointwise limit of its continuous partial sums. Since
\((H,\rho)\) is separable, every \(\rho\)-open set is a countable union of such
balls. 
\end{proof}

\begin{claim}\label{claim:nonFsigmaClosure}
For every \(\varepsilon>0\), the closure of the \(\rho\)-open ball 
$
B_\rho(0,\varepsilon):=\{x\in H:\rho(x,0)<\varepsilon\}
$ 
in \(\mathbb T\) is not contained in \(H\).
\end{claim}
\begin{proof}
Fix \(\varepsilon>0\). Since \(q_j>2^{j+10}\gamma_j\), we have
\[
\forall j \in \omega, \quad 
\frac{1+3\gamma_j}{q_j}
\leq
\frac{4\gamma_j}{q_j}
<
\frac{1}{2^{j+8}}.
\]
Hence it is possible to choose \(j\in\omega\) such
that $(1+3\gamma_j)/q_j<\varepsilon$. 
Put \(t_j:=\nicefrac{1}{q_j}\in\mathbb T\). We first observe that
\(t_j\notin H\). Indeed, by \eqref{eq:gamma-congruence}, for every \(n\geq j\)
we have \(\gamma_n\equiv\gamma_j\pmod {q_j}\). Since \(0<\gamma_j<q_j\), it
follows that 
$\gamma_nt_j=\frac{\gamma_j}{q_j}\neq0$ for all $n\ge j$. 
Thus \((\gamma_nt_j:n\in\omega)\) is eventually constant and nonzero, and
therefore \(t_j\notin \mathsf H_{(\gamma_n)}(\mathrm{Fin})\). In particular,
\(t_j\notin H\).

Now fix \(m>j\) and define
\[
h_m:=\frac1{q_j}-\frac{\gamma_j}{q_j\gamma_m}\in\mathbb T.
\]
We claim that \(h_m\in H\). Let \(n\geq m\). Then 
$
\gamma_nh_m
=
\frac{\gamma_n-\gamma_j(\gamma_n/\gamma_m)}{q_j}.
$ 
Again by \eqref{eq:gamma-congruence}, we have
\(\gamma_n\equiv\gamma_j\pmod {q_j}\). Moreover, since \(n\geq m>j\), the
integer \(\gamma_n/\gamma_m\) is congruent to \(1\) modulo \(q_j\). Hence
\[
\gamma_n-\gamma_j(\gamma_n/\gamma_m)\equiv0\pmod {q_j},
\]
and therefore \(\gamma_nh_m=0\) in \(\mathbb T\) for every \(n\geq m\). It
follows that \(\lim_n\gamma_nh_m=0\). Similarly, \(d_nh_m=0\) for every \(n\geq m\),
so that \(V(h_m)<\infty\). Therefore \(h_m\in H\).

To complete the proof of the claim, we estimate \(\rho(h_m,0)\). First,
\begin{equation}\label{eq:hmnorm}
\|h_m\|
\leq
\frac1{q_j}+\frac{\gamma_j}{q_j\gamma_m}
\leq
\frac1{q_j}+\frac{\gamma_j}{q_j}.
\end{equation}
Since \(d_n=\gamma_{n+1}-\gamma_n=\gamma_nq_0q_1\cdots q_n\), the prime
\(q_j\) divides \(d_n\) for every \(n\geq j\). Hence
\[
V\left(\frac1{q_j}\right)
=
\sum_{n<j}\left\|\frac{d_n}{q_j}\right\|
\leq
\frac{\gamma_j-\gamma_0}{q_j}
<
\frac{\gamma_j}{q_j}.
\]
Similarly, since \(\gamma_m\mid\gamma_n\) for every \(n\geq m\), we have
\[
V\left(\frac{\gamma_j}{q_j\gamma_m}\right)
\leq
\frac{\gamma_j}{q_j\gamma_m}\sum_{n<m}d_n
=
\frac{\gamma_j(\gamma_m-\gamma_0)}{q_j\gamma_m}
<
\frac{\gamma_j}{q_j}.
\]
By subadditivity of \(V\), together with \eqref{eq:hmnorm}, we obtain
\[
\rho(h_m,0)
=
\|h_m\|+V(h_m)
<
\frac{1+3\gamma_j}{q_j}
<
\varepsilon.
\]
Thus \(h_m\in B_\rho(0,\varepsilon)\) for
every \(m>j\).

Finally, since \(\lim_m\gamma_m=\infty\), we have 
$
\lim_m h_m=
\lim_mt_j\left(1-\frac{\gamma_j}{\gamma_m}\right)=t_j
$ 
in \(\mathbb T\). This completes the proof since \(t_j\notin H\).
\end{proof}

Now, suppose for the sake of contradiction that \(H\) is \(F_\sigma\) in \(\mathbb T\).
Then there are closed sets \(F_m\subseteq\mathbb T\) such that
\[
H=\bigcup_{m\in\omega}F_m.
\]
Since the \(\rho\)-topology refines the topology
inherited from \(\mathbb T\), each \(F_m\) is closed in the Polish group
\((H,\rho)\). Thanks to Claim \ref{claimHPolishablerho} and the Baire category theorem, there is some $F_m$ with nonempty interior. Hence there are 
\(x\in H\) and \(\varepsilon>0\) such that 
$
x+B_\rho(0,\varepsilon)\subseteq F_m.
$ 
By Claim \ref{claim:nonFsigmaClosure}, it is possible to choose 
$
t\in\overline{B_\rho(0,\varepsilon)}^{\,\mathbb T}\setminus H.
$ 
Then
\[
x+t\in
\overline{x+B_\rho(0,\varepsilon)}^{\,\mathbb T}
\subseteq
F_m
\subseteq H.
\]
Since \(x\in H\) and \(H\) is a subgroup, this implies that 
$
t=(x+t)-x\in H,
$ 
which is the desired contradiction impossible. Therefore \(H\) is not \(F_\sigma\).

\medskip

\textbf{$\bm{H}$ is $\bm{\mathcal{I}_{1/n}}$-characterized}. 
We shall construct $\bm a\in \omega^\omega$ such that
$H=\mathsf H_{\bm a}(\I_{1/n})$. 
Let
\[
J:=\{(0,n):n\in\omega\}\cup\{(1,n,r):n\in\omega,\ r\geq1\}.
\]
For $j\in J$, define integers $c_j$ and positive real numbers $\lambda_j$ by
\[
c_{(0,n)}:=\gamma_n,
\,\,\,\,
c_{(1,n,r)}:=rd_n,
\,\,\,\,
\lambda_{(0,n)}:=1,
\,\,\text{ and }\,\,
\lambda_{(1,n,r)}:=\frac1{r^2}.
\]
By the greedy algorithm, it is also possible to pick 
pairwise disjoint finite subsets $(B_j: j \in J)$ of $\omega$ such that
$$
\forall j \in J, \quad 
\lambda_j
\leq
\sum_{m\in B_j}\frac1{m+1}
\leq
2\lambda_j.
$$

At this point, define $\bm a=(a_m:m\in\omega)$ by $a_m:=c_j$ if $m\in B_j$, and 
$a_m:=0$ if $m\notin\bigcup_{j\in J}B_j$. Fix $x\in\mathbb T$ and
$\varepsilon>0$. By the choice of $B_j$, we have 
\begin{equation}\label{eq:firstequivalence}
    \begin{split}
        \{m\in\omega:\|a_mx\|\geq\varepsilon\}\in\I_{1/n}
&\quad \Longleftrightarrow\quad 
\sum_{\{j\in J:\|c_jx\|\geq\varepsilon\}}\lambda_j<\infty\\
&\quad \Longleftrightarrow\quad 
\sum_{\{n:\|\gamma_nx\|\geq\varepsilon\}}1
+
\sum_{n=0}^{\infty}\sum_{r=1}^{\infty}
\frac1{r^2}\mathbf 1_{\{\|rd_nx\|\geq\varepsilon\}}
<\infty. 
    \end{split}
\end{equation}

Now, for each $\varepsilon>0$ and $t \in \mathbb{T}$, define
\[
F_\varepsilon(t):=
\sum_{r=1}^{\infty}
\frac1{r^2}\mathbf 1_{\{\|rt\|\geq\varepsilon\}}.
\]
For every $\varepsilon>0$, there is $C_\varepsilon<\infty$ such that
\begin{equation}\label{eq:Feps-upper}
\forall t \in \mathbb{T}, \quad 
F_\varepsilon(t)\leq C_\varepsilon\|t\|.
\end{equation}
Indeed, if $\delta:=\|t\|>0$, then $\|rt\|<\varepsilon$ whenever
$r<\varepsilon/\delta$. Hence
$F_\varepsilon(t)\leq\sum_{r\geq\varepsilon/\delta}r^{-2}$, which is
bounded by a constant multiple of $\delta$; the case $\delta=0$ is trivial.

Conversely, for $\eta:=\nicefrac1{16}$, there is $c>0$ such that
\begin{equation}\label{eq:Feta-lower}
\forall t \in \mathbb{T}, \quad 
c\|t\|\leq F_\eta(t).
\end{equation}
Indeed, put $\delta:=\|t\|$. If $\delta\geq\eta$, the term $r=1$ gives the
claim, after decreasing $c$ if necessary. If $0<\delta<\eta$, then for each
integer $r$ with $(8\delta)^{-1}\leq r\leq (4\delta)^{-1}$ we have
$\eta<\nicefrac18\leq r\delta\leq\nicefrac14$, and therefore
$\|rt\|\geq\eta$. Summing $r^{-2}$ over these integers gives a lower bound
which is of the type $c\delta$ for some $c>0$.

It follows from \eqref{eq:Feps-upper} and \eqref{eq:Feta-lower} that, for
every sequence $(t_n:n\in\omega)$ in $\mathbb T$,
\begin{equation}\label{eq:l1-Feps-equivalence}
\sum_{n=0}^{\infty}\|t_n\|<\infty
\quad\Longleftrightarrow\quad
\sum_{n=0}^{\infty}F_\varepsilon(t_n)<\infty
\text{ for every }\varepsilon>0.
\end{equation}
Applying \eqref{eq:l1-Feps-equivalence} to $t_n=d_nx$, we get that
\[
\sum_{n=0}^\infty\|d_nx\|<\infty
\quad \Longleftrightarrow \quad 
\sum_{n=0}^{\infty}\sum_{r=1}^{\infty}
\frac1{r^2}\mathbf 1_{\{\|rd_nx\|\geq\varepsilon\}}<\infty
\text{ for every }\varepsilon>0.
\]
On the other hand,
\[
\lim_{n\to \infty} \gamma_nx=0
\quad \Longleftrightarrow \quad 
\sum_{\{n:\|\gamma_nx\|\geq\varepsilon\}}1<\infty
\text{ for every }\varepsilon>0
\]

Combining these two equivalences with
\eqref{eq:firstequivalence}, we conclude that $x\in\mathsf H_{\bm a}(\I_{1/n})$ if and only if $\lim_n \gamma_nx=0$ and $\sum_n \|d_nx\|<\infty$, i.e., if and only if $x \in H$. 
\end{proof}

\begin{remark}\label{rmk:strangeLQC}
    Thanks to Theorem \ref{thm:analyticPidealscharacterization}, the subgroup $H$ constructed in the proof above is Polishable. We claim that $H$ is also locally quasi-convex with respect to its Polish group topology. 
    Indeed, recall from the proof of Theorem \ref{thm:strangecounterexample} that, setting $d_n:=\gamma_{n+1}-\gamma_n$ for all $n\in\omega$, such topology is generated by the translation-invariant metric
    $$
    \rho(x,y):=\|x-y\|+\sum_{n\in\omega}\|d_n(x-y)\|
    $$
    for all $x,y\in H$, cf. Claim \ref{claimHPolishablerho}. 

    Fix a neighborhood $U$ of $0$ in $(H,\rho)$, and choose $\varepsilon>0$ such that $B_\rho(0,\varepsilon)\subseteq U$. For notational convenience, set $d_{-1}:=1$. Pick also a positive integer $K$ such that 
    $1/(4K)<\varepsilon$, and define
    $$
    A_K:=
    \left\{
    \left|k\sum_{j\in F}\eta_jd_j\right|:
    F\in[\{-1\}\cup\omega]^{<\omega},
    \,\,\eta_j\in\{-1,1\},
    \,\,\text{ and }\,\,1\leq k\leq K
    \right\}.
    $$
    Consider
    $$
    V:=
    \left\{
    x\in H:
    \sup_{a\in A_K}\|ax\|\leq\frac14
    \right\}.
    $$
    
    If $\rho(x,0)<1/(4K)$ and $a=\left|k\sum_{j\in F}\eta_jd_j\right|\in A_K$, then
    $$
    \|ax\|
    \leq
    k\sum_{j\in F}\|d_jx\|
    \leq
    K\rho(x,0)
    <
    \frac14.
    $$
    Hence $B_\rho(0,1/(4K))\subseteq V$, so $V$ is a neighborhood of $0$.

    We claim that $V\subseteq B_\rho(0,\varepsilon)$. Fix $y\in H$ such that $\rho(y,0)\geq\varepsilon$, and, for each $j\in\{-1\}\cup\omega$, choose $t_j\in[-\nicefrac12,\nicefrac12]$ such that $d_jy=t_j+\mathbb Z$. Then
    $$
    \sum_{j\in\{-1\}\cup\omega}|t_j|
    =
    \rho(y,0)
    \geq\varepsilon
    >
    \frac{1}{4K}.
    $$
    Therefore there is a finite set $F\subseteq\{-1\}\cup\omega$ such that
    $$
    \frac{1}{4K}
    <
    s:=\sum_{j\in F}|t_j|
    \leq
    \frac12.
    $$
    Indeed, if some $|t_j|>1/(4K)$, take $F=\{j\}$; otherwise take the first finite partial sum which exceeds $1/(4K)$, which is at most $1/(2K)\leq\nicefrac12$. 
    Now, 
    for each $j\in F$, choose $\eta_j\in\{-1,1\}$ such that $\eta_jt_j=|t_j|$. If $s>\nicefrac14$, set $k:=1$; otherwise, set $k:=\lfloor1/(4s)\rfloor+1$. Since $s>1/(4K)$, we have $1\leq k\leq K$, and in both cases 
    $
    \nicefrac14<ks\leq\nicefrac12.
    $ 
    Hence, setting $a:=\left|k\sum_{j\in F}\eta_jd_j\right|\in A_K$, we obtain
    $$
    \|ay\|
    =
    \left\|k\sum_{j\in F}\eta_jd_jy\right\|
    =
    \|ks\|
    =
    ks
    >
    \frac14.
    $$
    Thus $y\notin V$, proving that $V\subseteq B_\rho(0,\varepsilon)\subseteq U$.

    Finally, $V$ is quasi-convex: if $y\in H\setminus V$, then there exists $a\in A_K$ such that
    $$
    \sup_{x\in V}\|ax\|\leq\nicefrac14<\|ay\|.
    $$
    Therefore $H$ is locally quasi-convex with respect to its Polish group topology. 
\end{remark}

\begin{remark}\label{rmk:strangenotLQCksdjhg}
    There exists an $\I_{1/n}$-characterized subgroup of $\mathbb{T}$ which is $\mathsf{D}(\mathbf{\Pi}^0_2)$-complete and Polishable, but not locally quasi-convex.

    Indeed, it is enough to modify slightly the construction in the proof of Theorem \ref{thm:strangecounterexample}. 
    Set $j_0:=0$ and $\gamma_0:=1$, and put $C:=\sum_{r\in\omega}2^{-r/2}$. Suppose that $\gamma_{j_s}$ has been defined. Pick a prime number $q_s>2\gamma_{j_s}$ such that
    $$
    \frac{1+\gamma_{j_s}}{q_s}
    +(j_s+C)\sqrt{\frac{\gamma_{j_s}}{q_s}}
    <2^{-s}.
    $$
    Set $Q_s:=q_0\cdots q_s$, $R_s:=1+Q_s$, $N_s:=\lfloor R_s^{3/4}\rfloor$, and $j_{s+1}:=j_s+N_s$. For every $j_s\leq n<j_{s+1}$, define $\gamma_{n+1}:=R_s\gamma_n$. Hence
    $$
    \lim_{s\to \infty}\frac{N_s}{R_s}=0,
    \qquad
    \lim_{s\to \infty}\frac{N_s}{\sqrt{R_s}}=\infty,
    \quad \text{ and } \quad 
    \frac{\gamma_n}{\gamma_{j_s}}\equiv1\pmod{q_s}
    \quad\text{for all }n\geq j_s.
    $$
    Put $d_n:=\gamma_{n+1}-\gamma_n$ and define the subgroup
    $$
    H:=
    \left\{
    x\in\mathbb{T}:
    \sum_{n\in\omega}\sqrt{\|d_nx\|}<\infty
    \text{ and }
    \lim_{n\to \infty}\gamma_nx=0
    \right\}.
    $$
    
    As in Claim \ref{claimHPolishablerho}, $H$ is Polishable and its Polish group topology is generated by 
    $$
    \rho(x,y):=
    \|x-y\|+\sum_{n\in\omega}\sqrt{\|d_n(x-y)\|}.
    $$ 
    Moreover, repeating the proof of Theorem \ref{thm:strangecounterexample}, with the weights $r^{-2}$ replaced by $r^{-3/2}$, one obtains an integer sequence $\bm{a}$ such that
    $H=\mathsf{H}_{\bm{a}}(\I_{1/n})$. Indeed, if
    $$
    F_\varepsilon(t):=
    \sum_{r=1}^{\infty}r^{-3/2}
    \mathbf{1}_{\{\|rt\|\geq\varepsilon\}},
    $$
    then, for every $\varepsilon>0$, there are constants $C_\varepsilon,c>0$ such that
    $$
    c\sqrt{\|t\|}
    \leq F_{1/16}(t)
    \quad\text{and}\quad
    F_\varepsilon(t)\leq C_\varepsilon\sqrt{\|t\|}
    $$
    for all $t\in\mathbb{T}$. The same argument also shows that $H$ is $G_{\delta\sigma}$.

   To show that $H$ is not $F_\sigma$, put $t_s:=1/q_s$ for each $s \in \omega$ and, for $m>j_s$,
    $$
    h_{s,m}:=
    \frac1{q_s}-\frac{\gamma_{j_s}}{q_s\gamma_m}.
    $$
    As in Claim \ref{claim:nonFsigmaClosure}, we have $t_s\notin H$, $h_{s,m}\in H$, $\lim_mh_{s,m}=t_s$, and
    $$
    \rho(h_{s,m},0)
    \leq
    \frac{1+\gamma_{j_s}}{q_s}
    +(j_s+C)\sqrt{\frac{\gamma_{j_s}}{q_s}}
    <2^{-s}.
    $$
    The analogue Baire category argument shows that $H$ is not $F_\sigma$. Together with the above observation, Corollary \ref{corollary:lupini} implies that $H$ is $\mathsf{D}(\mathbf{\Pi}^0_2)$-complete.

    It remains to show that $H$ is \emph{not} locally quasi-convex. For $j_s\leq n<j_{s+1}$, put $x_n:=1/\gamma_{n+1}$. Then $x_n\in H$ and, uniformly for $1\leq\ell\leq N_s$,
    $$
    \rho(\ell x_n,0)
    \leq
    \frac{N_s}{R_s}
    +(1+C)\sqrt{\frac{N_s}{R_s}}
    \longrightarrow0 \qquad \text{ as }s\to \infty.
    $$
    Suppose for the sake of contradiction that $H$ is locally quasi-convex. Pick a quasi-convex neighborhood $V$ of $0$ contained in $B_\rho(0,1)$. For all sufficiently large $s$, we have $\ell x_n\in V$ whenever $j_s\leq n<j_{s+1}$ and $1\leq\ell\leq N_s$. Set
    $$
    y_s:=\sum_{n=j_s}^{j_{s+1}-1}(-1)^{n-j_s}x_n.
    $$
    We claim that $y_s\in V$. Otherwise, there exists a continuous character $\phi \in \widehat{H}$ such that $\sup_{x\in V}\|\phi(x)\|\leq\nicefrac14<\|\phi(y_s)\|$. Since $\ell x_n\in V$ for every $1\leq\ell\leq N_s$, we easily obtain $\|\phi(x_n)\|\leq1/(4N_s)$, and hence 
    $$
    \|\phi(y_s)\|
    \leq \sum_{j_s\le n<j_{s+1}}\|\phi(x_n)\|
    \le \frac14,
    $$
    which is a contradiction.

    On the other hand, for every $j_s\leq n<j_{s+1}$, the alternating signs give 
    $
    \|d_ny_s\|\geq 1/{R_s}
    $ 
    for all sufficiently large $s$. Consequently,
    $$
    \rho(y_s,0)
    \geq
    \frac{N_s}{\sqrt{R_s}}
    \longrightarrow\infty 
    \qquad \text{ as }s\to \infty,
    $$
    contradicting $y_s\in V\subseteq B_\rho(0,1)$. Therefore $H$ is not locally quasi-convex.
\end{remark}

\section{Proofs of Results from Section \ref{sec:exampleapplications}}\label{sec:proofexamples}

\begin{proof}
[Proof of Theorem \ref{thm:armacost}]
Set $\bm{a}:=(b^n: n \in \omega)$. Fix $x \in \mathbb{T}$, with its representation as in
\eqref{eq:representationx}, and set for simplicity $d_k:=d_{x,k}$ for all $k \in \omega$. Define
$$
B_{m}:=
\{k \in \omega: \overline{d_{k+1}d_{k+2}\cdots d_{k+m}}
\notin \{0^m,(b-1)^m\}\}
\quad \text{ for all }m\ge 1.
$$
Observe also that
$$
b^n x=\overline{0.d_{n+1}d_{n+2}d_{n+3}\ldots}_{(b)}
\quad \text{ in } \mathbb{T}
$$
for all $n \in \omega$.

\begin{claim}\label{claim:armacost}
$x \in \mathsf{H}_{\bm{a}}(\I)$ if and only if $B_m\in \I$ for all $m\ge 1$.
\end{claim}

\begin{proof}
First, suppose that $x \in \mathsf{H}_{\bm{a}}(\I)$, and fix $m\geq 1$.
If $k\in B_m$, then the block 
$
\overline{d_{k+1}\cdots d_{k+m}}
$ 
is neither constantly $0$ nor constantly $b-1$. Hence the point $b^kx$ is at
distance at least $b^{-m}$ from $0$. Thus
$ 
B_m\subseteq \{k\in\omega:\|b^kx\|\geq b^{-m}\}\in \I.
$ 

Conversely, suppose that $B_m\in \I$ for every $m\geq 1$.
Pick $\varepsilon>0$, and choose $m\geq 1$ such that $b^{-m}<\varepsilon$.
If $k\notin B_m$, then the block 
$
\overline{d_{k+1}\cdots d_{k+m}}
$ 
is either constantly $0$ or constantly $b-1$. In both case, we have
$
\|b^kx\|\le b^{-m}<\varepsilon.
$ 
Therefore
$$
\{k\in\omega:\|b^kx\|\geq \varepsilon\}\subseteq B_m \in \I.
$$
As $\varepsilon>0$ was arbitrary, we conclude that
$x\in \mathsf{H}_{\bm{a}}(\I)$. 
\end{proof}

\begin{claim}\label{claim:armacost2}
For every $m\ge 1$, we have
$$
B_m=
\bigcup_{j=1}^m (P_x-j)
\cup
\bigcup_{j=1}^{m-1} (Q_x-j).
$$
\end{claim}

\begin{proof}
Fix $k\in\omega$ and $m\ge 1$. The block 
$
\overline{d_{k+1}\cdots d_{k+m}}
$ 
is neither constantly $0$ nor constantly $b-1$ if and only if one of the
following two alternatives holds. 
Either some digit of the block belongs to $\{1,\ldots,b-2\}$, which is
equivalent to saying that $k\in P_x-j$ for some $j\in \{1,\ldots,m\}$. 
Or all digits of the block belong to $\{0,b-1\}$, but the block is not
constant. This is equivalent to saying that there is a change inside the
block, namely 
$
d_{k+j}\neq d_{k+j+1}
$ 
for some $j \in \{1,\ldots,m-1\}$, that is, $k\in Q_x-j$ for some
$j \in \{1,\ldots,m-1\}$.
\end{proof}

Putting together Claim \ref{claim:armacost} and Claim \ref{claim:armacost2}, we conclude that $x \in \mathsf{H}_{\bm{a}}(\I)$ if and only if $P_x-m \in \mathcal{I}$ and $Q_x-m \in \mathcal{I}$ for all $m\ge 1$. 
\end{proof}

\medskip

\begin{proof}
    [Proof of Corollary \ref{cor:armacostconsequence}]
The first part is immediate by Theorem \ref{thm:armacost}. For the second part, if $b=2$, observe that $P_x=\emptyset$ for all $x \in \mathbb{T}$. Hence $x \in \mathsf{H}_{(2^n)}(\I)$ if and only if $Q_x \in \I$. 
\end{proof}

\medskip

\begin{proof}
    [Proof of Theorem \ref{thm:propertiespowers}]
    In this proof, set $\bm{a}:=(b^n: n \in \omega)$.
    
    \ref{prop:1powers} First, suppose that $\mathcal I=\mathrm{Fin}$.
By Corollary \ref{cor:armacostconsequence}, we have that $x\in\mathsf H_{\bm a}(\mathrm{Fin})$ if and only if $P_x\cup Q_x\in\mathrm{Fin}$. 
This means that, eventually, the canonical base $b$-digits of $x$ have no digit
in $\{1,\ldots,b-2\}$ and have no digit changes. Hence the canonical expansion
of $x$ is eventually constantly $0$ or eventually constantly $b-1$. Since the
latter case is impossible, it follows that $\mathsf H_{\bm a}(\mathrm{Fin})
=\mathbb Z[1/b]/\mathbb Z$, which is countable. 

Conversely, suppose that $\mathcal I\neq \mathrm{Fin}$ is translation invariant, and pick an infinite set
$A\in\mathcal I$. 
For every infinite set $T\subseteq A$, define a sequence of digits
$(d_k:k\in\omega)$ by setting $d_0:=0$, requiring each $d_k$ to belong to
$\{0,b-1\}$, and imposing
$$
d_k\neq d_{k+1}
\quad \text{ if and only if } \quad
k\in T.
$$
At this point, define 
\begin{equation}\label{eq:definitionxT}
x_T:=\overline{0.d_1d_2d_3\ldots}_{(b)}.
\end{equation}
Since $T$ is infinite, the resulting expansion is not eventually constant, hence
it is canonical. In addition, 
$P_{x_T}=\emptyset$ and $Q_{x_T}=T \subseteq A \in \I$. Hence by Corollary \ref{cor:armacostconsequence}, we get $x_T\in\mathsf H_{\bm a}(\mathcal I)$. 
Since $T\mapsto x_T$ is injective, we conclude that $\mathsf H_{\bm a}(\mathcal I)$ has cardinality $\mathfrak{c}$. 

    \medskip

    \ref{prop:2powers} 
    First, suppose that \(\mathcal I\) is \(F_\sigma\). 
    Then there is an increasing sequence
\((K_m:m\in\omega)\) of hereditary compact subsets of \(2^\omega\) such that 
$
\mathcal I=\bigcup_{m}K_m.
$ 
Define also the compact space $\Sigma_b:=\{e\in\{0,1,\ldots,b-1\}^\omega:e_0=0\}$ and the continuous map $\pi:\Sigma_b\to\mathbb T$ by 
$$
\forall e \in \Sigma_b, \quad 
\pi(e):=\sum_{k \in \omega}e_k b^{-k}. 
$$
In addition, define also the map $R:\Sigma_b\to \mathcal{P}(\omega)$ by 
$$
\forall e\in\Sigma_b, \quad 
R(e):=
\{k\in\omega:e_k\in\{1,\ldots,b-2\}
\, \text{ or }\,
e_k\neq e_{k+1}\}.
$$
Observe that $R$ is continuous, because membership of a given $k$ in $R(e)$ depends only on $e_k$ and $e_{k+1}$.
\begin{claim}\label{claim:represpowerskjhfg}
    $\mathsf H_{\bm a}(\mathcal I)
=
\pi\bigl[\{e\in\Sigma_b:R(e)\in\mathcal I\}\bigr]$. 
\end{claim}
\begin{proof}
    If $e\in\Sigma_b$ represents $x:=\pi(e) \in \mathbb{T}$, then $e$ differs from the canonical base $b$-representation of $x$ at most by replacing an eventually constant tail $0^\infty$ with an eventually constant tail $(b-1)^\infty$, or conversely. 
Therefore $R(e)$ and $P_x\cup Q_x$ have finite symmetric difference. This implies that
$$
R(e)\in\mathcal I
\quad \text{ if and only if } \quad
P_x\cup Q_x\in\mathcal I.
$$
The claim follows by Corollary \ref{cor:armacostconsequence}. 
\end{proof}
Using Claim \ref{claim:represpowerskjhfg} and that $\mathcal I=\bigcup_{m}K_m$, we obtain
$$
\mathsf H_{\bm a}(\mathcal I)
=
\pi[R^{-1}[\I]]=
\bigcup_{m\in\omega}\pi\bigl[R^{-1}[K_m]\bigr].
$$
For every $m\in\omega$, the set $R^{-1}[K_m]$ is compact in $\Sigma_b$,
and hence $\pi[R^{-1}[K_m]]$ is compact in $\mathbb T$. Therefore
$\mathsf H_{\bm a}(\mathcal I)$ is an $F_\sigma$ subgroup.

\smallskip

Conversely, suppose that $\mathcal I$ is translation invariant and that
$\mathsf H_{\bm a}(\mathcal I)$ is $F_\sigma$. We proceed as in item \ref{prop:1powers}: for each $T\subseteq A:=\omega$, define $\rho(T):=x_T$ as in \eqref{eq:definitionxT} (analogously, also for finite $T$). In this way, the map $\rho: \mathcal{P}(\omega)\to \mathbb{T}$ is continuous, since each digit $d_k$ depends only on $T\cap \{0,1,\ldots,k-1\}$. 
\begin{claim}\label{claim:reprepowersHa}
    $\rho^{-1}[\mathsf H_{\bm a}(\mathcal I)]=\mathcal I$.
\end{claim}
\begin{proof}
    If $T$ is finite, then $T\in\mathcal I$, and the digit sequence defining
$\rho(T)$ is eventually constant. Hence the canonical base $b$-representation
of $\rho(T)$ is eventually $0$, and so
$\rho(T)\in\mathsf H_{\bm a}(\mathrm{Fin})\subseteq \mathsf H_{\bm a}(\mathcal I)$. 

Now suppose that $T$ is infinite. Then the digit sequence defining $\rho(T)$
has infinitely many changes, hence it is canonical. Moreover,
$
P_{\rho(T)}=\emptyset$ and $Q_{\rho(T)}=T$. It follows by Corollary \ref{cor:armacostconsequence} that $\rho(T)\in\mathsf H_{\bm a}(\mathcal I)$ if and only if $T \in \I$. 
\end{proof}
We conclude by Claim \ref{claim:reprepowersHa} that $\I \le_{\mathrm{W}} \mathsf H_{\bm a}(\mathcal I)$, hence $\I$ is $F_\sigma$.

    \medskip

    \ref{prop:3powers} 
    The \textsc{If} part is clear. Conversely, suppose that $\mathcal I$ and $\mathcal J$ are translation invariant
and that 
$
\mathsf H_{\bm a}(\mathcal I)=\mathsf H_{\bm a}(\mathcal J). 
$ 
Fix an infinite set $T\subseteq \omega$, and define $\rho(T)$ as in the proof of item \ref{prop:2powers}. Using Claim \ref{claim:reprepowersHa}, we obtain 
\begin{displaymath}
    \begin{split}
        T \in \I 
        &\quad \Longleftrightarrow \quad \rho(T) \in \mathsf H_{\bm a}(\mathcal I)\\
        &\quad \Longleftrightarrow \quad \rho(T) \in \mathsf H_{\bm a}(\mathcal J)
        \quad \Longleftrightarrow \quad T \in \J. 
    \end{split}
\end{displaymath}
Taking into account that $\mathrm{Fin}\subseteq \I\cap \J$, we conclude that $\I=\J$. 
\end{proof}

\medskip

\begin{proof}
[Proof of Corollary \ref{cor:consequencepowers22}]
    \ref{item:1corpowers} It follows by the proof of Theorem \ref{thm:propertiespowers}\ref{prop:1powers} that $\mathsf H_{(b^n)}(\mathrm{Fin})
=\mathbb Z[1/b]/\mathbb Z$, which is countably infinite.

    \medskip

    \ref{item:2corpowers} Suppose that $\mathcal{I}\neq \mathrm{Fin}$ is an $F_\sigma$ translation invariant ideal. Then $\mathsf{H}_{(b^n)}(\I)$ is an $F_\sigma$ uncountable subgroup by Theorem \ref{thm:propertiespowers}\ref{prop:1powers}-\ref{prop:2powers}. In addition, setting 
    $$
    x:=\overline{0.0(b-1)0(b-1)\cdots}_{(b)}\in \mathbb{T}
    $$
    we have $Q_x=\omega \notin \I$, so that $x\notin \mathsf{H}_{(b^n)}(\I)$ by Corollary \ref{cor:armacostconsequence}. 
    Hence $\mathsf{H}_{(b^n)}(\I)\neq \mathbb{T}$.

    \medskip

    \ref{item:3corpowers} Suppose that $\I$ is a translation invariant generalized density ideal with $\I\neq \mathrm{Fin}$. 
    
    Thanks to \cite[Proposition 2.3]{MR4404626}, if $\I$ is $F_\sigma$ then $\I$ is a copy on $\omega$ of $\mathrm{Fin}\oplus \mathcal{P}(\omega)$, namely, there exists an infinite $A\subseteq \omega$ such that $\omega\setminus A$ is infinite and $\I=\{S\subseteq \omega: A\cap S\in \mathrm{Fin}\}$. 
    Define the infinite set $S_\star:=S_\star:=\{n\in A:n+1\notin A\}$. Hence $S_\star\notin\I$, whereas $S_\star+1\subseteq\omega\setminus A$, and therefore $S_\star+1\in\I$. This contradicts the translation invariance of $\I$.
    
    If $\I$ is not $F_\sigma$, it follows by Theorem \ref{thm:Fsigmadelcomplete} and Theorem \ref{thm:propertiespowers}\ref{prop:2powers} that the subgroup $\mathsf{H}_{(b^n)}(\I)$ is $F_{\sigma\delta}$-complete. 
\end{proof}

\medskip

\begin{proof}
    [Proof of Lemma \ref{lem:fiborepresent}]
  Pick $t \in [0,1)$ such that $x=t+\mathbb{Z}$. Define 
  \begin{equation}\label{eq:defdeltan}
  \forall n \in \omega, \quad 
  \delta_n:=f_nt-p_n(x)
  \end{equation}
    Then $|\delta_n|\le \nicefrac12$ for all $n \in \omega$. Moreover, using the identity \(f_{n+2}=f_{n+1}+f_n\), we get
\begin{equation}\label{eq:defdeltanjshw}
r_{x,n}=p_{n+2}(x)-p_{n+1}(x)-p_n(x)
      =\delta_n+\delta_{n+1}-\delta_{n+2}.
\end{equation}
Since the left-hand side is an integer and the right-hand side belongs to
\((-\nicefrac{3}{2},\nicefrac{3}{2})\), it follows that \(r_{x,n}\in\{-1,0,1\}\) for every
\(n\in\omega\).

Now, the sequence \((p_n(x):n\in\omega)\) satisfies
\(p_{n+2}(x)=p_{n+1}(x)+p_n(x)+r_{x,n}\). Hence, by induction, for every
\(N\geq 1\), 
$$
p_N(x)=p_1(x)f_N+p_0(x)f_{N-1}
+\sum_{n=0}^{N-2}r_{x,n}f_{N-n-1}.
$$
Dividing by \(f_N\) and recalling that \(p_0(x)=0\), we obtain 
$$
\frac{p_N(x)}{f_N}
=
p_1(x)+\sum_{n=0}^{N-2}r_{x,n}\frac{f_{N-n-1}}{f_N}.
$$
For each fixed \(n\in\omega\), we have
\(\lim_N f_{N-n-1}/f_N=\varphi^{-n-1}\). In addition, there is a constant
\(C>0\) such that 
$
f_{N-n-1}/f_N\leq C\varphi^{-n}
$ 
for all \(N\geq n+2\). 
Since \(\sum_n\varphi^{-n}<\infty\) and
\(r_{x,n}\in\{-1,0,1\}\), it follows that
$$
\lim_{N\to\infty}\frac{p_N(x)}{f_N}
=
p_1(x)+\sum_{n=0}^{\infty}r_{x,n}\varphi^{-n-1}.
$$
On the other hand, since \(p_N(x)=f_Nt-\delta_N\), the sequence
\((\delta_N:N\in\omega)\) is bounded, and \(\lim_N f_N=\infty\), we have 
$ 
\lim_{N}p_N(x)/f_N=t.
$ 
Therefore
$$
t
=
p_1(x)+\sum_{n=0}^{\infty}r_{x,n}\varphi^{-n-1}.
$$
Since \(p_1(x)\in\mathbb Z\), this proves the claim in \(\mathbb T\).
\end{proof}

\medskip

\begin{proof}
   [Proof of Theorem \ref{thm:fibonacci}]
   Fix $x \in \mathbb{T}$ as $t \in [0,1)$ such that $x=t+\mathbb{Z}$. Using the same notation as in the proof of Lemma \ref{lem:fiborepresent}, observe by the identity \eqref{eq:defdeltanjshw} that, if \(n\in R_x\), then at least one of
\(|\delta_n|,|\delta_{n+1}|,|\delta_{n+2}|\) is greater than or equal to
\(\nicefrac13\).

First, suppose that \(x\in \mathsf H_{(f_n)}(\I)\). Then $B:=\{n\in\omega:|\delta_n|\geq\nicefrac13\}\in\I$. By the previous observation and the hypothesis that $\I$ is translation invariant, it follows that 
$
R_x\subseteq B\cup(B-1)\cup(B-2) \in \I.
$ 
Hence \(R_x\in\I\).

Conversely, suppose that \(R_x\in\I\). 
We need to show that \(x\in \mathsf H_{(f_n)}(\I)\). 
To this aim, fix \(\varepsilon>0\) and choose a sufficiently large integer \(M\in\omega\) such that, whenever
\(u_0,\ldots,u_{2M}\in[-\nicefrac12,\nicefrac12]\) and
\(u_{j+2}=u_{j+1}+u_j\) for all \(j<2M-1\), then \(|u_M|<\varepsilon\). 
Note that this is really possible. Indeed, every sequence satisfying the Fibonacci recurrence has the form
\(u_j=\alpha \varphi^j+\beta (-\nicefrac{1}{\varphi})^j\) for some constants $\alpha,\beta \in \mathbb{R}$. 
Since \(u_0\) and \(u_{2M}\) are
bounded by \(\nicefrac12\), it follows that \(|\alpha |\leq C\varphi^{-2M}\) and
\(|\beta|\leq C\) for some constant \(C>0\) independent of \(M\). 
Therefore \(|u_M|\leq 2C\varphi^{-M}\). 

At this point, define the set
$$
R^\sharp:=\bigcup_{k=-M}^M (R_x+k),
$$
which belongs to $\I$ since $R_x \in \I$ and $\I$ is translation invariant by hypothesis. 
It follows that, if \(n\notin R^\sharp\) and \(n\geq M\), then
\(r_{x,k}=0\) for every \(k=n-M,\ldots,n+M-2\). Hence
\[
\delta_{k+2}=\delta_{k+1}+\delta_k
\]
for every \(k=n-M,\ldots,n+M-2\). 
Thus, set \(u_j:=\delta_{n-M+j}\) for \(j=0,\ldots,2M\). Then \(u_j\in[-\tfrac12,\tfrac12]\) and, by the preceding recurrence, \(u_{j+2}=u_{j+1}+u_j\) for every \(j<2M-1\). Hence, by the defining choice of \(M\), we have $|\delta_n|=|u_M|<\varepsilon$. 
It follows that 
$$
\{n\in\omega:|\delta_n|\geq\varepsilon\}
\subseteq R^\sharp \cup\{0,1,\ldots,M-1\} \in \I. 
$$
Since \(\varepsilon>0\) was arbitrary,
we conclude that \(x\in \mathsf H_{(f_n)}(\I)\).
\end{proof}

\medskip

\begin{proof}
    [Proof of Theorem \ref{thm:propertiesfibo}]
    \ref{prop:1fibo} If $\I=\mathrm{Fin}$ then 
    \begin{equation}\label{eq:fibofin}
    \mathsf{H}_{(f_n)}(\I)=\langle \varphi\rangle,
    \end{equation}
cf. \cite[Theorem 1]{MR947645} for a different proof. 
In fact, by Theorem \ref{thm:fibonacci}, we have $\mathsf H_{(f_n)}(\mathrm{Fin})=\{x\in\mathbb T:R_x\in\mathrm{Fin}\}$. 
If \(R_x\) is finite, then $x=\sum_{n\in R_x} r_{x,n}\varphi^{-n-1}$ in $\mathbb T$ by Lemma \ref{lem:fiborepresent}. 
Since \(\varphi^{-1}=\varphi-1\), every \(\varphi^{-n-1}\) belongs to
\(\varphi \mathbb Z+\mathbb Z\). This proves that $\mathsf H_{(f_n)}(\mathrm{Fin})\subseteq \langle\varphi\rangle$. 
Conversely, using the well-known identity  $f_n\varphi=f_{n+1}+(-1)^{n+1}\varphi^{-n}$ for all $n\in \omega$ (which follows e.g. by Binet's formula), 
we get
\[
\forall m\in \omega, \quad 
\limsup_{n\to \infty}\|f_n(m\varphi)\|\leq \limsup_{n\to \infty}|m|\varphi^{-n}=0.
\]
Thus \(m\varphi\in\mathsf H_{(f_n)}(\mathrm{Fin})\) for every \(m\in\mathbb Z\), and so $\langle\varphi\rangle\subseteq\mathsf H_{(f_n)}(\mathrm{Fin})$. Therefore  \eqref{eq:fibofin} holds. 

\begin{claim}\label{claim:fibocoding}
There is \(L\in\omega\) such that, for every \(i \in \{0,1,\ldots,L-1\}\) and every 
nonempty $A\subseteq \omega$ with $\min A\ge L$ and $n\equiv i\bmod{L}$ for all $n \in A$, 
there is \(x_A\in\mathbb T\) such that
$$R_{x_A}=A.$$
\end{claim}

\begin{proof}
Pick a sufficiently large \(L\in \omega\) such that, whenever a nonempty \(A\subseteq\omega\) has distance at least \(L\) between its distinct points and \(\min A\geq L\), then
\[
\forall N \in \omega, \qquad 
\left|f_N\sum_{n\in A}\varphi^{-n-1}
-
q_N
\right|<\frac12, 
\quad \text{ where }\,\,
q_N:=\sum_{\substack{n\in A\\ n\leq N-2}}f_{N-n-1}.
\]
Note that such an \(L\) exists because, after applying Binet's formula, each summand contributes an error which decays exponentially with its distance from \(N-1\). Since the elements of \(A\) are separated by at least \(L\), successive error terms on either side of \(N-1\) are smaller by a factor at most \(\varphi^{-L}\). Hence the total error is bounded by a geometric series with common ratio \(\varphi^{-L}\), which can be made arbitrarily small by taking \(L\) large. 

Now, fix \(i\in\{0,1,\ldots,L-1\}\) and a nonempty $A\subseteq \omega$ with $\min A\ge L$ and $n\equiv i\bmod{L}$ for all $n \in A$. Thus, we define  
\begin{equation}\label{eq:definitionxA}
x_A:=\sum_{n\in A}\varphi^{-n-1}\quad\text{in }\mathbb T.
\end{equation}
By the choice of \(L\), \(q_N\) is the nearest integer to \(f_Nx_A\), hence \(q_N=p_N(x_A)\). It is immediate from the definition of \(q_N\) that 
$
q_{N+2}=q_{N+1}+q_N+\mathbf 1_A(N)
$ 
for every \(N\in\omega\). It follows that, for all $N \in \omega$, we have 
\[
r_{x_A,N}=p_{N+2}(x_A)-p_{N+1}(x_A)-p_N(x_A)=\mathbf 1_A(N)
\]
Therefore \(R_{x_A}=A\).
\end{proof}

Now, suppose that \(\mathcal I\neq\mathrm{Fin}\). Then there is an infinite set
\(A\in\mathcal I\). Fix \(L\) as in Claim \ref{claim:fibocoding}. Then it is possible to pick $i \in \{0,1,\ldots,L-1\}$ such that 
\[
B:=A\cap\{i+L(k+1):k\in\omega\}
\]
is infinite.  
Thanks to Claim \ref{claim:fibocoding} again, for every nonempty \(C\subseteq B\) there exists \(x_C\in\mathbb T\) (defined as in \eqref{eq:definitionxA}) such that $R_{x_C}=C$. 
Since \(C\subseteq B\subseteq A\in\mathcal I\), we have \(C\in\mathcal I\), and therefore $x_C\in\mathsf H_{(f_n)}(\mathcal I)$  by Theorem \ref{thm:fibonacci}. 
If \(C\neq D\), then \(R_{x_C}=C\neq D=R_{x_D}\). Since \(R_x\) is determined by \(x\), this implies \(x_C\neq x_D\). 
Therefore \(\mathsf H_{(f_n)}(\mathcal I)\) is uncountable. 

\medskip

    \ref{prop:2fibo} 
    First, suppose that \(\mathcal I\) is \(F_\sigma\). 
    Then there is an increasing sequence
\((K_m:m\in\omega)\) of hereditary compact subsets of \(2^\omega\) such that 
$
\mathcal I=\bigcup_{m}K_m.
$
For \(x\in\mathbb T\), define
\[
B_x:=\{n\in\omega:\|f_nx\|>\nicefrac14\}.
\]
\begin{claim}\label{claim:Bxfibo}
$\mathsf H_{(f_n)}(\mathcal I)=\{x\in\mathbb T:B_x\in\mathcal I\}$.
\end{claim}
\begin{proof}
If \(x\in\mathsf H_{(f_n)}(\mathcal I)\), then \(B_x\in\mathcal I\) by definition. Vice versa, if \(B_x\in\mathcal I\), then, recalling the definition of $\delta_n$ in \eqref{eq:defdeltan} and using the equality 
$
r_{x,n}=\delta_n+\delta_{n+1}-\delta_{n+2},
$ 
we get $R_x\subseteq B_x\cup(B_x-1)\cup(B_x-2)$. 
Since \(\mathcal I\) is translation invariant, then \(R_x\in\mathcal I\), and Theorem \ref{thm:fibonacci} gives \(x\in\mathsf H_{(f_n)}(\mathcal I)\). 
\end{proof}

\begin{claim}\label{claim:Bxfibo2}
$C_m:=\{x\in\mathbb T:B_x\in K_m\}$ 
is closed for each $m \in \omega$. 
\end{claim}
\begin{proof}
Fix $m \in \omega$ and $x \in \mathbb{T}$. If \(x\notin C_m\), then \(B_x\notin K_m\). Since \(K_m\) is hereditary compact, there is a finite set \(F\subseteq B_x\) such that \(F\notin K_m\). Since the condition 
$
\|f_ny\|>\nicefrac14
$ 
is open in \(y\), for each fixed \(n\), there is a neighbourhood \(U\) of \(x\) such that \(F\subseteq B_y\) for all \(y\in U\). Hence \(B_y\notin K_m\) for all \(y\in U\), and so \(U\cap C_m=\emptyset\). Therefore \(C_m\) is closed.
\end{proof}

Putting together Claim \ref{claim:Bxfibo} and Claim \ref{claim:Bxfibo2}, it follows that $\mathsf H_{(f_n)}(\mathcal I)=\bigcup_{m}C_m$ is \(F_\sigma\). 

\smallskip

Conversely, suppose that \(\mathsf H_{(f_n)}(\mathcal I)\) is an \(F_\sigma\) subgroup. Fix \(L\) as in Claim \ref{claim:fibocoding}. For every $i\in \{0,1,\ldots,L-1\}$, define 
$
e_i(k):=i+L(k+1)
$ 
for all \(k\in\omega\). Also, define the map $\Phi_i:2^\omega\to\mathbb T$ by 
$$
\forall S\subseteq \omega, \quad 
\Phi_i(S):=x_{e_i[S]}=\sum_{k\in S}\varphi^{-e_i(k)-1}.
$$
Taking into account that the above series converges uniformly on $S\subseteq \omega$, each map \(\Phi_i\) is continuous. 
Moreover, by Claim \ref{claim:fibocoding} we have $R_{\Phi_i(S)}=e_i[S]$ for every \(S\subseteq\omega\). 
Hence by Theorem \ref{thm:fibonacci} we obtain that, for all $i \in \{0,1,\ldots,L-1\}$, 
\[
\I_i:=
\{S\subseteq\omega:e_i[S]\in\mathcal I\}
=\Phi_i^{-1}\big[\mathsf H_{(f_n)}(\mathcal I)\big]
\]
is an $F_\sigma$ on $\omega$. Taking into account that $\I=\{S\subseteq \omega: e_i^{-1}[S]\in\mathcal I_i\text{ for every }i<L\}$, we conclude that $\I$ is an $F_\sigma$ ideal as well. 

\medskip

    \ref{prop:3fibo} 
    The \textsc{If} part is clear. Conversely, suppose that 
$
\mathsf H_{(f_n)}(\mathcal I)=\mathsf H_{(f_n)}(\mathcal J).
$ 
Fix \(L\) as in Claim \ref{claim:fibocoding}. For every \(i\in \{0,1,\ldots,L-1\}\), define $e_i(k)$ as in item \ref{prop:2fibo}. Now, fix $i<L$ and $S\subseteq \omega$. Thanks to Claim \ref{claim:fibocoding} and Theorem \ref{thm:fibonacci}, we get 
$e_i[S]\in\mathcal I$ if and only if $x_{e_i[S]}\in\mathsf H_{(f_n)}(\mathcal I)$, and similarly for $\J$. This implies that 
\begin{equation}\label{eq:equivalencedkjh}
e_i[S]\in\mathcal I
\quad \text{ if and only if }\quad 
e_i[S]\in\mathcal J.
\end{equation}
To conclude, pick an arbitrary set $A\subseteq \omega$. 
Taking into account \eqref{eq:equivalencedkjh}, we obtain that 
\begin{displaymath}
    \begin{split}
        A \in \I 
        &\quad\Longleftrightarrow\quad
A\cap e_i[\omega]\in\mathcal I\text{ for every }i<L\\
&\quad\Longleftrightarrow\quad
A\cap e_i[\omega]\in\mathcal J\text{ for every }i<L
\quad\Longleftrightarrow\quad 
A \in \J.
    \end{split}
\end{displaymath}
Therefore $\I=\J$. 
\end{proof}

\medskip

\begin{proof}
[Proof of Corollary \ref{cor:consequencefibo22}]
    \ref{item:1corpowersfibo} It follows by \eqref{eq:fibofin} in the proof of Theorem \ref{thm:propertiesfibo}\ref{prop:1fibo} that $\mathsf{H}_{(f_n)}(\mathrm{Fin})=\langle \varphi\rangle$. Since $\varphi$ is irrational, then $\langle \varphi\rangle$ is countably infinite. 

    \medskip

    \ref{item:2corpowersfibo} Suppose that $\mathcal{I}\neq \mathrm{Fin}$ is an $F_\sigma$ translation invariant ideal. Then $\mathsf{H}_{(f_n)}(\I)$ is an $F_\sigma$ uncountable subgroup by Theorem \ref{thm:propertiesfibo}\ref{prop:1fibo}-\ref{prop:2fibo}.

It remains to prove that \(\mathsf H_{(f_n)}(\mathcal I)\) is proper. 
Fix \(L\) as in Claim \ref{claim:fibocoding} and choose $i \in \{0,1,\ldots,L-1\}$ such that $E_i:=\{i+L(k+1):k\in\omega\}\in \I^+$. By Claim \ref{claim:fibocoding}, there exists
\(x_{E_i}\in\mathbb T\) such that \(R_{x_{E_i}}=E_i\). Hence 
\[
x_{E_i}\notin\mathsf H_{(f_n)}(\mathcal I).
\]
by Theorem \ref{thm:fibonacci}. Therefore \(\mathsf H_{(f_n)}(\mathcal I)\neq\mathbb T\).

\medskip

\ref{item:3corpowersfibo} 
It goes verbatim as in the proof of Corollary \ref{cor:consequencepowers22}\ref{item:3corpowers}, 
replacing Theorem \ref{thm:propertiespowers}\ref{prop:2powers}
with Theorem \ref{thm:propertiesfibo}\ref{prop:2fibo}. 
\end{proof}

\medskip

\begin{proof}
[Proof of Theorem \ref{thm:factorialcharacterization}]
Fix \(x\in\mathbb T\) and define
\[
\forall n\geq1,\qquad
\tau_{x,n}:=
\sum_{k=n}^{\infty}
\frac{q_{x,k}}{(n+1)(n+2)\cdots(k+1)}.
\]
Then
$
n!x=\tau_{x,n}
$ in $\mathbb T$, 
\(0\leq\tau_{x,n}<1\), and
$
\tau_{x,n}
=
(q_{x,n}+\tau_{x,n+1})/(n+1).
$ 
Hence
\[
\forall n\geq1,\qquad
\left\|
n!x-\frac{q_{x,n}}{n}
\right\|
\leq 
\left\|
n!x-\frac{q_{x,n}}{n+1}
\right\|+\left(\frac{q_{x,n}}{n}-\frac{q_{x,n}}{n+1}\right)\le 
\frac2{n+1}.
\]
Since the right hand side converges to $0$, the claim follows by \cite[Corollary 3.4  and Lemma 3.5(i)]{MR3920799}. 
\end{proof}

\medskip

\begin{proof}
    [Proof of Proposition \ref{prop:factnotFsigma}] 
    Suppose for the sake of contradiction that $H:=\mathsf H_{(n!)}(\mathcal I)$ is an \(F_\sigma\) subgroup. 
    Define the continuous map $\Phi:\mathcal P(\omega)\to\mathbb T$ by 
    $$
    \forall S \subseteq \omega, \quad 
    \Phi(S):=\sum_{n\in S\setminus \{0\}}\frac{\lfloor n/2\rfloor}{(n+1)!}.
    $$
    \begin{claim}\label{claim:reductionPhiSHfactorial}
        $\Phi^{-1}[H]=\I$. 
    \end{claim}
    \begin{proof}
    Pick $S\subseteq \omega$. Thanks to Theorem \ref{thm:factorialcharacterization}, if \(S\in\mathcal I\), then \(\Phi(S)\in H\). Conversely, if \(S\notin\mathcal I\), then 
    $\{n\in\omega: \left\|q_{\Phi(S),n}/n\right\|\geq \nicefrac{1}{3}\}$ contains \(S\) modulo a finite set, and hence it does not belong to
\(\mathcal I\). Thus \(\Phi(S)\notin H\). Therefore $S \in \I$ if and only if $\Phi(S) \in H$.
\end{proof}

Since \(H\) is \(F_\sigma\) and \(\Phi\) is continuous, it follows by Claim \ref{claim:reductionPhiSHfactorial} that \(\mathcal I\) is an \(F_\sigma\) ideal. By Mazur's theorem \cite{MR1124539}, there is a lscsm $\varphi:\mathcal P(\omega)\to[0,\infty]$ such that $\mathcal I=
\{A\subseteq\omega:\varphi(A)<\infty\}$. 
Now, we choose pairwise disjoint finite sets \((F_s:s\in\omega)\) such that $\varphi(F_s)>s$ for all $s \in \omega$ (which is possible since $\varphi$ is a lscsm and $\varphi(\omega)=\infty$).
Fix also a bijection $\langle\cdot,\cdot\rangle:\omega^2\to\omega$. Increasing the sets \(F_s\), if necessary, we may suppose without loss of generality that, whenever
\(s=\langle i,j\rangle\) and \(n\in F_s\), then \(n\geq 4(i+3)\). 
It follows by the monotonicity of $\varphi$ that if $X \in \mathrm{Fin}^+$ then $\bigcup_{s \in X}F_s\in \I^+$.

\begin{claim}\label{claim:continuousreduction}
    $\emptyset\otimes\mathrm{Fin}\le_{\mathrm{W}} H$. 
\end{claim}
\begin{proof}
For each \(A\subseteq\omega^2\), define
$$
x_A:=\sum_{n=1}^{\infty}\frac{w_{A,n}}{(n+1)!}\in\mathbb{T},
$$
where
$$
w_{A,n}:=
\begin{cases}
\,\left\lfloor\dfrac{n}{i+3}\right\rfloor
& \text{if }n\in F_{\langle i,j\rangle}\text{ and }(i,j)\in A,\\[1.2em]
\,0
& \text{otherwise}.
\end{cases}
$$
Note that, since \(n\geq4(i+3)\) on \(F_{\langle i,j\rangle}\), we have
$$
\frac{1}{2(i+3)}
\leq
\frac{w_{A,n}}{n}
\leq
\frac{1}{i+3}
$$
for all \(n\in F_{\langle i,j\rangle}\) such that \((i,j)\in A\).
Since the digits \(w_{A,n}\) are always smaller than \(n\), the expansion above
is canonical. Moreover, the map
$
A\mapsto x_A
$
is continuous, because every digit \(w_{A,n}\) depends on at most one
coordinate of \(A\). To complete the proof, it is enough to show that
\begin{equation}\label{eq:EmptyTensorFinReduction}
x_A\in H
\quad\Longleftrightarrow\quad
A\in\emptyset\otimes\mathrm{Fin}.
\end{equation}

First, suppose that \(A\in\emptyset\otimes\mathrm{Fin}\), so that \(A_i\) is
finite for every \(i\in\omega\). Fix \(m\geq3\). If
$
\|w_{A,n}/n\|\geq\nicefrac1m,
$
then \(n\in F_{\langle i,j\rangle}\) for some \(i\leq m-3\) and some
\(j\in A_i\). Hence
$$
\left\{
n\in\omega:
\left\|\frac{w_{A,n}}{n}\right\|\geq\frac1m
\right\}
\subseteq
\bigcup_{i\leq m-3}\,
\bigcup_{j\in A_i}F_{\langle i,j\rangle}.
$$
The right-hand side is a finite union of finite sets, and therefore it belongs
to \(\mathcal I\). Since \(m\geq3\) was arbitrary, Theorem
\ref{thm:factorialcharacterization} gives \(x_A\in H\).

Conversely, suppose that \(A\notin\emptyset\otimes\mathrm{Fin}\). Then there
exists \(i\in\omega\) such that \(A_i\) is infinite. Put \(m:=2(i+3)\). For
every \(j\in A_i\) and every \(n\in F_{\langle i,j\rangle}\), we have
$$
\left\|\frac{w_{A,n}}{n}\right\|
=
\frac{w_{A,n}}{n}
\geq
\frac{1}{2(i+3)}
=
\frac1m.
$$
Consequently,
$$
\bigcup_{j\in A_i}F_{\langle i,j\rangle}
\subseteq
\left\{
n\in\omega:
\left\|\frac{w_{A,n}}{n}\right\|\geq\frac1m
\right\}.
$$
Since \(A_i\) is infinite, the set
$
\{\langle i,j\rangle:j\in A_i\}
$
is infinite. Hence
$$
\bigcup_{j\in A_i}F_{\langle i,j\rangle}\in\mathcal I^+,
$$
and therefore
$$
\left\{
n\in\omega:
\left\|\frac{w_{A,n}}{n}\right\|\geq\frac1m
\right\}
\in\mathcal I^+.
$$
It follows again by Theorem \ref{thm:factorialcharacterization} that
\(x_A\notin H\). This proves \eqref{eq:EmptyTensorFinReduction}.
\end{proof}

Since $\emptyset\otimes\mathrm{Fin}$ is not an $F_\sigma$ ideal, it follows
by Claim \ref{claim:continuousreduction} that $H$ is not $F_\sigma$.
\end{proof}

\medskip

\begin{remark}\label{rmk:factorial}
    Thanks to Claim \ref{claim:reductionPhiSHfactorial}, 
    we have $\I\le_{\mathrm{W}}\mathsf{H}_{(n!)}(\I)$ for all ideals $\I$. In addition, since the map $\Phi$ does not depend on the choice of the ideal, it easily follows that if $\I,\J$ are two ideals on $\omega$, then 
    $$
    \I\subseteq \J \quad \Longleftrightarrow \quad \mathsf{H}_{(n!)}(\I)\subseteq \mathsf{H}_{(n!)}(\J)
    $$ 
    (we omit details). In particular, since there are $2^{\mathfrak{c}}$ ideals, there are $2^{\mathfrak{c}}$ subgroups of the type $\mathsf{H}_{(n!)}(\I)$. 
\end{remark}

\medskip

\begin{proof}
    [Proof of Proposition \ref{prop:noPi02}]
    Recall by Remark \ref{rmk:examplecomplete} that $\mathsf{H}_{(n!)}(\mathrm{Fin})$ is an $F_{\sigma\delta}$-complete subgroup of $\mathbb{T}$, hence 
    $$
    \mathscr{H}(\mathrm{Fin})\setminus \mathbf{\Sigma}^0_3(\mathbb{T})\neq \emptyset. 
    $$
    Now, suppose that $\I$ is a meager ideal. Then  $\mathscr{H}(\mathrm{Fin})\subseteq \mathscr{H}(\I)$ by \cite[Corollary 2.7]{FKLT26}. It follows that  $\mathscr{H}(\I)\setminus \mathbf{\Sigma}^0_3(\mathbb{T})\neq \emptyset$ as well. Hence, suppose hereafter that $\I$ is not meager. In particular, $\I$ is not analytic, cf. \cite[Figure 2]{FKL24}. We conclude by Corollary \ref{cor:pointclassessimple} that $\mathscr{H}(\I)\setminus \mathbf{\Sigma}^1_1(\mathbb{T})\neq \emptyset$. This completes the proof. 
\end{proof}

\medskip

\begin{proof}
[Proof of Corollary \ref{cor:Fsigmadeltacomplete}]
    Thanks to Proposition \ref{prop:noPi02}, there exists $H \in \mathscr{H}(\I)$ which is not $G_{\delta\sigma}$. However, we have $\mathrm{rank}(H)\le \mathrm{rank}(\I)=3$ by Corollary \ref{cor:borelrank}\ref{item:2borrank} hence $H$ is $F_{\sigma\delta}$. The completeness claim follows by \cite[Theorem 22.10 and Exercise 24.20]{MR1321597}. 
\end{proof}

\medskip

\begin{proof}
[Proof of Proposition \ref{prop:RKorder}] 
Fix an almost disjoint family $\{A_\xi:\xi<\mathfrak c\}$ of infinite subsets of $\omega$, and bijections $e_\xi:\omega\to A_\xi$ for all $\xi<\mathfrak c$. Let $\J$ be the ideal on $\omega$ generated by $\mathrm{Fin}$ together with all sets $e_\xi[B]$, where $\xi<\mathfrak c$ and $B\in\I_\xi$. 

First, observe that $\J$ is proper. Indeed, if $\omega\in\J$, then there are a finite set $F\subseteq\omega$, indices $\xi_0,\ldots,\xi_{k-1}<\mathfrak c$, and sets $B_i\in\I_{\xi_i}$ such that
$$
\omega\subseteq F\cup\bigcup_{i< k}e_{\xi_i}[B_i].
$$
Pick an index $\eta<\mathfrak c$ distinct from $\xi_0,\ldots,\xi_{k-1}$. Since the family $\{A_\xi:\xi<\mathfrak c\}$ is almost disjoint, the set
$
A_\eta\cap(F\cup\bigcup_{i< k}A_{\xi_i})
$
is finite, contradicting that $A_\eta$ is infinite.

We claim that, for every $\xi<\mathfrak c$ and every $S\subseteq A_\xi$, we have
\begin{equation}\label{eq:traceJ}
S\in\J
\quad\text{ if and only if }\quad
e_\xi^{-1}[S]\in\I_\xi.
\end{equation}
The \textsc{If} part follows by definition. Conversely, suppose that $S\in\J$. Then there are a finite set $F\subseteq\omega$, indices $\xi_0,\ldots,\xi_{k-1}<\mathfrak c$, and sets $B_i\in\I_{\xi_i}$ such that
$
S\subseteq F\cup\bigcup_{i< k}e_{\xi_i}[B_i].
$
Since $A_\xi\cap A_{\xi_i}$ is finite whenever $\xi_i\neq\xi$, it follows that
$$
e_\xi^{-1}[S]\subseteq e_\xi^{-1}[F\cap A_\xi]
\cup\bigcup_{\substack{i< k\\ \xi_i=\xi}}B_i
\cup E
$$
for some finite set $E\subseteq\omega$. Since $\I_\xi$ is an ideal containing $\mathrm{Fin}$, we conclude that $e_\xi^{-1}[S]\in\I_\xi$, which proves \eqref{eq:traceJ}.

Now, fix $\xi<\mathfrak c$ and a sequence $\bm a\in\mathbb Z^\omega$. Define $\bm b\in\mathbb Z^\omega$ by $b_{e_\xi(n)}:=a_n$ for all $n\in\omega$, and $b_n:=0$ for all $n\notin A_\xi$. Then, for every $x\in\mathbb T$ and every $\varepsilon>0$, we get
$$
\{n\in\omega:\|b_nx\|\geq\varepsilon\}
=
e_\xi[\{n\in\omega:\|a_nx\|\geq\varepsilon\}].
$$
It follows by \eqref{eq:traceJ} that $\mathcal J\text{-}\lim_n b_nx=0$ if and only if $\mathcal I_\xi\text{-}\lim_n a_nx=0$. Therefore
$
\mathsf H_{\bm a}(\I_\xi)=\mathsf H_{\bm b}(\J),
$
and hence $\mathscr H(\I_\xi)\subseteq\mathscr H(\J)$. This completes the proof. 
\end{proof}

\medskip

\begin{proof}
    [Proof of Corollary \ref{cor:analyticconalaytc}] 
    For each $\xi<\mathfrak{c}$, we obtain by \cite[Theorem 2.11]{FKLT26} that there exists an ideal $\I_\xi$ on $\omega$ such that $H_\xi \in \mathscr{H}(\I_\xi)$, 
    cf. also \cite[Theorem 2.1]{MR2227021}. The conclusion follows by Proposition \ref{prop:RKorder} and the fact that there are at most $\mathfrak{c}$ subgroups of $\mathbb{T}$ which are analytic or coanalytic. 
\end{proof}

\medskip

\begin{proof}
    [Proof of Proposition \ref{prop:equalitymathscrHFin}] 
    The \textsc{If} part goes as in the proof of the implication \ref{item:1countablygenerated} $\implies$ \ref{item:2countablygenerated} of Proposition \ref{prop:locallyquasiconvex}. 
    The \textsc{Only If} part follows by Proposition \ref{prop:locallyquasiconvex2}. 
\end{proof}

\medskip

\begin{proof}
    [Proof of Corollary \ref{cor:lastone}] 
    Since $\I$ is meager, we have $\mathscr{H}(\mathrm{Fin})\subseteq \mathscr{H}(\I)$ by \cite[Corollary 2.7]{FKLT26}. The conclusion follows by Proposition \ref{prop:equalitymathscrHFin}.
\end{proof}

\section{Open Questions}\label{sec:openquestions}

In this section, we list our open questions. 
First, we start asking whether the analogue of Theorem \ref{thm:increasingSequenceSameFin} holds for the ideal $\mathcal{Z}$.
\begin{question}\label{q:Zstrictlyincreasing}
    Pick a proper $\mathcal{Z}$-characterized subgroup $H$ of $\mathbb{T}$. Does there exist a strictly increasing sequence of positive integers $\bm{a}$ such that $H=\mathsf{H}_{\bm{a}}(\mathcal{Z})$? Is it possible to characterize the ideals $\I$ on $\omega$ with the analogue property?
\end{question}

Second, it would be interesting to know whether the Borel rank of $\I$ provides not only an upper bound on the Borel rank of each $\mathsf{H}_{\bm{a}}(\I)$, as in Corollary \ref{cor:borelrank}, but also a structural existence condition for all its smaller values. 
\begin{question}\label{q:smallervalues}
    Let $\I$ be a Borel ideal on $\omega$. Is it true that for each countable ordinal $3\le \xi\le \mathrm{rank}(\I)$ there exists $\bm{a} \in \omega^\omega$ such that 
    $
    \mathrm{rank}(\mathsf{H}_{\bm{a}}(\I))=\xi
    $?
\end{question}

Third, looking at the maximal Borel rank of $\I$-characterized subgroups, we might ask whether the analogue of Corollary \ref{cor:Fsigmadeltacomplete} holds for higher finite Borel ranks: 
\begin{question}\label{q:maximallevel}
Let $\I$ be a Borel ideal with finite rank $m:=\mathrm{rank}(\I)$. Does there exist a $\mathbf{\Pi}^0_m$-complete $\I$-characterized subgroup of $\mathbb{T}$?
\end{question}

Fourth, Theorem \ref{thm:Fsigmadelcomplete} shows that generalized density ideals satisfy a particularly strong dichotomy, 
while Theorem \ref{thm:strangecounterexample} shows that the same conclusion fails for arbitrary analytic $P$-ideals. In addition, if $\I\neq \mathrm{Fin}$ is an $F_\sigma$ translation invariant ideal then $\mathsf{H}_{(f_n)}(\I)$ is an uncountable proper $F_\sigma$ subgroup by Corollary \ref{cor:consequencefibo22}\ref{item:2corpowersfibo}. 
Hence it would be interesting to determine exactly where the boundary lies. 
\begin{question}\label{q:dichotomyideals}
    Describe the class of the analytic $P$-ideals $\I$ on $\omega$ such that every proper $\I$-characterized subgroup of $\mathbb{T}$ is either countable or $F_{\sigma\delta}$-complete.
\end{question}

Lastly, taking into account Proposition \ref{prop:locallyquasiconvex2} and Theorem \ref{thm:Fsigmadelcomplete}, we conclude with the following question.
\begin{question}\label{q:important}
     Let $H$ be a proper subgroup of $\mathbb{T}$. If $H$ is characterized then either $H$ is countable or $H$ is $F_{\sigma\delta}$-complete, locally quasi-convex Polishable, and its Polish group topology is UFSS. Does the converse hold?
\end{question}



\bibliographystyle{amsplain}
\bibliography{refs}

\end{document}